\documentclass[noams,lp]{compositio}

\usepackage{amsbsy,amssymb,amscd,amsfonts,latexsym,amstext,delarray,
 amsmath,stackrel,lineno}

\input xypic
\usepackage[pdfborder={0 0 0}]{hyperref}

\hypersetup{
    colorlinks = true,
    linkcolor = blue,
    anchorcolor = red,
    citecolor = green,
    filecolor = red,
    pagecolor = red,
    urlcolor = magenta
}

\newcommand*\patchAmsMathEnvironmentForLineno[1]{%
  \expandafter\let\csname old#1\expandafter\endcsname\csname #1\endcsname
  \expandafter\let\csname oldend#1\expandafter\endcsname\csname end#1\endcsname
  \renewenvironment{#1}%
     {\linenomath\csname old#1\endcsname}%
     {\csname oldend#1\endcsname\endlinenomath}}%
\newcommand*\patchBothAmsMathEnvironmentsForLineno[1]{%
  \patchAmsMathEnvironmentForLineno{#1}%
  \patchAmsMathEnvironmentForLineno{#1*}}%
\AtBeginDocument{%
\patchBothAmsMathEnvironmentsForLineno{equation}%
\patchBothAmsMathEnvironmentsForLineno{align}%
\patchBothAmsMathEnvironmentsForLineno{flalign}%
\patchBothAmsMathEnvironmentsForLineno{alignat}%
\patchBothAmsMathEnvironmentsForLineno{gather}%
\patchBothAmsMathEnvironmentsForLineno{multline}%
}

\newtheorem{thm}{Theorem}[section] 
\newtheorem{defn}[thm]{Definition} 
\newtheorem{prop}[thm]{Proposition}
\newtheorem{cor}[thm]{Corollary}
\newtheorem{lem}[thm]{Lemma}

\newtheorem{rem}[thm]{Remark}
\newtheorem{example}[thm]{Example}

\def\Aut{{\rm Aut}}

\def\Hom{{\rm Hom}}

\def\Ker{{\rm Ker}}

\def\sign{{\rm sign}}

\def\Spec{{\rm Spec\,}}

\def\C{{\mathbb C}}
\def\F{{\mathbb F}}

\def\N{{\mathbb N}}

\def\Q{{\mathbb Q}}
\def\R{{\mathbb R}}
\def\Z{{\mathbb Z}}

\def\cE{{\mathcal E}}

\def\cO{{\mathcal O}}

\def\cT{{\mathcal T}}

\def\trop{{\R^\flat}}
\def\troc{{\cT\C}}

\def\cfl{{\C^\flat}}

\def\qqq{\,,\ \forall}

\newcommand{\ie}{{\it i.e.\/}\ }
\newcommand{\eg}{{\it e.g.\/}\ }
\newcommand{\cf}{{\it cf.\/}\ }

\def\sin{{{\rm sin}}}
\def\cos{{{\rm cos}}}

\def\ker{{\mbox{Ker}}}

\def\Hom {{\mbox{Hom}}}

\def\ffa{\mathfrak{a}}

\def\ffm{\mathfrak{m}}

\def\rmax{\R_+^{\rm max}}
\def\te{Teichm\"uller }

\def\trop{{\R^\flat}}
\def\sign{\mathbf{S}}

\newcommand{\Cp}{\mathbb{C}_p}

\newcommand{\eps}{\varepsilon}
\newcommand{\ra}{\rightarrow}

\newcommand{\topf}{\cE}

\newcommand{\OO}{\mathcal{O}}

\newcommand{\etplus}{\OO_F}

\newcommand{\btplusarc}{{B^{b,\,+}_\infty}}
\newcommand{\jj}{J}
\newcommand{\btpluszero}{{B^{b,\,+}_{\infty,\,0}}}

\newcommand{\mik}{\mathcal{M}}
\newcommand{\mikfield}{\mathfrak{M}}

\def\wittO{W_{\cO_K}(\cO_F)}
\def\wittOovpi{B^{b,+}}
\def\inverta{B^b}
\def\rigid{B^+}

\newcommand{\ai}{{\rm Ai}}
\newcommand{\bi}{{\rm Bi}}

\def\arf{\mathbf{F}}
\def\varh{T}

\begin{document}

\title{Universal thickening  of the field of real numbers}
\author{Alain Connes}
\email{alain@connes.org}
\address{Coll\`ege de France,
3 rue d'Ulm, Paris F-75005 France\newline
I.H.E.S. and Ohio State University.}
\author{Caterina Consani}
\email{kc@math.jhu.edu}
\address{Department of Mathematics, The Johns Hopkins
University\newline Baltimore, MD 21218 USA.}
\dedication{To the memory of M.~Krasner, in recognition of his farsightedness.}
\classification{13F35, 11F85, 11S15, 11S20, 44A40}
\keywords{Witt vectors, hyperfields, rings of periods.}
\thanks{ The second author was partially supported by the NSF grant DMS 1069218
and would like to thank the Coll\`ege de France for some financial support.
}

\begin{abstract} We define the universal  thickening of the field of real numbers. This construction is performed in three steps which parallel the universal perfection, the Witt construction and a completion process. We show that the transposition of the perfection process at the real archimedean place is identical to the ``dequantization" process and yields Viro's tropical real hyperfield $\trop$. Then, we prove that the archimedean Witt construction in the context of hyperfields allows one to recover a field from a hyperfield, and we obtain the universal pro-infinitesimal thickening $\R_\infty$ of $\R$. Finally, we provide the real analogues of several algebras used in the construction of the rings of $p$-adic periods. We supply the canonical decomposition of elements in terms of \te lifts, we make the link with the Mikusinski field of operational calculus and compute the Gelfand spectrum of the archimedean counterparts of the rings of $p$-adic periods. In the second part of the paper we also discuss the complex case and its relation with the theory of oscillatory integrals in quantum physics.
\end{abstract}

\maketitle

\tableofcontents  

\section{Introduction}
\label{sec:introduction}


This paper establishes an analogue of the construction of  the rings of periods of $p$-adic Hodge theory (\cf\eg \cite{F2,F3,F4})
when a $p$-adic field is replaced by  the field $\R$ of real numbers. We show that the original ideas of M. Krasner, which were motivated by the correspondence he first unveiled between  Galois theories in unequal characteristics (\cite{Kr}), reappear unavoidably when the above analogy is developed. 
The interest in pursuing this construction is enhanced by our recent discovery of the {\em arithmetic site}  \cite{CCas} with its structure sheaf of semirings of characteristic $1$,  whose geometric points involve in a crucial manner the tropical semifield $\rmax$. The encounter of a structure of characteristic $1$ which is deeply related to the non-commutative geometric approach to the Riemann Hypothesis has motivated our search for the replacement of the $p$-adic constructions at the real archimedean place.\newline
We recall that the definition of the rings of $p$-adic periods is based on three main steps.
The first process (universal perfection) is a functorial construction which links a $p$-perfect field $L$ of characteristic zero (\eg the field $\C_p$ of $p$-adic complex numbers) to a perfect  field $F(L)$ of characteristic $p$ (\cf Appendix \ref{uniperf} for notations). This process uses, in a fundamental way, the properties of the $p$-adic topology and is based on extractions of roots of order $p$. Elements of $F(L)$ are sequences $x=(x^{(n)})_{n\geq 0}$, $x^{(n)}\in L$ which satisfy the condition: $(x^{(n+1)})^p=x^{(n)}$, for all $n\in\Z_{\ge 0}$. The product in $F(L)$ is defined by the rule $(xy)^{(n)}=x^{(n)} y^{(n)}$. The definition of the sum is less obvious:
\begin{equation}\label{defnsump}
    (x+y)^{(n)}=\lim_{m\to \infty}(x^{(n+m)}+y^{(n+m)})^{p^m}
\end{equation}
and the existence of the limit is ensured by properties of the $p$-adic topology. The map $\varphi(x)=x^p$ is an automorphism by construction.\newline
The second step is the $p$-isotypical Witt construction which defines a functorial process lifting back from characteristic $p$ to characteristic zero (\cf Appendix \ref{FW}).\newline
Finally, in the third step one defines the various rings of periods $B$, by making use of:
\begin{itemize}
  \item The integer ring $\cO_F\subset F(\C_p)$ and other natural rings obtained from it.
  \item The canonical covering homomorphism $\theta: W(\cO_F)\to \C_p$ (\cf Appendix~\ref{algebras}).
  \item Natural norms of $p$-adic type and corresponding completions.
\end{itemize}
These constructions then provide, for each ring of periods, a functor from the category of $p$-adic Galois representations to a category of  modules  whose definition no longer involves the original (absolute) Galois group but trades it for an action of the Frobenius $\varphi$, a differential operator etc. There is a very rich literature covering all these topics, starting of course with the seminal, afore mentioned papers of J-M. Fontaine; we refer to \cite{B} for a readable introductory overview.

For the field $\R$ of real numbers, Galois theory is of little help since
$\Aut(\R)$ is the trivial group and $\Aut_\R(\C)=\Z/2\Z$ is too small. However, the point that we want to emphasize in this paper is that the transposition of the above three steps is still meaningful and yields non-trivial, relevant rings endowed with a {\em canonical} one parameter group of automorphisms $\arf_\lambda$, $\lambda\in\R_+^\times$, which replaces the Frobenius $\varphi$. More precisely, the analogues of the above three steps are:
\begin{enumerate}
  \item A dequantization process, from fields to hyperfields.
  \item An extension ($W$-models) of the Witt construction that lifts back structures from hyperfields to fields.
  \item Various completion processes which yield the relevant Banach and Frechet algebras.
\end{enumerate}
  We discovered that the two apparently unrelated processes of dequantization on one side
  (\cf \cite{Lit}) and the direct transposition of the perfection process to the real archimedean place on the other, are in fact identical. The  perfection process, in the case of the local field $\R$, starts  by considering  the set $F(\R)$ made by sequences $x=(x_n)_{n\geq 0}$, $x_n\in \R$, which satisfy the condition
$x_{n+1}^\kappa=x_n$ for all $n\in\Z_{\ge 0}$. Here, $\kappa$ is a  fixed  positive  {\em odd} rational number  (\ie $|\kappa|_2=1$) such that $|\kappa|_\infty<1$. In Section~\ref{sec:1stsection} we prove that by applying the same algebraic rules as in the construction of the field $F(L)$ in the $p$-adic case, one inevitably obtains a hyperfield (in the sense of M.~Krasner)  $\trop$. The hyper-structure on $\trop$  is perfect, independent of the choice of $\kappa$ and it turns out that $\trop$ coincides with the tropical real hyperfield introduced by O.~Viro in \cite{V} as the dequantization of $\R$. The main relevant feature of the hyperfield $\trop$ is to be no longer rigid (unlike the field $\R$) and some of its properties are summarized as follows

\vspace{.05in}

\textbf{Theorem}~{\em
$(i)$~$\trop$ is a perfect hyperfield of characteristic one, \ie $x+x=x$, $\forall x\in \trop$ and  for any odd integer $n>0$, the map $\trop\ni x\mapsto x^n$ is an automorphism of $\trop$.

  $(ii)$~$\Aut(\trop)=\R_+^\times$, with a canonical one parameter group of automorphisms $\theta_\lambda$, $\lambda\in\R_+^\times$.

  $(iii)$~The map $x\to x_0$ defines a bijection of sets $\trop\stackrel{\sim}{\to} \R$; the inverse image of the interval $[-1,1]\subset\R$ is the maximal compact sub-hyperring $\cO\subset\trop$.
}

\vspace{.05in}

No mathematician will abandon with light heart the familiar algebraic framework of rings and fields for the esoteric one of hyperstructures. When Krasner introduced hyperfields (and hyperrings) motivated by the correspondence he had unveiled between the Galois theories in unequal characteristics (\cite{Kr}), the main criticism which prevailed was that all the interesting and known examples of hyperfields (and hyperrings) are obtained as quotients $K/G$ of a field (or ring) by a subgroup $G\subset K^\times$ of its multiplicative group, so why not to encode the structure  by the (classical) pair $(K,G)$
rather than by the hyperfield $K/G$. The second step (ii) in our construction exploits exactly that criticism and turns it into a construction which has the additional advantage to parallel the classical $p$-isotypical Witt construction. Given a hyperfield $H$, a $W$-{\em model} of $H$ is by definition a triple $(K,\rho,\tau)$ where
\begin{itemize}
  \item $K$ is a field
  \item $\rho:K\to H$ is a homomorphism of hyperfields
  \item $\tau: H\to K$ is a multiplicative section of $\rho$.
\end{itemize}
The notion of morphism of $W$-models is straightforward to define. A $W$-model of $H$ is said to be {\em universal}  if it is an initial object in the category of $W$-models of $H$. When such universal model exists it is unique up to canonical isomorphism and we denote it by $W(H)$.
In Section~\ref{sec:2ndsection} we show that the universal $W$-model of $\trop$ exists and it coincides with the triple which was constructed in \cite{CC1,C}, by working with the tropical semi-field $\rmax$ of characteristic one and implementing some concrete formulas, involving entropy, which extend the \te formula for sums of \te lifts to the case of characteristic one. We let $W=\text{Frac}({\Q[\R_+^\times]})$ be the  field of fractions of the group ring of the multiplicative group $\R_+^\times$, $\tau_W: \R_+^\times\to \Q[\R_+^\times]\subset W$ be the canonical group homomorphism and  $\rho_W: W\to \trop\sim\R$ be the map defined by
$$
\rho_W(\sum_i\alpha_i\tau_W(x_i)/\sum_j\beta_j\tau_W(y_j))=\text{sign}(\frac{\alpha_0}{\beta_0}) \frac{x_0}{y_0}
$$
where $x_0=\text{sup}\{x_i\}$, $y_0=\text{sup}\{y_j\}$ and $\alpha_i,\beta_j\in\Q$.

\vspace{.05in}

\textbf{Theorem}~{\em  The triple $(W=\text{Frac}({\Q[\R_+^\times]}),\rho_W, \tau_W)$ is the universal $W$-model for $H=\trop$. The homomorphism $\rho_W$ induces an isomorphism of hyperfields  $W/G\stackrel{\sim}{\to} \trop$, where $G = \ker(\rho_W: W^\times \to \trop^\times)$.
}

\vspace{.05in}

 In the $p$-isotypical Witt construction $R\mapsto W(R)$, the respective roles of
 $(H,K,\rho,\tau)$ correspond to the initial perfect ring $R$, the $p$-isotypical Witt ring $W(R)$, the residue homomorphism $\rho: W(R)\to R$ and the \te lift $\tau:R\to W(R)$. It is important to underline here the fact that while homomorphisms of fields are necessarily injective this restriction no longer applies to hyperfields. This is the reason why one can work directly with fields in the definition of $W$-models. The subring $W_{\Z}(H)\subset W(H)$ generated by the range of the section $\tau$ provides then a ring theoretic structure, at the real archimedean place, and it generates the field $W(H)$. Moreover, the definition of the subring $W_{\Z}(R)\subset W(H)$ is meaningful for any sub-object $R\subset H$. This construction applies in particular to the maximal compact sub-hyperring $\cO\subset \trop$ and it provides the starting structure from where one develops the construction of the real archimedean analogues of the various rings used in $p$-adic Hodge theory. Finally, the functoriality of the set-up of the universal $W$-models yields, for $H=\trop$, a {\em canonical} one parameter group of automorphisms
 \begin{equation}\label{autf}
  \arf_\lambda=W(\theta_\lambda)\in \Aut(W(\trop)), \ \ \lambda\in\R_+^\times
 \end{equation}
 which are compatible with (\ie preserve globally)  the various subrings defined above. The analogue of the covering map $\theta$ is defined, likewise in the $p$-adic case, as the unique ring homomorphism
 \begin{equation}\label{thetamap}
    \theta: W_{\Q}(\trop)\to \R, \qquad \theta(\tau(x))=x_0\qqq x=(x_n)_{n\geq 0}\in F(\R)=\trop.
 \end{equation}
  Using the map $\theta$ we define the universal formal pro-infinitesimal thickening of $\R$ as the $\Ker(\theta)$-adic completion of $W_{\Q}(\trop)$, \ie $\R_\infty=\varprojlim_n W_{\Q}(\trop)/\Ker(\theta)^n$.
 In Section~\ref{entrop} (\cf Theorem \ref{thmkerth}), we show
 that $\R_\infty$ is more substantial than the ring $\R[[T]]$ of formal power series with real coefficients.  For each non-trivial group homomorphism $\ell:\R_+^\times\to \R$, we define a surjective ring homomorphism $\R_\infty\twoheadrightarrow\R[[T]]$. In fact we find that the real vector space $\Omega_\R=\Ker(\theta)/\Ker(\theta)^2$  is infinite dimensional and it is inclusive of the $\R$-linearly independent set of natural periods $\pi_p=[p]-p$, indexed by prime numbers. Theorem \ref{thmentropysym} gives the presentation of $\Omega_\R$ by generators $\varepsilon(x)$, $x\in \R$, and relations ($(A)$, $(B)$, $(C)$), which coincide with the defining relations of the argument of the $1.\frac 12$ logarithm (\cf \cite{Ko,C}) intrinsically related to the entropy function.

 \vspace{.05in}

\textbf{Theorem}~{\em  The space  $\Ker(\theta)/\Ker(\theta)^2$ is the infinite dimensional $\R$-vector space $\Omega_\R$ generated by the symbols $\varepsilon(x)$, $x\in \R$, with relations
\begin{eqnarray}
  (A) &:& \varepsilon(1-x)=\varepsilon(x)\nonumber \\
  (B) &:& \varepsilon(x+y)=\varepsilon(y)+(1-y)\,\varepsilon(\frac{x}{1-y})+y \, \varepsilon(-\frac x y)\qqq y\notin\{0,1\}\nonumber  \\
  (C) &:& x \,\varepsilon(1/x)=-\varepsilon(x)\qqq x\neq 0. \nonumber
\end{eqnarray}
}

 The above  real archimedean analogue of the $p$-isotypical Witt construction is purely algebraic and the archimedean analogue of the $p$-adic topology plays a dominant role in the third step (iii) of our construction (\cf
 Section~\ref{sec:Frechet}). This process yields $\R$-vector spaces and, in direct analogy with the theory of $p$-adic rings of periods, the definition of several  Banach and Frechet algebras obtained as completions using the direct analogues of the $\|.\|_\rho$ norms of the $p$-adic theory  (\cf Appendix~\ref{algebras}).  In Theorem \ref{mainwitt} we show that the real archimedean analogue $\btplusarc$ of the ring $\wittOovpi$ of $p$-adic Hodge theory (\cf Appendix~\ref{algebras}) is the Banach algebra of convolution of finite real Borel measures on $[0,\infty)$. In Section~\ref{sec:4thsection} we investigate the ideals and the Gelfand spectrum of the  Frechet algebras  obtained from $\btplusarc$ by completion with respect to the archimedean analogue of the norms $\|f\|_\rho$ used in $p$-adic Hodge theory (\cf Appendix~\ref{algebras}). In Theorem \ref{kertheta3} we show that the Gelfand spectrum $\Spec(B^+_{\C,0})$ of the Frechet algebra $B^+_{\C,0}$  is the one point compactification $Y=\C^+\cup\{\infty\}$  of the open half-plane $\C^+=\{z\in \C\mid \Re(z)>0\}$. It follows that the above algebras  can be faithfully represented as algebras of holomorphic functions of the complex variable $z\in Y$, and moreover
  \begin{itemize}
    \item The \te lift $[x]$ of an element $x\in [-1,1]$ is given by the function $z\mapsto {\rm sign}(x)|x|^z$.
    \item For $\rho>0$ the analogue of the $\|.\|_\rho$ norm is given, with $\alpha=-\frac{1}{\log\rho}$, by
$$
    \|f\|_\rho=\int_{0}^\infty e^{-\xi \alpha} \vert d\phi(\xi)\vert, \  \  \forall f(z)=
    \int_{0}^\infty e^{-\xi z} d\phi(\xi)
$$

where the function $\phi(\xi)$ is of bounded variation.
    \item The one parameter group $\arf_\lambda$ acts on $\C^+$ by scaling  $z\to \lambda z$ and it fixes $\infty\in Y$.
  \end{itemize}
  In Appendix \ref{tropical} we explain the relation of the point of view taken in this paper and our earlier archimedean Witt construction in the framework of perfect semi-rings of characteristic one. It is simply given by the change of variables $z=\frac 1 T$ as explained in \eqref{mapback}. Our analogy with the $p$-adic case is based on the following canonical decomposition of the elements of $\btplusarc$
(\cf Section~\ref{sec:Frechet}, Theorem~\ref{thmdec}; the symbol $\smile$ denotes the hyperaddition in $\trop$, \cf\eqref{taurgood})

\vspace{.05in}

\textbf{Theorem}~{\em Let $f\in\btplusarc$. Then, there exists a real number $s_0>-\infty$ and a real measurable function $s\geq s_0\mapsto f_s\in [-1,1]\setminus \{0\}$, unique except on a set of Lebesgue measure zero, such that $f_s\smile f_t=f_s$ for $s\leq t$ and so that} $f=\int_{s_0}^\infty [f_s]e^{-s}ds$.

\vspace{.05in}

We use this canonical decomposition as a substitute of the  $p$-adic decomposition of every element $x\in\wittOovpi$  in the form  $x=\sum_{n\gg-\infty}[x_n]\pi^n$, with $x_n\in\cO_F,\forall n$ (we refer again to Appendix~\ref{algebras} for notations).
The relation between the asymptotic expansion of $f(z)$ for $z\to \infty$ in powers of $T=\frac 1z$ and the Taylor expansion at $\xi=0$ of the function $\phi(\xi)$ of 
 the formula $ f(z)=
    \int_{0}^\infty e^{-\xi z} d\phi(\xi)
 $
 is given by the Borel transform. In the simplest example  $\phi(\xi)=\frac{\xi}{1+\xi}$ which corresponds to the Euler divergent series
$$
f(z)=f(1/T)\sim \sum (-1)^n n! T^n=\sum (-1)^n n! z^{-n}
$$
the canonical decomposition is given by the fast convergent expression 
$f=\int_{0}^\infty [f_s]e^{-s}ds$ where $f_s=e^{1-e^s}\in [-1,1]\setminus \{0\}$ for all $s\geq 0$.

By exploiting Titchmarsh's theorem we then show that, in general,  the leading term $f_{s_0}$ in the expansion has a multiplicative behavior in  analogy with the $p$-adic counterpart. This part is directly related to the construction of the Mikusinski field: in Proposition \ref{propextend} we provide the precise relation by  constructing an embedding of the algebra $B^+_\infty$ in the Mikusinski field $\mikfield$.

 In Section \ref{complex} we start the development of the complex case, namely when the local field $\R$ is replaced by the field $\C$ of complex numbers. We describe an intriguing  link between the process of dequantization of $\C$ and the oscillatory integrals which appear everywhere in physics problems. We illustrate this connection by treating in details the case of the Airy function and by showing how the asymptotic expansion of this function (already obtained by Stokes in the nineteenth century) involves an hypersum in the hyperfield quotient of a  field of complex valued functions by a subgroup of its multiplicative group. More in general, in the context of gauge theories in physics, the presence of several critical points is unavoidable and for this reason we expect that the formalism deployed by the theory of hyper-structures (hyperrings and hyperfields) might shed some light on the evaluation of Feyman integrals in that context.\newline
Motivated by the Wick rotation in quantum physics, which allows one to trade an oscillatory integral for an integral of real exponentials, we study a simple ``toy model" $\cfl$    of the dequantization of the field of complex numbers by paralleling the various steps explained before for the real case.
In particular, we prove that $\cfl$ is the natural  perfection of the hyperfield $\troc$  introduced by Viro. The infinite dimensional, complex vector space $\Omega_\C=\Ker(\theta)/\Ker(\theta)^2$ naturally associated to the universal, formal pro-infinitesimal thickening $\C_\infty$ of $\C$ contains two $\C$-linearly independent types of periods. The first set is the natural complexification of the set of real periods $\pi_p$, while the second period $\varepsilon$ is purely complex and it corresponds to $2i\pi$.\newline
Appendix \ref{tabrec} reports a table which describes the archimedean structures that we have defined and discussed in this paper and their $p$-adic counterparts. \newline
In Appendix \ref{uniperf}  we provide a short overview of the well-known construction of universal perfection in number theory.\newline
In Appendix \ref{FW} we develop a succinct presentation of the isotypical Witt construction
 $R\mapsto W(R)$ as a prelude to the theory of $W$-models. The objects of the basic category are triples $(A,\rho,\tau)$. The algebraic geometric meaning of $\rho$ is clear ($\rho:A\to R$ is a ring homomorphism) while the algebraic significance of $\tau$ (a multiplicative section of $\rho$) only becomes conceptual by using the $\F_1$-formalism of monoids.\newline
Finally, in Appendix \ref{algebras} we shortly review  some relevant constructions in $p$-adic Hodge theory which lead to the definition of the rings of $p$-adic periods.

\section{Perfection in characteristic one and dequantization}
\label{sec:1stsection}

In this section we prove that the functor defined by Fontaine (\cite{F2} \S2.1)  which associates to any $p$-perfect field $L$ a perfect field $F(L)$ of characteristic $p$ has an analogue at the real archimedean place. We find that starting with the field $\R$ of real numbers and taking the limit of the field laws yields unavoidably a hyperfield structure (\cf\cite{Kr,CC2} \S2). This construction shows on one side that hyperfields appear naturally as limit of fields and it also provides on the other side an ideal candidate, namely the tropical real hyperfield $\trop$ introduced in \cite{V} (\S7.2), as a replacement at the real archimedean place, of Fontaine's universal perfection structure.
We refer to Appendix~\ref{uniperf} for a short overview of Fontaine's original construction.

Given a  $p$-perfect field $L$, one defines a perfect field $F(L)$ of characteristic $p$
\begin{equation}\label{FdirectL}
F=F(L)=\{x = (x^{(n)})_{n\ge 0}|x^{(n)}\in L, (x^{(n+1)})^p = x^{(n)}\}
\end{equation}
 with the two operations ($x,y\in F$)
 \begin{equation}\label{Fdirect1L}
 (x+y)_{n} = \lim_{m\to\infty}(x^{(n+m)}+y^{(n+m)})^{p^m},\quad (xy)^{(n)} = x^{(n)}y^{(n)}.
 \end{equation}
 Formulas \eqref{FdirectL} and \eqref{Fdirect1L} are sufficiently simple to lend themselves to an immediate generalization.

 Let us start with a topological field $\topf$ and a rational number $\kappa$ and let consider the following set
\begin{equation}\label{Fdirectbis}
F=F(\topf)=\{x = (x_{n})_{n\ge 0}|x_{n}\in\topf, (x_{n+1})^\kappa = x_{n}\}
\end{equation}
 with the two operations ($x,y\in F$)
 \begin{equation}\label{Fdirect1bis}
 (x+y)_{n} = \lim_{m\to\infty}(x_{n+m}+y_{n+m})^{\kappa^m},\quad (xy)_{n} = x_{n}y_{n}.
 \end{equation}
If $\topf$ is a $p$-perfect field one has $\kappa=p$ and thus the $p$-adic (normalized) absolute value yields $|\kappa|_p=\frac 1p <1$. When
 $\topf=\R$ one chooses $\kappa$ such that the usual archimedean absolute value yields $|\kappa|_\infty<1$. We assume that the $2$-adic valuation of $\kappa$ is zero (\ie that the numerator and the denominator of $\kappa$ are odd), so that the operation $x\mapsto x^\kappa$ is well-defined on $\R$.
 The following theorem implements the point of view of \cite{V} to establish a precise link between the process of ``dequantization" in idempotent analysis and the universal perfection construction.
 \begin{thm}\label{limitcase}
 $(1)$~The map $F\ni x\to x_{0}\in \R$ defines a bijection of sets and preserves the multiplicative structures. \vspace{.05in}

  $(2)$~The addition defined by \eqref{Fdirect1bis} is well defined but not associative. \vspace{.05in}

   $(3)$~The addition given by the limit of the graphs in \eqref{Fdirect1bis} is multivalued, associative and defines a hyperfield structure on $F$ which coincides with the real tropical hyperfield $\trop$ of \cite{V}. \vspace{.05in}
 \end{thm}

\begin{figure}
\centering
\includegraphics[width=3.5in]{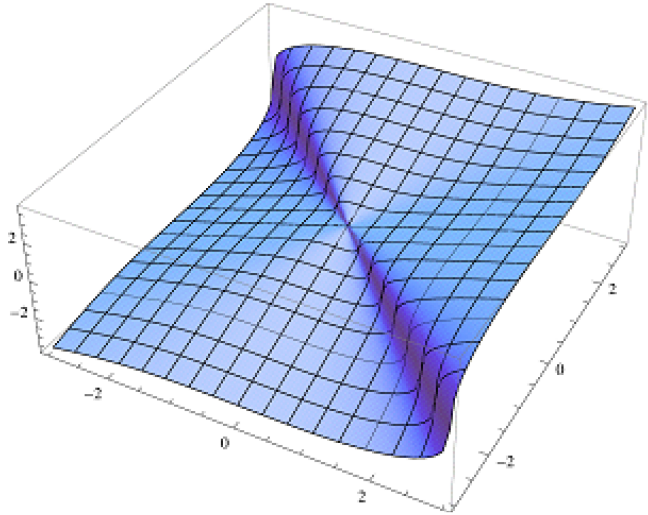}
\caption{Graph of the  addition in $\R$ after conjugation by $x\mapsto x^3$,
\ie of $(x^3+y^3)^{\frac 13}$.}\label{3addition}
\end{figure}

\begin{figure}
\centering
\includegraphics[width=3.5in]{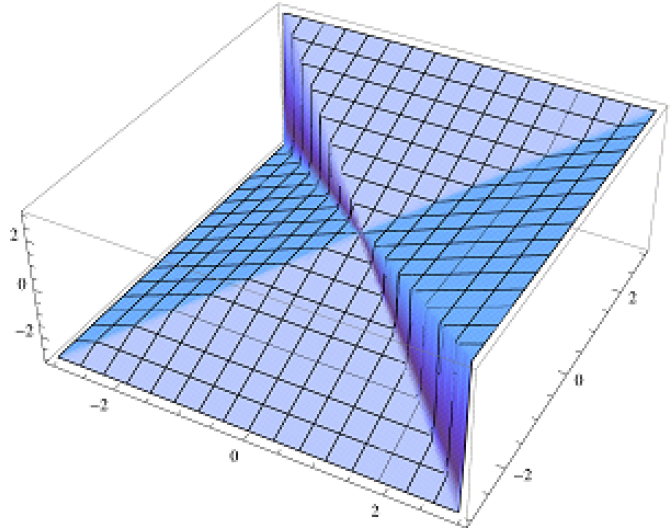}
\caption{Graph of the  addition in $\R$ after conjugation by $x\mapsto x^{3^n}$ for $n$ large. It converges to the graph of a function which is multivalued on the line $y=-x$.}\label{limitaddition}
\end{figure}

\proof $(1)$~By construction the map $x\mapsto x^\kappa$ is a bijection of $\R$, thus the map $F\ni x\mapsto x^{(0)}\in \R$ is a bijection of sets. It also preserves the multiplicative structure, due to the definition of the multiplication on $F$ as in the second formula in \eqref{Fdirect1bis}.\newline
$(2)$~To avoid confusion with the ordinary addition, we denote the addition in $F$, as in the first formula in \eqref{Fdirect1bis} and expressed in terms of $x_0\in \R$, by $x +' y$. More explicitly, it  is given by the formula
\begin{equation*}
 x +' y=  \lim_{m\to\infty}(x^{\kappa^{-m}}+y^{\kappa^{-m}})^{\kappa^m}
\end{equation*}
and is easy to compute. In fact, it is given by
\begin{equation}\label{taur}
    x +' y=\left\{
                 \begin{array}{ll}
                   x, & \hbox{if $|x|>|y|$ or $x=y$;} \\
                   y, & \hbox{if $|x|<|y|$ or $x=y$;} \\
                   $0$, & \hbox{if $y=-x$.}
                 \end{array}
               \right.
\end{equation}
In particular one finds $x+'x=x$, $\forall x\in F$. The associative law cannot hold since
for any $y\in F$ with $|y|<|x|$ one has
\begin{equation*}
    (y +' x)+' -x=x +' -x=0\,, \ \ y  +' (x +' -x)=y +' 0 =y\,.
\end{equation*}

$(3)$~For each non negative integer $m$, the graph $G_m$ of the addition  conjugated by the map $x\mapsto x^{\kappa^{-m}}$ is connected (\cf Figure~\ref{3addition}). When $m\to \infty$ these graphs converge, as closed subsets of $\R\times \R\times \R$ (as Figure~\ref{limitaddition} shows) to the graph $G$ of the addition $\smile$ on the  hyperfield $\trop$.  The obtained hyperaddition of real numbers is the following
\begin{equation}\label{taurgood}
    x \smile y=\left\{
                 \begin{array}{ll}
                  x, & \hbox{if $|x|>|y|$ or $x=y$;} \\
                   y, & \hbox{if $|x|<|y|$ or $x=y$;} \\
                    $[-x,x]$, & \hbox{if $y=-x$.}
                 \end{array}
               \right.
\end{equation}
One can see in Figure \ref{limitaddition} how the limit of the graphs of the conjugates of addition becomes multivalued on the anti-diagonal $y=-x$ and fills up the interval $[-x,x]$. One checks directly that with this hyperaddition $\trop$
is a hyperfield.\endproof

Notice that replacing the sum \eqref{taur} by the multivalued one \eqref{taurgood} is the {\em only way} of making the latter one associative without altering the first two lines of \eqref{taur}. Indeed, the fact that $0\in x+(-x)$ implies that for any $y$ with $|y|<|x|$ one has
\begin{equation*}
    y\in y+(x+(-x))=(y+x)+(-x)=x+(-x).
\end{equation*}

\begin{rem}\label{assochyper}{\rm
The second statement of Theorem~\ref{limitcase} shows that there is no ``formal" proof of associativity when addition is defined by \eqref{Fdirect1bis} and assuming that the limit exists. The third statement of the theorem implies that hyperfields naturally arise when one considers limits of field structures on the same topological space, since as the proof of the statement (3) shows, the limit of univalent maps giving addition may well fail to be univalent.

The abstract reason behind the associativity of the hyperlaw given by the limit $G$ of the graphs $G_n$ of the conjugate $+_n$ of the addition in $\R$ is that for any convergent sequence $z_n\to z$, $z\in G(x,y)$ there exist convergent sequences
$x_n\to x$ and $y_n\to y$ such that $z_n=x_n+_n y_n$. This is easy to see if $|y|<|x|$ or $y=x$ since (with $m$ an odd integer)
\begin{equation*}
    \partial_x (x^m+y^m)^{1/m}=\left(1+(y/x)^m\right)^{-1+\frac 1m}.
\end{equation*}
 If $y=-x$ the result also holds since the range of the real map $\epsilon\mapsto ((1+\epsilon)^m-1)^{1/m}$, for $|\epsilon|\leq \frac 1m$, is connected and fills up the interval $(-1,1)$ when  $m\to \infty$ ($m$ odd).
}\end{rem}
\begin{defn}\label{defn1} $(1)$~A hyperfield $H$ is of characteristic one if $x+x=x$, $\forall x\in H$.

$(2)$~A hyperfield $H$ of characteristic one is perfect if and only if for any odd integer $n>0$, the map $H\ni x\mapsto x^n$ is an automorphism of $H$.
\end{defn}
\begin{prop} \label{propviro}
$(i)$~The real tropical hyperfield $\trop$ is perfect and of characteristic one.

$(ii)$~The map $\lambda\mapsto \theta_\lambda$, $\theta_\lambda(x)={\rm sign}(x)|x|^\lambda$, $\forall x\in \R$, defines a group isomorphism
 $\theta: \R_+^\times\stackrel{\sim}{\to} \Aut(\trop)$. If $\lambda=\frac ab\in \Q_+^\times$ is odd (\ie both $a$ and $b$ are odd integers) one has $\theta_\lambda(x)=x^\lambda$, $\forall x\in \R$.

$(iii)$~The compact subset $[-1,1]\subset \trop$ is the maximal compact sub-hyperring
$\cO$ of $\trop$.

$(iv)$~The hyperfield $\trop$ is complete for the distance given by $d(x,y)=|x_0-y_0|$, for $x,y\in F=\trop$.
\end{prop}
\proof $(i)$ follows from the equality $1+1=1$ which holds in $\trop$. The perfection follows from $(ii)$.\newline
 $(ii)$~The maps $\theta_\lambda$ are automorphisms for the multiplicative structure. They also preserve the hyperaddition $\smile$ on $\trop$. By construction, they agree with $x\mapsto x^\lambda$ when $\lambda\in\Q_+^\times$ is odd. Let $\alpha\in \Aut(\trop)$. Since $\alpha$ is an automorphism of the multiplicative group $\R^\times=\R_+^\times\times \{\pm 1\}$, one has $\alpha(-1)=-1$, and $\alpha$ preserves globally $\R_+^\times$. The compatibility with the hyperaddition shows that $\alpha$ defines an increasing group automorphism of $\R_+^\times$, thus it coincides with $\theta_\lambda$.\newline
$(iii)$~The compact subset $[-1,1]\subset \trop$ is stable under multiplication and hyperaddition. For any element  $x\in\trop$ with  $x\notin [-1,1]$, the integer powers $x^n$ form an unbounded subset of $\trop$, the maximality property then follows.\newline
$(iv)$~By Theorem~\ref{limitcase} (1), the map $x\to x_0\in \R$ defines a bijection of sets which is an isometry for the distance $d(x,y)$. The conclusion follows.
\endproof

We refer to Appendix~\ref{uniperf} (Proposition~A.1) for the $p$-adic counterpart of the above.

\section{The algebraic Witt construction for hyperfields}
\label{sec:2ndsection}

 In this paper we use the following formulation of the classical $p$-isotypical Witt construction which associates to a perfect ring $R$ of characteristic $p$ the strict $p$-ring $W(R)$ of Witt vectors. We denote by $\rho_R: W(R)\to R$ the canonical homomorphism and $\tau_R: R\to W(R)$ the multiplicative section given by the \te lift.
 The following proposition is an immediate corollary of Theorem 1.2.1 of \cite{F3}.
\begin{prop} \label{propwittiso} Let $p$ be a prime number and $R$
a perfect ring of characteristic $p$. The triple $(W(R),\rho_R,\tau_R)$ is the universal object among triples $(A,\rho,\tau)$ where $A$ is a (commutative) ring, $\rho:A\to R$ is a ring homomorphism with multiplicative section
$\tau:R\to A$ and the following condition holds
\begin{equation}\label{projcond}
    A=\varprojlim_n A/\Ker(\rho)^n.
\end{equation}
\end{prop}
We refer to Appendix~\ref{FW} for an elaboration on the nuance, due to the presence of the multiplicative lift $\tau$ in the currently used formulation, with respect to the classical notion of universal $p$-adic thickening.
Next, we proceed in a similar manner with hyperfields, by suitably transposing the above set-up
\begin{defn}\label{defnwittmo} Let $H$ be a hyperfield.
 A Witt-model ($W$-model) of $H$  is a triple $(K,\rho,\tau)$, where $K$ is a field, $\rho: K\to H$ is a homomorphism of hyperfields,  and $\tau$ is a multiplicative section of $\rho$.
\end{defn}
 A morphism $(K_1,\rho_1,\tau_1)\to (K_2,\rho_2,\tau_2)$ of $W$-models of $H$ is a field homomorphism $\alpha: K_1\to K_2$ such that the following equations hold
\begin{equation}\label{19}
\tau_2=\alpha\circ\tau_1,\quad \rho_1=\rho_2\circ\alpha.
\end{equation}
\begin{defn}\label{defnwittmo&} A $W$-model for a hyperfield $H$ is universal if there exists a unique morphism from this model  to any other $W$-model of $H$.
\end{defn}
If a universal $W$-model exists then it is unique up-to unique isomorphism and in that case we denote it by $(W(H),\rho_H,\tau_H)$.

Next, we study the $W$-models for the real tropical hyperfield $H = \trop$.
 As a first step we construct a particular $W$-model for $H$ and then we shall prove that in fact it is the universal one.

Let $W=\text{Frac}({\Q[\R_+^\times]})$ be the field of fractions of the rational group ring $\Q[\R_+^\times]$
of the multiplicative group $\R_+^\times$. We let $\tau_W: \R_+^\times\to W$ be the canonical group homomorphism $\tau_W(x)(=[x])$ and define the map $\rho_W: W\to \trop\sim\R$ by
\begin{equation}\label{7}
\rho_W(\sum_i\alpha_i\tau_W(x_i)/\sum_j\beta_j\tau_W(y_j))=\text{sign}(\frac{\alpha_0}{\beta_0}) \frac{x_0}{y_0}
\end{equation}
where $x_0=\text{sup}\{x_i\}$, $y_0=\text{sup}\{y_j\}$ and $\alpha_i,\beta_j\in\Q$. We extend $\tau_W$ to a multiplicative section $\tau_W: \trop\to W$ of $\rho_W$ by setting $\tau_W(0)=0$ and $\tau_W(-x)=-\tau_W(x)$. It is straightforward to verify that $(W,\rho_W,\tau_W)$ is a $W$-model of $\trop$ (the details are provided in the next proof). We claim that it is also the universal one, moreover it describes the full structure of the hyperfield $\trop$ as the quotient of a field by a subgroup of its multiplicative group.

\begin{thm}\label{thmuni} The triple $(W=\text{Frac}({\Q[\R_+^\times]}),\rho_W, \tau_W)$ is the universal $W$-model for $H=\trop$. The homomorphism $\rho_W$ induces an isomorphism of hyperfields  $W/G\stackrel{\sim}{\to} \trop$, where $G = \ker(\rho_W: W^\times \to \trop^\times)$.
\end{thm}
\proof First we show that the triple $(W=\text{Frac}({\Q[\R_+^\times]}),\rho_W, \tau_W)$ is a $W$-model and it also fulfills the second property. The map $\tau_W: \trop\to W$, $x\to\tau_W(x)=[x]$ is multiplicative by construction and it is immediate to check that $\rho_W\circ \tau_W=id$, thus $\tau_W=[~]$ is a multiplicative section of the map $\rho_W$ defined by \eqref{7}. To understand $\rho_W$  it is useful to consider the field homomorphism $\Phi: W\to \mathcal M(\C)$,  where  $\mathcal M(\C)$ is the field of meromorphic functions on $\C$, defined by the formula
\begin{equation}\label{8}
\Phi(\sum_i\alpha_i\tau_W(x_i)/\sum_j\beta_j\tau_W(y_j))(z) = \frac{\sum_i\alpha_ix_i^z}{\sum_j\beta_j y_j^z}.
\end{equation}
Then, keeping in mind the notation of \eqref{7}, one deduces the following  interpretation of $\rho_W$
\begin{equation}\label{9}
\Phi(X)(z) \sim (\frac{\alpha_0}{\beta_0}) (\frac{x_0}{y_0})^z\quad\text{when}~z\to+\infty
\end{equation}
(since for $x_j<x_0, y_j<y_0$ one has $x_j^z\ll x_0^z$,
$y_j^z\ll y_0^z$ when $z\to+\infty$).
One can thus state that
\begin{equation}\label{10}
\exists a\in\R_+, \ \Phi(X)(2n+1)\sim a\rho_W(X)^{2n+1}\quad\text{when}~n\to+\infty.
\end{equation}
This means that one can define $\rho_W: W \to \trop$ by the formula
\begin{equation}\label{11}
\rho_W(X) =\lim_{n\to+\infty}(\Phi(X)(2n+1))^{1/(2n+1)}.
\end{equation}
 Notice that \eqref{11} is well defined because the odd roots are uniquely defined in $\R$. The map $\rho_W$ is clearly multiplicative, next we show that it induces an isomorphism of hyperfields  $W/G\stackrel{\sim}{\to} \trop$, where the subgroup $G$ is the kernel at the multiplicative level, \ie $G = \ker(\rho_W: W^\times \to \trop^\times)$. By definition the underlying set of  $G$ is made by the ratios
$(\sum_i\alpha_i\tau_W(x_i))/(\sum_j\beta_j\tau_W(y_j))\in W$, such that: $x_0=y_0$ and sign($\alpha_0$)$=$sign($\beta_0$). What remains to show is that the hyper-addition in $\trop$ coincides with the quotient addition rule $x+_G y$ on  $W/G=\R$. By definition one has
\begin{equation}\label{12}
x+_G y = \{\rho_W(X+Y)|~\rho_W(X)=x, \rho_W(Y)=y\}.
\end{equation}
We need to consider three cases:
\begin{enumerate}
\item[a)] Assume $|x|<|y|$. Then for some $a,b\in\R_+$ one has
$
\Phi(X)(2n+1)\sim a x^{2n+1}$ and $\Phi(Y)(2n+1)\sim b  y^{2n+1}$
 and thus it follows that  $\Phi(X+Y)(2n+1)\sim b y^{2n+1}$. Then one gets  $\rho_W(X+Y)=y$.
 \item[b)] Assume $x=y$. Then  for some $a,b\in\R_+$ one has
 $
 \Phi(X)(2n+1)\sim a x^{2n+1}$, $\Phi(Y)(2n+1)\sim b x^{2n+1}$ and $\Phi(X+Y)(2n+1)\sim (a+b)x^{2n+1}$,
 thus one gets  $\rho_W(X+Y)=x$, since $a+b>0$.
 \item[c)] Assume $y=-x$. In this case, we have with $a,b\in\R_+$:
 $
 \Phi(X)(2n+1)\sim a x^{2n+1}$, $\Phi(Y)(2n+1)\sim b y^{2n+1}=-bx^{2n+1}$. In this case we can only conclude that
 $
 |\Phi(X+Y)(2n+1)|\lesssim c|x|^{2n+1}
 $
 which gives $|\rho_W(X+Y)|\le|x|$. Moreover,  by choosing $X = a\tau_W(x)$, $Y = b\tau_W(y)$ for suitable real $a,b$, we conclude that $\{x,y\}\subset x+_G y$. In fact by taking  $z$, $|z|<|x|$, $X = \tau_W(x)+\tau_W(z)$, $Y=\tau_W(y)=-\tau_W(x)$, one gets $\rho_W(X+Y)=z$ and hence $x+_G y$ is the whole interval between $x$ and $y$.
 \end{enumerate}
 This shows that the quotient addition rule on $W/G$ coincides with \eqref{taurgood}.

Finally, we show that the triple $(W,\rho_W, \tau_W)$ is the universal $W$-model for $H=\trop$.

Let $(K,\rho,\tau)$ be a $W$-model for $H=\trop$.
First we prove that the field $K$ is of characteristic zero.
Indeed,  one has $\rho(1)=1$ and thus for any positive integer $p$ one derives
\[\rho(p)\in \underbrace{\rho(1)+\cdots+\rho(1)}_{p-times}=\{1\},\] so that $p\neq 0$, $\forall p$. Next,
we note that $\tau(-1)=-1\in K$. Indeed, since $\tau$ is multiplicative one has $\tau(-1)^2=\tau(1)=1$ and $\tau(-1)\neq 1$ since $(\rho\circ\tau)(-1)=-1\neq\rho(1)=1$. By multiplicativity of $\tau$ we thus get
\begin{equation}\label{15}
\tau(-x)=-\tau(x)\quad \forall x\in H=\trop.
\end{equation}
We first define the homomorphism $\alpha:W\to K$ on $R= \Q[\R_+^\times]\subset W$. We denote  an element in $R$ as $x = \sum_ia_i\tau_W(x_i)$ with $a_i\in\Q$. We define $\alpha: R \to K$ by the formula
\begin{equation}\label{16}
\alpha\left(\sum_ia_i\tau_W(x_i)\right)=\sum_i a_i\tau(x_i).
\end{equation}
By applying the property \eqref{15} we get  $\alpha(\tau_W(x))=\tau(x)$, $\forall x\in H=\trop$.
We check now that $\alpha: R \to K$ is injective.
Let $R\ni x = \sum_i a_i\tau_W(x_i)\neq 0$. Let $x_0=\text{max}\{x_i\}$. One then has
\[
\rho(\alpha(x))=\rho\left(\sum_i a_i\tau(x_i)\right)\in \sum_i\rho(a_i\tau(x_i)).
\]
 One has  $\rho(a)=1$ for $a\in \Q$, $a>0$ and $\rho(-1)=-1$   and hence $\rho(a_i\tau(x_i))=\epsilon_i\rho\tau(x_i)=\epsilon_i x_i$
 with $\epsilon_i$ the sign of $a_i$. Thus $\rho(\alpha(x))\in \sum_H \epsilon_i x_i=\epsilon_0 x_0\neq 0$ and the injectivity is proven. From the injectivity just proven it follows that $\alpha: R \to K$  defines  a field homomorphism  $\alpha: \text{Frac}(R)\to K$. By construction one has $\alpha(\tau_W(x))=\tau(x)$ $\forall x\in \R$.

 It remains to show the second equality of \eqref{19}.
 Consider $\rho\circ\alpha$: to show that  this is equal to $\rho_W$ it is enough to prove that they agree on $R$ since both maps are multiplicative. Let $x =\sum_i a_i\tau_W(x_i)\in R$, then one has $\rho(\alpha(x))\in \sum_i\rho(a_i\tau(x_i))$ and by  the above argument one gets $\rho(\alpha(x))=\epsilon_0 x_0=\rho_W(x)$. This shows that the morphism $\alpha$ exists and it is unique because $\alpha(\tau_W(x))$ is necessarily equal to $\tau(x)$ so that by linearity we know $\alpha$ on $R$ and hence on $\text{Frac}(R)$.
 \endproof

 The following functoriality will be applied later in the paper.
 \begin{prop}\label{propfunct} Let $H$ be a hyperfield and assume that the universal $W$-model $W(H)$ for $H$ exists.
 Then there is a canonical   group homomorphism
 \begin{equation}\label{18}
 W: \text{Aut}(H)\to \text{Aut}(W(H)),\qquad W(\theta)=\alpha.
 \end{equation}
 \end{prop}
 \proof
 Let $\theta\in \text{Aut}(H)$ be an automorphism  of $H$. Let $(W(H)=K,\rho,\tau)$ be the universal $W$-model for $H$. Then, we consider the triple: $(K,\rho'=\theta^{-1}\circ\rho,\tau'=\tau\circ\theta)$.
 One sees that $\rho'$ is a homomorphism of hyperfields, that $\tau'$ is multiplicative and that $\rho'(\tau'(x))= \theta^{-1}(\rho(\tau(\theta(x))))=x$. Thus $(K=W(H),\rho',\tau')$ is also a $W$-model of $H$. Then it follows from  universality  that there exists a field homomorphism $\alpha: K\to K$ such that  the rules \eqref{19} hold, and in particular $\tau\circ\theta=\alpha\circ\tau$.
 From this and the fact that $\alpha$ is the identity when $\theta$ is the identity one deduces the existence of  the  group homomorphism \eqref{18}.\endproof

  When $H=\trop$ we derive from the above proposition the existence of a one parameter group  of automorphisms of $W(\trop)$ given by  the (images of the) $\theta_\lambda\in\text{Aut}(\trop)$ (\cf Proposition~\ref{propviro}). These form the one parameter group  of Frobenius  automorphisms
  \begin{equation}\label{froblambda}
  W(\theta_\lambda)=\arf_\lambda\in\text{Aut}(\text{Frac}(\Q[\R_+^\times])).
 \end{equation}
  The universal $W$-model $W(H)$ of a hyperfield $H$ (when it exists) inherits automatically the refined structure of the {\em field of quotients} of a natural ring
\begin{prop}\label{integ} Let $H$ be a hyperfield with a universal $W$-model $(W(H),\rho,\tau)$. Let $W_\Z(H)\subset W(H)$ (resp. $W_\Q(H)\subset W(H)$) be the (integral) subring (resp. sub $\Q$-algebra) generated by the $\tau(x)$'s, $x\in H$. Then one has
\begin{equation}\label{20}
W(H)=\text{Frac}(W_\Z(H))=\text{Frac}(W_\Q(H)).
\end{equation}
\end{prop}
\begin{proof} Let $K = \text{Frac}(W_\Z(H))\subset W(H)$. The map $\tau: H \to W(H)$ has, by construction, image in $K$ and using the restriction $\rho_K$ of $\rho: W(H) \to H$ to $K$ one gets a $W$-model $(K,\rho_K,\tau_K)$ for $H$. Thus by universality there exists a field homomorphism $\alpha: W(H)\to K$ such that $\alpha\circ \tau_K(x) = \tau(x)$, $\forall x\in H$. Hence $\alpha$ is surjective on $K$ (as $K$ is generated by the $\tau(x)$'s) and also injective (as field homomorphism). Then $\alpha: W(H)\stackrel{\sim}{\to}K$ is a field isomorphism. A similar proof shows the second equality in \eqref{20}.
\end{proof}
\begin{example}{\rm Let $H =\sign=\{0, \pm 1\}$, be the hyperfield of signs (\cf \cite{CC2}, Definition~2.2)  then $W_\Z(\sign)=\Z$.}
\end{example}
\begin{example}{\rm Let $H = \trop$, then $W_\Q(\trop)=\Q[\R_+^\times]$.}
\end{example}

  In $p$-adic Hodge theory (\cf \cite{FF}, \S5) one defines for a finite extension $E$ of $\Q_p$ with residue field $\F_q$ and for any (real) valued complete, algebraically closed field $F$ of characteristic $p$  extension of an algebraically closed field $k|\F_q$, a ring homomorphism $\theta: W_{\cO_E}(\cO_F) \to \C_p$, $\theta(\sum_{n\ge 0}[x_n]\pi^n) = \sum_{n\ge 0}x_n^{(0)}\pi^n$ ($\pi=\pi_E$ is a chosen uniformizer of $\cO_E$).

  At the real archimedean place, we have the following counterpart

\begin{prop}\label{proptheta} There exists a unique ring homomorphism $\theta: W_\Q(\trop)\to \R$ such that
\begin{equation}\label{22}
\theta([x]) =\theta(\tau(x))= x^{(0)} = x,\quad\forall x\in \trop.
\end{equation}
\end{prop}
\begin{proof} It follows from Proposition~\ref{integ} that $W_\Q(\trop) = \Q[\R_+^\times]$ and thus the natural map $\R_+^\times\to\R$ extends by linearity and uniquely to a ring homomorphism ($[~]=\tau$)
\begin{equation}\label{23}
\theta(\sum_ia_i[x_i]) = \sum_ia_ix_i\in\R.
\end{equation}
\end{proof}

\section{Universal formal pro-infinitesimal thickening of the field $\R$}\label{entrop}

Theorem~\ref{thmuni} states the existence of a universal algebraic object whose definition is independent of the completeness condition \eqref{projcond} of Proposition \ref{propwittiso}. Given a universal object among triples $(A,\rho,\tau)$ where $A$ is a (commutative) ring, $\rho:A\to R$ is a ring homomorphism with multiplicative section
$\tau:R\to A$, the corresponding data obtained by passing to the completion $\varprojlim_n A/\Ker(\rho)^n$ is automatically universal among the triples which fulfill \eqref{projcond}. This suggests to consider
the homomorphism  $\theta: W_\Q(\trop)\to \R$ of Proposition~\ref{proptheta} and introduce the following
  \begin{defn} The universal formal pro-infinitesimal thickening $\R_\infty$ of $\R$ is the $\Ker(\theta)$-adic completion of $W_\Q(\trop)$, \ie \[\R_\infty=\varprojlim_n W_\Q(\trop)/\Ker(\theta)^n.\]
  \end{defn}
 Next theorem shows that $\R_\infty$ has a richer structure than the ring $\R[[T]]$ of formal power series  with real coefficients. In fact we prove that the real vector space $\Ker(\theta)/\Ker(\theta)^2$ is infinite dimensional.

\begin{thm}\label{thmkerth}
$(i)$~Let $\ell:\R_+^\times\to \R$ be a  group homomorphism, then the  map
 \begin{equation}\label{taul}
    \cT_\ell(X)(z):=\sum_ia_i e^{\log(x_i)+ (z-1)\ell(x_i)}\qqq X=\sum_ia_i[x_i]\in W_\Q(\trop)
\end{equation}
defines a ring homomorphism $\cT_\ell:W_\Q(\trop)\to C^\infty(\R)$ to the ring of smooth real functions and
\begin{equation}\label{valatval}
   \theta(X)=\cT_\ell(X)(1),~\forall X\in W_\Q(\trop).
\end{equation}

$(ii)$~The Taylor expansion at $z=1$ induces a  ring homomorphism
\begin{equation}\label{surjectl}
   \R_\infty\stackrel{\cT_\ell}{\to}\R[[z-1]], \ \
\end{equation}
which is surjective if $x\ell(x)+(1-x)\ell(1-x)\neq 0$ for some $x\in \R_+^\times$.

$(iii)$~$\Ker(\theta)/\Ker(\theta)^2$ is an infinite dimensional real vector space  and the ``periods" $\pi_p=[p]-p\in \Ker(\theta)/\Ker(\theta)^2$, for $p$ a prime number, are linearly independent over $\R$.
\end{thm}
\proof $(i)$~For any $z\in \R$ the  map $\log+ (z-1)\ell$ defines a group homomorphism $\R_+^\times\to \R$ and the conclusion follows.\newline
$(ii)$~It follows from $(i)$ and \eqref{valatval} that one obtains the ring homomorphism \eqref{surjectl} since for $X\in \Ker(\theta)^n$ the function  $\cT_\ell(X)(z)$ vanishes of order $\geq n$ at $z=1$.
We show, under the assumption of $(ii)$, that $\cT_\ell$ as in \eqref{surjectl} is surjective. Let $x\in \R$ with $x\ell(x)+(1-x)\ell(1-x)\neq 0$. One has $s(x):=1-[x]-[1-x]\in \Ker(\theta)$; the first derivative of $\cT_\ell(s(x))(z)$ at $z=1$ is equal to $-(x\ell(x)+(1-x)\ell(1-x))$ and it does not vanish. Using the $\R$-linearity which follows from $W_\Q(\trop)/\Ker(\theta)=\R$ and by implementing the powers $s(x)^n$, one derives the surjectivity.\newline
$(iii)$~One has $\pi_n:=[n]-n\in \Ker(\theta)$ for any integer $n$ and hence for any prime number $n=p$. Next we show that these elements are linearly independent in $\Ker(\theta)/\Ker(\theta)^2$. The latter is a vector space over $W_\Q(\trop)/\Ker(\theta)=\R$, and the multiplication by a real number $y\in \R$ is provided by the multiplication by any $s\in  W_\Q(\trop)$ such that $\theta(s)=y$. In particular we can always choose the lift of the form $s=a[b]$ where $a\in \Q$ and $b\in \R_+^\times$. Assume now that there is a linear relation of the form
\begin{equation}\label{relat}
   X= \sum_i a_i[b_i]\pi_{p_i}\in \Ker(\theta)^2
\end{equation}
where $a_i\in \Q$, $b_i\in \R_+^\times$ and $p_i$ are distinct primes. Let $\ell:\R_+^\times\to \R$ be a  group homomorphism, then by using $(ii)$ we see that $\cT_\ell(X)$ vanishes at order $\geq 2$ at $z=1$, and thus
   $\left(\frac{d}{dz}\right)_{z=1}\cT_\ell (X)=0$.
Using $\pi_{p_i}\in \Ker(\theta)$, we derive, using \eqref{taul},
\begin{equation*}
\left(\frac{d}{dz}\right)_{z=1}\cT_\ell(a_i[b_i]\pi_{p_i})=a_i b_i p_i\ell(p_i)
    \end{equation*}
so that $\sum_i a_i b_i p_i\ell(p_i)=0$.
Since the logarithms of prime numbers are rationally independent, and since the additive group $\R$ is divisible, hence injective among abelian groups, one can construct a group homomorphism $\ell:\R_+^\times\to \R$ such that the values of $\ell(p_i)$ are arbitrarily chosen real numbers.  In view of the fact that the above relation is always valid, we derive that all the coefficients $a_i b_i p_i$ must  vanish and that the original relation is therefore trivial.\endproof

Next, we extend the above construction to obtain linear forms on $\Ker(\theta)/\Ker(\theta)^2$.
\begin{lem}\label{quasi}$(i)$~Let $\delta:W_\Q\to \R$ be a $\Q$-linear map such that
\begin{equation}\label{qlin}
    \delta(fg)=\theta(f)\delta(g)\qqq g\in \Ker(\theta).
\end{equation}
Then $\delta$  vanishes on $\Ker(\theta)^2$ and it defines an $\R$-linear form on $\Ker(\theta)/\Ker(\theta)^2$.

$(ii)$~Let $\psi:\R\to \R$ be such that $\psi(-x)=-\psi(x)$ for all $x\in \R$ and
\begin{equation}\label{quasi1}
    \psi(x(y+z))-x\psi(y+z)=\psi(xy)-x\psi(y)+\psi(xz)-x\psi(z)\qqq x,y,z\in\R
\end{equation}
then the following equality defines a $\Q$-linear map $\delta_\psi:W_\Q\to \R$ fulfilling \eqref{qlin}
\begin{equation}\label{qlin1}
   \delta_\psi(\sum_j a_j [x_j]):= \sum_j a_j \psi(x_j)\qqq a_j\in \Q, \ x_j\in \R.
\end{equation}
\end{lem}
\proof $(i)$~For $f,g\in \Ker(\theta)$, it follows from \eqref{qlin} that $\delta(fg)=0$, thus $\delta$  vanishes on $\Ker(\theta)^2$. The action of $s\in \R$ on $g\in \Ker(\theta)/\Ker(\theta)^2$ is given by $fg$ for any $f\in W_\Q$ with $\theta(f)=s$. Thus the $\R$-linearity of the restriction of $\delta$ to $\Ker(\theta)$ follows from \eqref{qlin}.

$(ii)$~By construction the map $\delta$ defined by \eqref{qlin1} is well defined since $\psi$ is odd, and $\Q$-linear. To check
\eqref{qlin} we can assume that $f=[x]$ for some $x\in \R$. One then has
\begin{equation}\label{quasi2}
   \delta(fg)-\theta(f)\delta(g)=\sum_j b_j(\psi(xy_j)-x\psi(y_j))\qqq g=\sum_j b_j [y_j].
\end{equation}
The map $L:\R\to \R$ given by $L(y)=\psi(xy)-x\psi(y)$ is additive by \eqref{quasi1} and thus
\begin{equation*}
    \sum_j b_j(\psi(xy_j)-x\psi(y_j))=\sum_j b_j L(y_j)=L(\sum_j b_j y_j).
\end{equation*}
When $g\in \Ker(\theta)$, one derives $\theta(g)=\sum_j b_j y_j=0$ and thus one obtains \eqref{qlin}.\endproof

The next statements show how the entropy appears naturally to define ``periods".

 \begin{lem}\label{entropysym} $(i)$~The symbol $s(x):=1-[x]-[1-x]$ defines a map  $\R\to\Ker(\theta)\subset W_\Q(\trop)$ such that
\begin{eqnarray}
  (a) &:& s(1-x)=s(x)\nonumber \\
  (b) &:& s(x+y)=s(y)+[1-y]s(\frac{x}{1-y})+[y] s(-\frac x y)\nonumber  \\
  (c) &:& [x] s(1/x)=-s(x). \nonumber
\end{eqnarray}
$(ii)$~The $\R$-linear span in $\Ker(\theta)/\Ker(\theta)^2$ of the $s(x)$, for $x\in \R$ generates $\Ker(\theta)/\Ker(\theta)^2$.
\end{lem}
 \proof
$(i)$~The symbol $[x]$ extends  to $\R$ by $[-x]=-[x]$. The equality  $(a)$ holds by construction.
We check (b) ($(c)$ is checked in the same way). One has
\begin{equation*}
    [1-y]s(\frac{x}{1-y})=[1-y]-[x]-[1-y-x], \ \ [y] s(-\frac x y)=[y]+[x]-[x+y]
\end{equation*}
\begin{equation*}
  [1-y]s(\frac{x}{1-y})+[y] s(-\frac x y)=[1-y]-[1-y-x]+[y]-[x+y]=s(x+y)-s(y).
\end{equation*}

$(ii)$~The $\R$-linear span in $\Ker(\theta)/\Ker(\theta)^2$ of the $s(x)$ contains all the $[x+y]s(x/(x+y))$ for $x+y\neq 0$, and hence all the $[x+y]-[x]-[y]$. Let $f\in \Ker(\theta)$, then a non-zero integer multiple of $f$ is of the form
\begin{equation*}
    nf=\sum_j [x_j]-\sum_k [y_k]\,, \ \ \sum_j x_j=\sum_k y_k
\end{equation*}
and both $\sum_j [x_j]-[\sum_j x_j]$ and $\sum_k [y_k]-[\sum_k y_k]$ belong to the $\R$-linear span in $\Ker(\theta)/\Ker(\theta)^2$ of the $s(z)$.
\endproof
\begin{thm}\label{thmentropysym} The space  $\Ker(\theta)/\Ker(\theta)^2$ is the infinite dimensional $\R$-vector space $\Omega$ generated by the symbols $\varepsilon(x)$, $x\in \R$, with relations
\begin{eqnarray}
  (A) &:& \varepsilon(1-x)=\varepsilon(x)\nonumber \\
  (B) &:& \varepsilon(x+y)=\varepsilon(y)+(1-y)\,\varepsilon(\frac{x}{1-y})+y \, \varepsilon(-\frac x y)\qqq y\notin\{0,1\}\nonumber  \\
  (C) &:& x \,\varepsilon(1/x)=-\varepsilon(x)\qqq x\neq 0. \nonumber
\end{eqnarray}
\end{thm}
\proof The map $\varepsilon(x)\mapsto s(x)\in \Ker(\theta)/\Ker(\theta)^2$ is well defined and surjective by Lemma \ref{entropysym}. We show that it is injective. Let $M$ be an $\R$-linear form on $\Omega$. Next we prove that there exists $\delta:W_\Q\to \R$ fulfilling \eqref{qlin}, such that
\begin{equation}\label{deltas}
    \delta(s(x))=M(\varepsilon(x))\qqq x\in \R.
\end{equation}
The injectivity then follows using $\R$-linearity to get $M(Z)=0$ for any $Z$ in the kernel. We now prove \eqref{deltas}. Let $H(x)=M(\varepsilon(x))$, then, as explained in
\cite{Ko} and
 Remark 5.3 of \cite{C},  the function $\phi(x,y)=(x+y)H(\frac{x}{x+y})$,  $\phi(x,-x)=0$,
is a two cocycle on the additive group of $\R$, with coefficients in $\R$. This symmetric  cocycle  defines an extension in the category of torsion free divisible abelian groups, \ie of $\Q$-vector spaces, and hence  is a coboundary, $\phi=b\psi$. With $\psi_x(y):=\psi(xy)$ one has $b\psi_x(y,z)=\phi(xy,xz)=x\phi(y,z)$. This shows that  $\psi$ fulfills \eqref{quasi1} and that replacing $\psi$ by $\frac 12(\psi-\psi_{-1})$ one can assume that $\psi$ is odd.
Then $\delta_\psi$ defined in \eqref{qlin1} fulfills \eqref{deltas}
since $\delta_\psi(s(x))=b\psi(x,1-x)=H(x)$. \endproof
\begin{rem}\label{log}{\rm
 All Lebesgue measurable group homomorphisms $\ell:\R_+^\times\to \R$ are of the form $x\mapsto \lambda\log x$, and yield the linear form on $\Omega$ given by the entropy function. In the next section \ref{sec:Frechet} we  investigate the ring homomorphism $\cT_\ell$ of \eqref{taul} given by the measurable choice $\ell=\log$.
}\end{rem}

\section{The $\mathbb R$-algebras of analytic functions and their canonical form}

\label{sec:Frechet}

This section is concerned with the {\em topological} step inherent to the construction of the archimedean analogue of the rings which in the $p$-adic case are the analogues  in mixed characteristics of the ring of rigid analytic functions on the punctured unit disk in equal characteristics  (\cf Appendix~\ref{algebras}).
In Section~\ref{sec:2ndsection} we have seen that the structure of the $\Q$-algebras $W_\Q(\cO)\subset W_\Q(\trop)$ is inclusive of a one parameter group of automorphisms $\arf_\lambda=W(\theta_\lambda)\in \Aut(W_\Q(\trop))$ (\cf \eqref{froblambda}) preserving  $W_\Q(\cO)$, and of the ring homomorphism $\theta: W_\Q(\trop)\to \R$ (\cf Proposition~\ref{proptheta}). Thus the following map defines a ring homomorphism from the Witt ring $W_\Q(\cO)$ to the ring of real valued functions of one (positive) real variable endowed with the pointwise operations
\begin{equation}
  W_\Q(\cO)\to\mathcal F(\R_+^\times),\quad  x\to x(z)=\theta(\arf_z(x)) \qqq z>0.
\end{equation}
For $x \in W_\Q(\cO)$, one has $x=\sum_ia_i[x_i]$ with $x_i\in (0,1]$ (we keep the notation $[x]=\tau(x)$ of \S\ref{sec:2ndsection}). The function $x(z)=\sum_i a_i x_i^z$, $z>0$, is bounded by the norm $\|x\|_0=\sum_i|a_i|$. After performing the compactification of the  balls $\|x\|_0\leq R$ of this norm for the topology of simple convergence
\begin{equation}\label{simplecv}
    x_n\to x\Longleftrightarrow x_n(z)\to x(z)\qqq z>0,
\end{equation}
one obtains
 the Banach algebra $\btplusarc$  which is the real archimedean counterpart of the $p$-adic ring $\wittOovpi$ (\cf Appendix~\ref{algebras}).

 The main result of this section is stated in Theorem~\ref{thmdec} that describes the canonical expansion of the elements of $\btplusarc$. The following \S\ref{anasect} prepares the ground. In \S\ref{rn} we construct the Frechet $\R$-algebra obtained by completion for the analogue of the $\|.\|_\rho$ norms   and relate it to the Mikusinski field (\cf \S\ref{mikfieldr}).

\subsection{The algebra  $\btplusarc$}\label{anasect}

We denote by NBV (Normalized and of Bounded Variation) the class of real functions $\phi(\xi)$ of a real variable $\xi$ which are of bounded variations and normalized \ie point-wise left continuous and tending to $0$ as $\xi\to -\infty$. In fact we shall only work with functions that vanish for $\xi\leq 0$, and say that a real valued function $\phi$ on $(0,\infty)$ is NBV when its extension by  $0$ for $\xi \leq 0$ is NBV. Thus, saying that a function $\phi$ of bounded variation on $(0,\infty)$ is normalized just means that it is left continuous at every point of $(0,\infty)$. We refer to \cite{R}, Chapter 8.

We observe that with the notation of \eqref{taurgood}, for $a,b\in \R$, one has
\begin{equation}\label{acupb}
 a \smile b=a\Leftrightarrow |a|\geq |b|, \quad\text{and}\quad a=b\ \text{if}\ |a|= |b|.
\end{equation}
Throughout this section we continue to use the notation $[~]=\tau$ of \S\ref{sec:2ndsection}.
\begin{prop} \label{propnbv}
$(i)$~Let $\phi(\xi)$ be a real valued, left continuous function of $\xi\in (0,\infty)$ of bounded variation and let $V$ be its total variation. Then, there exists a measurable function $u\mapsto x_u\in [-1,1]\setminus \{0\}$ of $u \in [0,V)$ such that
$x_u\smile x_v=x_u$ for $u\leq v$  and
\begin{equation}\label{equalint}
    \int_0^V [x_u](z)du=\int_0^\infty e^{-\xi z}d\phi(\xi)\qqq z\in\R^\times_+.
\end{equation}
Moreover the function $u\mapsto x_u$ is unique almost everywhere (\ie except on a set of Lebesgue measure zero).

$(ii)$~Conversely, given $V<\infty$ and a measurable function $u\mapsto x_u\in [-1,1]\setminus \{0\}$ of $u \in [0,V)$, such that
$x_u\smile x_v=x_u$ for $u\leq v$, there exists a unique real valued left continuous function $\phi(\xi)$,  $\xi\in (0,\infty)$ of total variation $V$ such that \eqref{equalint} holds.
\end{prop}
\proof $(i)$~Let $\mu$ be the unique real Borel measure on $[0,\infty)$ such that: $\mu([0,\xi))=\phi(\xi),\forall\xi > 0$.
Then the definition of the integral on the right hand side of \eqref{equalint} is
\begin{equation*}
    \int_0^\infty e^{-\xi z}d\phi(\xi)=\int_0^\infty e^{-\xi z}d\mu.
\end{equation*}
The total variation function $T_\phi(\xi)$ which is defined as (with $\phi(0)=0$)
\begin{equation*}
    T_\phi(\xi)=\sup\left(\sum_{j=1}^n|\phi(\xi_j)-\phi(\xi_{j-1})|\right)\,, \ \ 0=\xi_0<\xi_1<\ldots <\xi_n=\xi
\end{equation*}
is equal to $|\mu|([0,\xi))$ where $|\mu|$ is the positive Borel measure on $[0,\infty)$ which is the total variation of $\mu$ (\cf \cite{R}, Theorem 8.14). We set, for $u\in [0,V)$
\begin{equation}\label{defnS}
    S(u)=\inf\{\xi\in(0,\infty)\mid T_\phi(\xi)> u\}\in [0,\infty).
\end{equation}
One has $S(u)\geq S(v)$, when $u\geq v$. Moreover the function $S$ of $u$ is right continuous and is finite since the total variation of $\phi$ is  $V=V(\phi)=\displaystyle{\lim_{\xi\to \infty}} T_\phi(\xi)$. Let $m$ be the Lebesgue measure on $(0,V(\phi))$. The direct image of $m$ by $S$ is equal to the measure $|\mu|$. Indeed, one has
\begin{equation*}
    T_\phi(\xi)>u\Longleftrightarrow S(u)<\xi
\end{equation*}
which shows that the Lebesgue measure $S(m)([0,\xi))$ of the set $\{u\mid S(u)<\xi\}$ is equal to $T_\phi(\xi)=|\mu|([0,\xi))$, $\forall\xi$. Let $h(\xi)$ be the essentially unique
measurable function with values in $\{\pm 1\}$ such that $\mu=h|\mu|$ (\cf \cite{R}, Theorem 6.14). We define the function $u\mapsto x_u$ by
\begin{equation}\label{xu}
    x_u=h(S(u))e^{-S(u)}\qqq u\in [0,V).
\end{equation}
One has $|x_u|\geq |x_v|$ for $u\leq v$, moreover $|x_u|= |x_v|\implies x_u=x_v$,
since $|x_u|= |x_v|$ implies $S(u)=S(v)$. Moreover, since $S(m)=|\mu|$ one gets
\begin{equation*}
    \int_0^V f(S(u))dm=\int_0^\infty f(\xi)|d\mu|
\end{equation*}
and taking $f(\xi)=h(\xi)e^{-z\xi}$ one obtains \eqref{equalint}. The proof of the uniqueness is postponed after the proof of $(ii)$.

$(ii)$~Let $\sigma(u)={\rm sign}(x_u)$, $S(u)=-\log(|x_u|)$. These functions are well defined on the interval  $J=[0,V)$. One has $S(u)\geq S(v)$, when $u\geq v$. Let $T(\xi)$ be defined by
\begin{equation}\label{defnT}
    T(\xi)=\sup\{u\in J\mid S(u)<\xi\}\in [0,V].
\end{equation}
 $T(\xi)$ is non-decreasing and left continuous by construction, and thus it belongs to the class NBV. By hypothesis $|x_u|= |x_v|\implies x_u=x_v$, so that the function $\sigma(u)={\rm sign}(x_u)$ only depends upon $S(u)$ and can be written as $h(S(u))$ where $h$ is measurable and takes values in $\{\pm 1\}$. We extend $h$ to a measurable function $h:[0,\infty)\to \{\pm 1\}$. Let $\mu= h dT$ be the real Borel measure such that
\begin{equation}\label{mumu}
    |\mu|([0,\xi))=T(\xi)\,, \ \ \mu=h|\mu|
\end{equation}
Then the function $\phi(\xi)=\mu([0,\xi))$ is  real valued, NBV and such that \eqref{equalint} holds. Indeed, one has
\begin{equation*}
    T(\xi)>u\Longleftrightarrow S_+(u)<\xi, \ \ S_+(u)=\lim_{\epsilon\to 0\atop \epsilon>0}S(u+\epsilon)
\end{equation*}
so that the map $S$ associated to the function $\phi$ in the proof of $(i)$ is equal to $S_+$ and thus it agrees with $S$ outside a countable set. Hence the function $x_u$ associated to $\phi$ by \eqref{xu} agrees with the original $x_u$ almost everywhere and one gets  \eqref{equalint}.

Finally, we prove the uniqueness statement of $(i)$. It is enough to show that the function $f(z)$ given by
\begin{equation}\label{f(z)}
    f(z)=\int_0^\infty [x_u](z)du
\end{equation}
uniquely determines $x_u$ almost everywhere. It follows from the above discussion that it is enough to prove that $f(z)$ uniquely determines the function $\phi(\xi)$. By \cite{R}, Theorem 8.14,
the function $\phi\in$ NBV is uniquely determined by the associated measure $\mu$. The latter is uniquely determined by $f$ since one has
\begin{equation*}
    f(z)=\int_0^\infty e^{-\xi z} d\mu(\xi).
\end{equation*}
Hence $f$ is the Laplace transform of $\mu$ and this property determines uniquely the finite measure $\mu$.\endproof

The following definition introduces the real archimedean counterpart of the ring $ \wittOovpi$ of $p$-adic Hodge theory.
\begin{defn}\label{defnbtplus} We denote by $\btplusarc$  the space of real functions of the form $\int_0^V [x_u]du$ where $V<\infty$, and $u\mapsto x_u\in [-1,1]$ is a measurable function of $u \in [0,V]$, such that
\[
x_u\smile x_v=x_u,\quad\text{for}\quad u\leq v.
\]
\end{defn}

Then Proposition~\ref{propnbv} shows that the functions in $\btplusarc$ are exactly
the Laplace transforms of finite real Borel measures
\begin{equation}\label{fzfz}
    f(z)=\int_0^\infty e^{-\xi z} d\mu(\xi).
\end{equation}
Moreover, when expressed in terms of $\mu$ one gets
\begin{equation}\label{equsup0}
   \|f\|_0:=\sup\{u\in[0,V]\mid x_u\neq 0\}=|\mu|([0,\infty))
\end{equation}
This can be seen using \eqref{defnT} to get
\begin{equation*}
    \sup\{u\in[0,V]\mid x_u\neq 0\}= \sup\{u\in[0,V]\mid S(u)<\infty\}=\lim_{\xi\to \infty}T(\xi)
    =|\mu|([0,\infty)).
\end{equation*}
This shows that $\btplusarc$ is the real Banach algebra of convolution of finite real Borel measures on $[0,\infty)$.
Then we obtain the following result

\begin{thm}\label{mainwitt} The space $\btplusarc$, endowed with the pointwise operations of functions and the map $f\mapsto \|f\|_0=\sup\{u\in\R_{\ge 0}\mid x_u\neq 0\}$,  is a real Banach algebra.
\end{thm}

\subsection{Canonical form of elements of $\btplusarc$}\label{compsect}

In the $p$-adic case (\cf Appendix~\ref{algebras}),  every element $x\in\wittOovpi$ can be written uniquely in the form
\begin{equation}\label{Bb0}
x=\sum_{n\gg-\infty}[x_n]\pi^n\,, \ \ x_n\in\cO_F.
\end{equation}
In the archimedean case one gets an analogous decomposition by applying Proposition \ref{propnbv}. In the next pages we shall explain this point with care since the decomposition of the elements in $\btplusarc$ does {\em not} arise by applying \eqref{equalint} naively.

\begin{thm}\label{thmdec} Let $f\in\btplusarc$. Then there exists $s_0>-\infty$ and a measurable function, unique   except on a set of Lebesgue measure zero,
$s\mapsto f_s\in [-1,1]\setminus \{0\}$, for $s>s_0$ such that $f_s\smile f_t=f_s$ for $s\leq t$ and
\begin{equation}\label{canonicalform}
    f=\int_{s_0}^\infty [f_s]e^{-s}ds.
\end{equation}
\end{thm}
\proof By Proposition~\ref{propnbv} there exists a measurable function $u\mapsto x_u\in [-1,1]\setminus \{0\}$ of $u \in [0,V)$ such that
$x_u\smile x_v=x_u$ for $u\leq v$  so that
\begin{equation}\label{start}
   f= \int_0^V [x_u]du.
\end{equation}
Define $f_s$ by the equality
\begin{equation}\label{defnfs}
    f_s=x_{(V-e^{-s})}\qqq s \geq s_0=-\log V.
\end{equation}
The function $V-e^{-s}$ is increasing and $d(V-e^{-s})=e^{-s} ds$, so that
\eqref{canonicalform} follows from \eqref{start} by applying a change of variables. \endproof

The scalars $\R$, \ie the constant functions in $\btplusarc$, are characterized by the condition
\begin{equation}\label{scalarchar}
    \int_{s_0}^\infty [f_s]e^{-s}ds\in \R \Longleftrightarrow f_s\in \sign\qqq s
\end{equation}
where $\sign = \{-1,0,1\}$ is the hyperfield of signs (\cf \cite{CC2}).
 This corresponds, in the $p$-adic case, to the characterization of the elements of the local field $K\subset\wittOovpi$ by the condition
\begin{equation*}
    \sum_{n\gg-\infty} [a_n]\pi^n\in K\Longleftrightarrow a_n\in k\qqq n
\end{equation*}
where $k$ is the residue field of $K$ (\cf \cite{FF1}, \S 2.1
and also Appendix~\ref{algebras}).

 In the $p$-adic case, the projection $\cO_F\to k_F$ (\cf Appendix~\ref{algebras}) induces an augmentation map $\varepsilon$ obtained by applying the above projection to each $a_n$ inside the expansion $f=\sum_{n\gg-\infty} [a_n]\pi^n$ (\cf \eqref{aug} in Appendix \ref{algebras}). In the real archimedean case, the corresponding projection is the map
 \[
 \trop\supset [-1,1]=\cO\to \sign,\quad x\mapsto \tilde x=\left\{ \begin{array}{rcl}0&\mbox{if}&x\in(-1,1)\\
 \pm 1&\mbox{if}&x=\pm 1\end{array}\right.
  \]
 When this projection is applied inside the expansion $f=\int_{s_0}^\infty [f_s]e^{-s}ds$ of elements in $\btplusarc$, it yields the following
 \begin{prop} \label{augmentation}
 For $f\in\btplusarc$, let $f=\int_{s_0}^\infty [f_s]e^{-s}ds$ be its canonical form. Then
\begin{equation}\label{charepsilon}
    \epsilon(f):=\int_{s_0}^\infty [\tilde{f_s}]e^{-s}ds=\lim_{z\to +\infty}f(z)
\end{equation}
defines a character $\epsilon: \btplusarc\to \R$ of the Banach algebra $\btplusarc$.
\end{prop}
\proof
When $z\to \infty$, one has for any  $s\geq s_0$, $[f_s](z)\to [\tilde{f_s}](z)$. Thus it follows from the Lebesgue dominated convergence theorem that (\cf \eqref{f(z)})
\begin{equation}\label{limlim}
   \lim_{z\to +\infty}f(z)=\int_{s_0}^\infty [\tilde{f_s}]e^{-s}ds.
\end{equation}
Since the operations in the Banach algebra $\btplusarc$  are pointwise when the elements are viewed as functions of $z$, the functional $\epsilon$ is a character.
\endproof

Next, we exploit a Theorem of Titchmarsh to show that the leading term of the canonical form
\eqref{canonicalform} behaves multiplicatively likewise its $p$-adic counterpart. In particular, it will also follow that the ring $\btplusarc$ is integral (\ie it has no zero divisors).
\begin{thm}\label{titsmarsh}
The following formula defines a multiplicative map from the subset of non zero elements of $\btplusarc$ to $(0,1]$
\begin{equation}\label{mult}
    |\rho|(f)=\lim_{\epsilon\to 0+}|f_{s_0+\epsilon}| \qqq f=\int_{s_0}^\infty [f_s]e^{-s}ds\in\btplusarc\setminus\{0\}.
\end{equation}
\end{thm}
\proof Using \eqref{defnfs}, we can write the map $|\rho|$, with the notations of Proposition~\ref{propnbv} as
\begin{equation*}
    |\rho|(x)=\lim_{\epsilon\to 0+}|x_\epsilon| \qqq x=\int_0^V [x_u]du.
\end{equation*}
By \eqref{defnS} one has
\begin{equation*}
  \lim_{\epsilon\to 0+}|x_\epsilon|=e^{-S(0)}, \ \  S(0)=\inf\{\xi\mid T_\phi(\xi)> 0\}\in [0,\infty).
\end{equation*}
In terms of the measure $\mu=d\phi$, $S(0)$ is the lower bound of the support of $\mu$. By Titchmarsh's Theorem \cite{T} (formulated in terms of distributions \cite{L}) one has
the additivity of these lower bounds  for the convolution of two measures on  $[0,\infty)$
\begin{equation*}
    \inf{\rm Support}(\mu_1\star\mu_2)=\inf{\rm Support}(\mu_1)
    +\inf{\rm Support}(\mu_2)
\end{equation*}
and hence the required multiplicativity. \endproof

\begin{rem}\label{care}{\rm  One important nuance between the $p$-adic case and the archimedean case is in the behavior of $s_0$ under the algebraic operations. As in the $p$-adic case the quantity $V=e^{-s_0}$ defines a norm, $\|.\|_0$, but this norm  is no longer ultrametric and is  {\em sub-multiplicative} (\cf Lemma \ref{lemsubmul}) while its
$p$-adic counterpart (\cf Lemma \ref{newnorm} (ii)) is multiplicative. It remains multiplicative for positive measures.
}\end{rem}

\subsection{The real archimedean norms $\|.\|_\rho$}\label{rn}

We recall that in $p$-adic Hodge theory one defines, for each $\rho\in (0,1)$, a multiplicative norm on the ring
\begin{equation}\label{Bbprime}
\wittOovpi=W_{\cO_K}(\cO_F)[\frac{1}{\pi}]=\{f=\sum_{n\gg-\infty}[a_n]\pi^n\in \mathfrak E_{F,K}|a_n\in\cO_F,\forall n\}
\end{equation}
(\cf Appendix~\ref{algebras} for notations) by letting
\begin{equation}\label{norm}
    |f|_\rho=\max_\Z |a_n|\rho^n.
\end{equation}
 To define the real archimedean counterpart of the norm $|.|_\rho$, one needs first to rewrite \eqref{norm} in a slightly different manner without changing the uniform structure that it describes. Let $q$ be the cardinality of the field of constants $k$  so that $|\pi|=q^{-1}$. Rather than varying $\rho\in (0,1)$ we introduce a real positive parameter $\alpha>0$ and make it varying so that $\rho^\alpha=q^{-1}$.
  \begin{lem} \label{newnorm}
  $(i)$~For $\rho\in (0,1)$,  we set
   $
 \alpha=\frac{\log q}{-\log\rho},\quad (\rho=q^{-1/\alpha}).
 $

  Then, for $f=\displaystyle{\sum_{n\gg-\infty}}[a_n]\pi^n\in\wittOovpi$ the equality
  \begin{equation}\label{norm1}
  |f|_\rho^\alpha=\max_\Z |a_n|^\alpha q^{-n}
\end{equation}
  defines a multiplicative norm on $\wittOovpi$ that describes the same uniform structure as the norm  $|.|_\rho$ and whose restriction to $\Q_p$ is independent of $\rho$. \vspace{.05in}

  $(ii)$~The limit, as $\rho\to 0$,  of $|f|^\alpha_\rho$ is the norm $|f|_0=q^{-r}$, where $r$ is the smallest integer such that $a_r\neq 0$ (\cf\cite{FF1}, \S3.1).
  \end{lem}
 \proof $(i)$~One has $\rho^\alpha=q^{-1}$, thus the second equality of \eqref{norm1} holds. Since $|\cdot|_\rho$ is a multiplicative norm the same statement holds for $|\cdot|_\rho^\alpha$. When restricted to $\Q_p$ the expression \eqref{norm1} is independent of $\alpha$ since one has
 $|a_n|\in \{0,1\}$ $\forall n$.

$(ii)$~As $\rho\to 0$, also $\alpha\to 0$ and  $|f|^{\alpha}_\rho\stackrel[\alpha\to 0]{}{\longrightarrow}\displaystyle{\max_{n\in\Z\atop a_n\neq 0}}~q^{-n}=|f|_0$.
\endproof

In the real archimedean case, the operation of taking the ``$\max$'' in \eqref{norm1} is replaced by an integration process.  By   re-scaling, we can replace $\log q$ by $1$; then the archimedean analogue of \eqref{norm1} is given, for each $\rho\in (0,1)$ and for $\alpha=-\frac{1}{\log\rho}$ by the formula
\begin{equation}\label{ananorm}
    \|f\|_\rho:=\int_{s_0}^\infty |f_s|^\alpha e^{-s}ds\qqq f=\int_{s_0}^\infty [f_s]e^{-s}ds\in\btplusarc.
\end{equation}

\begin{lem}\label{lemsubmul} Let $\rho\in [0,1)$. Equation \eqref{ananorm} defines a sub-multiplicative norm on $\btplusarc$. For $\rho>0$ one has, with $\alpha=-\frac{1}{\log\rho}$
\begin{equation}\label{ananormfz}
    \|f\|_\rho=\int_{0}^\infty e^{-\xi \alpha} |d\mu(\xi)|\qqq f(z)=
    \int_{0}^\infty e^{-\xi z} d\mu(\xi).
\end{equation}
For $\rho=0$ this norm coincides with the norm $\|f\|_0$ of Theorem~\ref{mainwitt}.
\end{lem}
\proof Using \eqref{defnfs}, we can write the functional $\|f\|_\rho$ with the notations of Proposition~\ref{propnbv}
\begin{equation}\label{rhonormxx}
   \|x\|_\rho=\int_0^V|x_u|^\alpha du\qqq x=\int_0^V [x_u]du.
\end{equation}
Using the notations of the proof of Proposition \ref{propnbv}, one has
\begin{equation*}
    \int_0^V|x_u|^\alpha du=\int_0^V e^{-\alpha S(u)}  du=
    \int_{0}^\infty e^{-\xi \alpha} |d\mu(\xi)|
\end{equation*}
since the image of the Lebesgue measure $m$ on $[0,V)$ by the map $S$ is the measure $|d\mu(\xi)|$. Thus we obtain \eqref{ananormfz}.
The sum of the functions associated with the measures $d\mu_j$ ($j=1,2$) corresponds to the measure $d\mu_1+d\mu_2$. Thus one derives the triangle inequality
$\|f_1+f_2\|_\rho\leq \|f_1\|_\rho+ \|f_2\|_\rho.
 $
 The product of the functions corresponds to the convolution $d\mu=d\mu_1\star d\mu_2$ of the measures
\begin{equation}\label{convol1}
\int_0^\infty h(\xi)d\mu(\xi)=\int_0^\infty\int_0^\infty h(\xi_1+\xi_2)d\mu_1(\xi_1)d\mu_2(\xi_2)
\end{equation}
which is the projection of the product measure $d\mu_1\otimes d\mu_2$ by the map $(\xi_1,\xi_2)\mapsto s((\xi_1,\xi_2))=\xi_1+\xi_2$. The module of the product measure $d\mu_1\otimes d\mu_2$
is $|d\mu_1|\otimes |d\mu_2|$.  Moreover, for $h\geq 0$ a real positive function one has
\[
\int_0^\infty h |d\nu|=\sup\{\left\arrowvert\int_0^\infty h\psi d\nu\right\arrowvert : |\psi|\leq 1\}.
 \]
 It follows that the module of the projection of a measure is less than or equal to the projection of its module. Thus we derive
\[
\|f_1 f_2\|_\rho=\int_0^\infty  e^{\xi/\log\rho}|d\mu(\xi)|\leq \int_0^\infty e^{(\xi_1+\xi_2)/\log\rho}|d\mu_1(\xi_1)\|d\mu_2(\xi_2)|=\|f_1\|_\rho\|f_2\|_\rho
\]
which proves that $\|.\|_\rho$ is sub-multiplicative.
The limit case $\rho=0$ arises by taking the limit $\alpha\to 0$ in \eqref{ananorm}, hence we obtain
\begin{equation}\label{ananorm0}
    \|f\|_0=\int_{0}^V du,\quad  V=\sup\{u\mid x_u\neq 0\}.
\end{equation}
Thus $\|f\|_0 $ agrees with the norm  of Theorem~\ref{mainwitt}.\endproof

 The norms $\|.\|_\rho$ behave coherently with  the action of the automorphisms $\arf_\lambda$, more precisely one has the equality
\begin{equation}\label{compnorm}
    \|\arf_\lambda(f)\|_\rho=\|f\|_{\rho^{1/\lambda}}.
\end{equation}
Indeed, for $f(z)=\int_0^\infty e^{-\xi z}d\mu(\xi)$ one has
 \begin{equation}\label{frobnorm}
 \arf_\lambda(f)(z)=\int_0^\infty e^{-\lambda\xi z}d\mu(\xi)=\int_0^\infty e^{-\xi z}d\mu(\xi/\lambda)
 \end{equation}
which gives \eqref{compnorm}.

By construction the archimedean norms $\|.\|_\rho$ fulfill the inequality
\begin{equation}\label{inequnorm}
   \|f\|_\rho\leq \|f\|_{\rho'}\qqq \rho\geq \rho'.
\end{equation}
In particular for $I\subset (0,1)$  a closed interval, one has, with $\rho_0=\min I$ the smallest element of $I$
\begin{equation}\label{normI}
    \|f\|_I=\sup_{\rho\in I} \|f\|_\rho=\|f\|_{\rho_0}.
\end{equation}
\begin{defn}\label{Binf+} We define $B^+_\infty$ to be the Frechet  algebra  projective limit of the Banach algebras  completion of $\btplusarc$  for the norms $\|\cdot\|_\rho$, for $\rho \in (0,1)$.
 \end{defn}
It is straightforward to check, using \eqref{ananormfz}, that $B^+_\infty$ is the convolution algebra of measures $\mu$ on $[0,\infty)$ such that
\begin{equation}\label{convolrho}
   \int_0^\infty  e^{-\alpha\xi}|d\mu(\xi)|<\infty \qqq \alpha >0.
   \end{equation}
   \begin{prop}\label{ideal}
  $(i)$~The measures which are absolutely continuous with respect to the Lebesgue measure form an ideal $\jj\subset B^+_\infty$.

  $(ii)$~Let $B^+_{\infty,0}\subset B^+_\infty$ be the sub-ring obtained by adjoining the unit: \[
   B^+_{\infty,0}=\jj +\R \subset B^+_\infty.
    \]
    Then for $f\in B^+_\infty$, one has $f\in B^+_{\infty,0}$ if and only if the map $\lambda\mapsto \arf_\lambda(f)\in B^+_\infty$ is continuous for the Frechet topology of $B^+_\infty$.
    \end{prop}
    \proof $(i)$~Follows from the well known properties of convolution of measures.

    $(ii)$~For $f\in B^+_{\infty,0}$ the associated measure $\mu$ is of the form $a\delta_0+h d\xi$ where $h$  is locally integrable and fulfills
    \begin{equation}\label{hh}
   \int_0^\infty  e^{-\alpha\xi}|h(\xi)|d\xi<\infty \qqq \alpha >0.
   \end{equation}
   The measure associated to $\arf_\lambda(f)$ is $a\delta_0+h_\lambda d\xi$ where
   $h_\lambda(\xi)=\frac 1\lambda h(\xi/\lambda)$ by \eqref{frobnorm} and one has
   \begin{equation*}
    \|\arf_\lambda(f)-\arf_{\lambda'}(f)\|_\rho=\int_0^\infty  e^{-\alpha\xi}|h_\lambda(\xi)-h_{\lambda'}(\xi)|d\xi
   \end{equation*}
   It follows that the map $\lambda\mapsto \arf_\lambda(f)\in B^+_\infty$ is continuous for the Frechet topology of $B^+_\infty$. The converse is proven using $\int h_n(\lambda)\arf_\lambda(f)d\lambda\to f$ for suitable functions $h_n$.\endproof

\subsection{Embedding in  the Mikusinski field $\mikfield$}\label{mikfieldr}

In operational calculus one introduces the Mikusinski ring $\mik(\R_+)$  whose elements are functions on $\R_+$ with locally integrable derivative, and where the product law is the Duhamel product:
\begin{equation}\label{convol2}
F\star G(t)=\frac{d}{dt}\int_0^t F(u)G(t-u)du
\end{equation}
(\cf \cite{GPS}). This ring   plays a main role in analysis, in view of some of its interesting properties among which we recall that $\mik(\R_+)$ is an {\em integral ring}  by the Titchmarsh's Theorem (\cf \cite{GPS}) and hence it has an associated {\em field of fractions} $\mikfield$ called the Mikusinski field.

The following proposition states the existence of a direct relation between the Frechet algebra $B^+_\infty$ (\cf Definition~\ref{Binf+}) and the Mikusinski field $\mikfield$.

For $f\in B^+_{\infty,0}$, $f(z)=
    \int_{0}^\infty e^{-\xi z} d\mu(\xi)$, we let
\begin{equation}\label{chirepmiku}
 \ffm(f)(\xi)=\mu([0,\xi]) \qqq \xi\geq 0.
 \end{equation}
 We follow the notation of \cite{GPS} and denote by $I$ the function $I(\xi)=\xi$ viewed as an element of $\mik(\R_+)$.

\begin{prop} \label{propextend}
$(i)$~The map $\btpluszero\ni f\mapsto \ffm(f)$ defines an isomorphism of $B^+_{\infty,0}$ with a sub-ring of
 $\mik(\R_+)$.\vspace{0.1in}

$(ii)$~There exists a unique function $\iota\in B^+_{\infty,0}$ such that $\ffm(\iota)=I$,  and one has $\iota(z)=\frac 1 z$, $\forall z>0$.\vspace{0.1in}

$(iii)$~The isomorphism $\ffm$, as in $(i)$, extends uniquely to an injective homomorphism of the Frechet algebra $B^+_\infty$ into
the field $\mikfield$.
\end{prop}

\proof
$(i)$~It is easy to check that the product \eqref{convol2} gives the primitive of the convolution product of the derivatives of $F$ and $G$. For $f\in \btpluszero$ the associated measure $\mu$ is of the form $F(0)\delta_0+dF$ where $dF=F'd\xi$ and $F'$ is locally integrable. Thus the convolution of the measures $\mu$ corresponds to the Duhamel product \eqref{convol2}, in terms of $\ffm(f)$. Hence, the map $f\mapsto \ffm(f)$ is an algebra homomorphism and it is injective by construction.

$(ii)$~The Lebesgue measure $d\xi$ fulfills \eqref{convolrho}, and one has
\begin{equation}\label{iotadefn}
\int_{0}^\infty e^{-z\xi }d\xi=\frac 1 z \qqq z>0
\end{equation}
thus $d\xi$  defines an element $\iota\in B^+_{\infty,0}$ such that  $\ffm(\iota)=I$.

$(iii)$~We first prove the following implication
\begin{equation}\label{idealagain}
  f\in B^+_\infty\implies \iota\cdot f\in B^+_{\infty,0}.
\end{equation}
Let $f\in B^+_{\infty}$, $f(z)=
    \int_{0}^\infty e^{-\xi z} d\mu(\xi)$
with $\int_{0}^\infty e^{-\alpha\xi}|d\mu|(\xi)<\infty$  $\forall\alpha>0$.
Let $\psi(u)=\int_0^u d\mu(\xi)$. Then one has
\[
\int_{0}^\infty e^{-uz}\psi(u)du=\int_0^\infty\left(\int_\xi^\infty
e^{-uz}du \right)d\mu(\xi)=\frac 1z \int_{0}^\infty e^{-\xi  z}d\mu(\xi)
\]
where the interchange of integration is justified by Fubini's theorem. It follows that
\[
(\iota\cdot f)(z)=\int_{0}^\infty e^{-u z}\psi(u)du
\]
and, since $\psi(u)$ is locally integrable, $\iota\cdot f\in B^+_{\infty,0}$. Next one defines
\begin{equation}\label{ffmdefn}
\ffm(f)=\frac{\ffm(\iota\cdot f)}{I}\in \mikfield\qqq f\in B^+_\infty.
\end{equation}
Since $\ffm(\iota)=I$ this is the unique extension of $\ffm$ as  a homomorphism from $B^+_\infty$ to $\mikfield$. It is well defined and yields the required injective homomorphism.\endproof

\section{Ideals and spectra of the  algebras $B^+_\infty$ and $B^+_{\infty,0}$}
\label{sec:4thsection}

We denote by $B^+_\C=B^+_\infty\otimes_\R \C$ and  $B^+_{\C,0}=B^+_{\infty,0}\otimes_\R \C$ the complexified algebras of the rings $B^+_\infty$ and $B^+_{\infty,0}$ (\cf Section~\ref{sec:Frechet}). In this section we investigate their ideals and Gelfand spectrum.

\subsection{The principal  ideals $\ker\,\theta_z$} In this section we show that for $z_0\in \C$, with $\Re(z_0)>0$
the kernel of the evaluation map $f\mapsto f(z_0)$ defines a {\em principal} ideal of  the
algebras $B^+_\C$ and $B^+_{\C,0}$.

We begin by stating the following lemma which allows one to divide $f$ by the polynomial $z-z_0$, when $f(z_0)=0$.

\begin{lem}\label{kertheta1}  Any $f\in  B^+_\C$ extends uniquely to an holomorphic function $z\mapsto f(z)$ of $z$, $\Re(z)>0$.

Let $z_0\in \C$ with $\Re(z_0)>0$. There exists a function $\mathfrak k\in B^+_{\C,0}$ such that
\begin{equation}\label{chifstart}
    f- \theta_{z_0}(f)=(z_0-z)\mathfrak k.
\end{equation}

\end{lem}
\proof Since $f\in B^+_\C$, there exists a complex Radon measure $\mu$ on $\R_+$ such that
  \begin{equation}\label{expressf}
f(u)=\int_0^\infty e^{-\xi u}d\mu(\xi)\qqq u>0,\qquad \ \int_0^\infty e^{-\alpha\xi}|d\mu(\xi)|<\infty\qqq \alpha>0.
 \end{equation}
 The integral
  \begin{equation}\label{chif0}
    \theta_{z_0}(f)=\int_0^\infty e^{-\xi z_0}d\mu(\xi)
\end{equation}
is finite and bounded in absolute value by the norm $\|f\|_\rho$, for $\rho=e^{-1/\Re(z_0)}$. Let $z\neq z_0$, then one has
\begin{equation}\label{expequbis}
 \frac{e^{-z\xi}-e^{-z_0\xi}}{z_0-z}= \int_0^\xi e^{-(z-z_0)u-z_0\xi}du
 \end{equation}
 and when $z\to z_0$ both sides of the above equality converge to the function $\xi e^{-z_0\xi}$.
The equality
 \begin{equation}\label{psiubis}
 \psi(u)=\int_u^\infty e^{z_0(u-\xi)}\, d\mu(\xi)\qqq u\in \R_+
 \end{equation}
 defines a complex valued function whose size is controlled by
 \begin{equation}\label{sizepsi}
 |\psi(u)|\leq\int_u^\infty e^{\Re(z_0)(u-\xi)}\, |d\mu(\xi)|.
 \end{equation}
 Next, we show that $\int_0^\infty e^{-\alpha u}|\psi(u)|du<\infty$, for $\alpha>0$.
 When $\alpha>0$ and $\alpha< \Re(z_0)$, by implementing \eqref{expequbis} (for $z=\alpha$ and with $\Re(z_0)$ instead of $z_0$) and Fubini's theorem to interchange the integrals, one has
 \begin{align*}
 \int_0^\infty e^{-\alpha u}|\psi(u)|du&\leq
 \int_0^\infty\int_u^\infty e^{-\alpha u}e^{\Re(z_0)(u-\xi)}|d\mu(\xi)|du=\\
 &=(\Re(z_0)-\alpha)^{-1}\int_0^\infty\left(e^{-\alpha\xi}-e^{-\Re(z_0)\xi}\right)|d\mu(\xi)|.
 \end{align*}
  This proves that the formula
 \begin{equation}\label{chik}
    \mathfrak k(z)=\int_0^\infty e^{-u z}\psi(u)du
 \end{equation}
 defines an element $\mathfrak k\in B^+_{\C,0}$ whose norm satisfies, for $\rho_0=e^{-1/\Re(z_0)}$
 \[
 \|\mathfrak k\|_\rho\leq  \int_0^\infty\int_u^\infty e^{u/\log(\rho)}e^{\Re(z_0)(u-\xi)}|d\mu(\xi)|du
 =(\|f\|_\rho-\|f\|_{\rho_0})/(1/\log(\rho)-1/\log(\rho_0)).
 \]
 Moreover, using again \eqref{expequbis} with $z \neq z_0$, one obtains
  \[
 \frac{f(z)-f(z_0)}{z_0-z}=\int_0^\infty \frac{e^{-z\xi}-e^{-z_0\xi}}{z_0-z}d\mu(\xi)=
 \int_0^\infty \int_0^\xi e^{-(z-z_0)u-z_0\xi}du\, d\mu(\xi)
 \]
 which gives
 \begin{equation}\label{chik1}
    \mathfrak k(z)=\frac{f(z)-f(z_0)}{z_0-z}
 \end{equation}
 and the equality \eqref{chifstart} follows.  \endproof

\begin{prop}\label{kertheta}
$(i)$~Let $z_0\in \C$ with $\Re(z_0)>0$. Then
\[
\theta_{z_0}: B^+_\C\to \C,\quad \theta_{z_0}(f)=f(z_0)
\]
defines a complex character  of the algebra $B^+_\C$. One has $\theta_1=\theta$.\vspace{0.1in}

$(ii)$~The ideal $\ker(\theta_{z_0})\subset B^+_\C$ is  generated by the function $\iota-z_0^{-1}$, with
\begin{equation}\label{defniota}
   \iota(z)=\int_0^\infty e^{-\xi z}d\xi=z^{-1}\qqq z>0.
\end{equation}

$(iii)$~$\iota-z_0^{-1}\in B^+_{\C,0}$ and $\iota-z_0^{-1}$ generates the ideal
$\ker(\theta_{z_0})\cap B^+_{\C,0}\subset B^+_{\C,0}$.
\end{prop}
\proof $(i)$~follows from the first statement of Lemma~\ref{kertheta1}.

$(ii)$~ Since one knows that
 \begin{equation*}
 \int_0^\infty e^{-\xi z}d\xi=\frac 1z,\qquad \int_0^\infty e^{\xi/\log(\rho)}d\xi<\infty \qqq\rho\in (0,1)
\end{equation*}
one derives that $\iota\in B^+_{\C,0}$.
Let $f\in  B^+_\C$, then by applying Lemma~\ref{kertheta1}, one sees that there exists a function $\mathfrak k\in B^+_{\C,0}$ such that \eqref{chifstart} holds. One then obtains
\begin{equation}\label{chif1}
f(z)=f(z_0)+(\frac 1z-\frac{1}{z_0})h(z)
\,, \ \ h=-z_0(f- f(z_0))+z_0^2 \mathfrak k
\end{equation}
since
\[
\frac{1}{z^{-1}-z^{-1}_0}=-z_0+\frac{1}{z_0-z}z_0^{2}.
\]

$(iii)$~By assuming that $f\in \ker(\theta_{z_0})$ we obtain the factorization $f=(\iota-z_0^{-1}) h$.
$(iii)$ then follows since $\iota\in B^+_{\C,0}$.  \endproof

\begin{lem}\label{kertheta2} Let $\alpha:B_\C^+\to \C$ be a ring homomorphism. Then, if $T_0=\alpha(\iota)\neq 0$ one has $\Re(T_0)>0$ and
\begin{equation}\label{to}
    \alpha=\theta_{z_0}, \ \ z_0=1/T_0.
\end{equation}
Similarly, the maps $\theta_{z_0}: B^+_{\C,0}\to \C$, with $\Re(z_0)>0$, define all the characters $\alpha$ of $B^+_{\C,0}$ with $\alpha(\iota)\neq 0$.
\end{lem}
\proof For any $\lambda\in \C$ with $\Re(\lambda)\geq 0$, $\lambda\neq 0$, one has
\[
 \frac{\varh}{\lambda \varh+1}=\int_0^\infty e^{-\xi/\varh}e^{-\lambda\xi}d\xi.
 \]
Thus there exists a function  $h_\lambda\in B^+_{\C,0}$ such that $(\lambda \iota+1) h_\lambda =\iota$. Then one obtains
\[
\alpha(\lambda \iota+1)\alpha(h_\lambda)=\alpha(\iota)\neq 0
\]
and $\lambda\alpha(\iota)+1\neq 0$ which shows that $\Re(\alpha(\iota))>0$. It follows from Lemma \ref{kertheta1} that the map $\theta_{z_0}$, $ z_0=1/T_0$ defines a character of $B_\C^+$. Moreover for any $f\in B_\C^+$ there exists $h\in B^+_\C$ such that \eqref{chif1} holds \ie
\[
    f=f(z_0)+(\iota-T_0)h.
\]
One then obtains $\alpha(f)=f(z_0)$ since $\alpha(\iota-T_0)=0$. \endproof

\subsection{Gelfand spectrum of $B^+_{\C,0}$}

We are now ready to compute the Gelfand spectrum of the Frechet algebra $B^+_{\C,0}$.

\begin{thm}\label{kertheta3} The Gelfand spectrum $\Spec(B^+_{\C,0})$ is the one point compactification $Y=\C^+\cup\{\infty\}$  of the open half-plane $\C^+=\{z\in \C\mid \Re(z)>0\}$. The one parameter group $\arf_\lambda$ acts on $\C^+$ by scaling  $z\to \lambda z$ and it fixes $\infty\in Y$.
\end{thm}
\proof Let $\alpha\in\Spec\, B^+_{\C,0}$ be a continuous homomorphism $\alpha: B^+_{\C,0}\to \C$. Let us first assume that $T_0=\alpha(\iota)\neq 0$. Then by Lemma~\ref{kertheta2} one has $\alpha=\theta_{z_0}$, $z_0=1/T_0$.

Assume now that $\alpha(\iota)=0$. Then, for any smooth function $k(\xi)$ with compact support one has
$$
\int_0^\infty e^{-\xi z}k'(\xi)d\xi=k(0)+z \int_0^\infty e^{-\xi z}k(\xi)d\xi
$$
This shows that if $k(0)=0$ the associated element of $B^+_{\C,0}$ belongs to the ideal generated by $\iota$. Thus this ideal is dense (for the norm $\|f\|_\rho$, \cf \S\ref{rn}) in the kernel of the character
\begin{equation}\label{charF}
  \theta_\infty: B^+_{\C,0}\to\C,\quad  \theta_\infty(f)=\lim_{z\to \infty}f(z).
\end{equation}
Thus by continuity we get $\alpha=\theta_\infty$, if $\alpha(\iota)=0$. This shows that $\Spec(B^+_{\C,0})$ is the space of characters $\theta_z: B^+_{\C,0}\to\C$, for $z\in Y=\C^+\cup\{\infty\}$. The action of $\arf_\lambda$ is such that
\begin{equation}\label{chialpha}
    \theta_z(\arf_\lambda(f))=\theta_{\lambda z}(f)\qqq f\in B^+_{\C,0}, \ \lambda\in \R^\times_+, \
    z\in Y.
\end{equation}
\endproof

\begin{cor} \label{Ymax} The map $Y\ni z\mapsto \Ker(\theta_z)\subset B^+_{\C,0}$ defines a bijection of $Y$ with the space of maximal closed ideals of the Frechet algebra $B^+_{\C,0}$.
\end{cor}
\proof This follows from the generalized Gelfand-Mazur theorem which shows that for any closed maximal ideal $J\subset B^+_{\C,0}$ there exists a continuous character $B^+_{\C,0}\to\C$ whose kernel is  $J$. \endproof

\section{The complex case and oscillatory integrals}\label{complex}

In the real case it was simple to evaluate the asymptotic behavior of integrals of real exponentials as in Proposition \ref{augmentation}. On the other hand, in the complex case we shall see that oscillatory integrals with several critical points  provide typical examples of application of the (multi-valued) law of addition in hyperfields. Rather than developing the general case we focus  on a well-known  example of asymptotic behavior of integrals of imaginary exponentials, namely  
 the case of the Airy function (see \cite{Stokes,Berry,Dingle}).  This function is defined by the formula
\begin{equation}\label{airy1}
   \ai(x)=\frac{1}{2\pi}\int_{-\infty}^\infty e^{i\left(\frac{s^3}{3}+xs\right)}ds.
\end{equation}
The integral makes sense in the complex domain along a path slightly above the real axis, \ie of the form $C=[-\infty+i\epsilon ,\infty+i\epsilon]$ with $\epsilon >0$.
This function fulfills the differential equation
\begin{equation}\label{diffequ}
    y''-zy=0
\end{equation}
and it is entire and given by the series
\begin{equation}\label{series}
\ai(z)=3^{-2/3}\sum_{n=0}^\infty\frac{z^{3n}}{9^n n!\Gamma(n+2/3)}
-3^{-4/3}\sum_{n=0}^\infty\frac{z^{3n+1}}{9^n n!\Gamma(n+4/3)}
\end{equation}
$$    =\frac{1}{3^{2/3} \Gamma(\frac{2}{3})}-\frac{z}{3^{1/3} \Gamma(\frac{1}{3})}+\frac{z^3}{6\times 3^{2/3} \Gamma(\frac{2}{3})}-\frac{z^4}{12 \left(3^{1/3} \Gamma(\frac{1}{3})\right)}+\ldots
$$

\begin{figure}
\begin{center}
\includegraphics[scale=0.9]{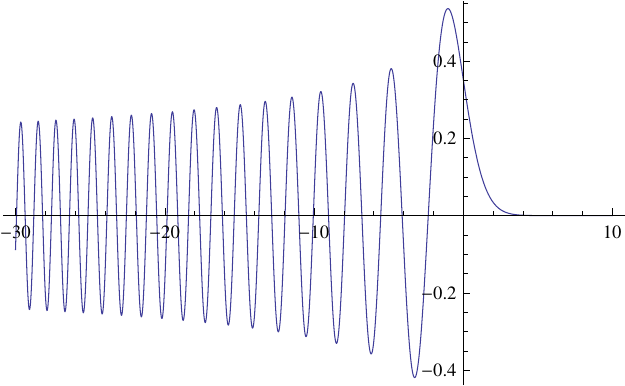}
\caption{Graph of the Airy function}
\end{center}
\end{figure}

We first consider the asymptotic expansion of $\ai(x)$ at infinity, on the positive real axis:
\begin{equation}\label{expandairy2}
\ai(z)\sim \frac{1}{4\pi^{3/2}}z^{-\frac 14} e^{-\frac 23 z^{3/2}}\sum\frac{\Gamma(n+\frac 56)
\Gamma(n+\frac 16)}{n!}(-3/4)^n z^{-\frac{3n}{2}}\,.
\end{equation}
The series on the right hand side is not convergent and the {\bf strong} meaning of the expansion is that the {\em ratio} of the left hand side by the truncated right hand side is ``under control'' \ie it is of the form $1+O(z^{-m})$, with $m$ depending on the truncation. For instance, the ratio of
$\ai(\frac 1T)$ and the approximation
$$
h_5(T)=e^{-\frac{2}{3} \left(\frac{1}{T}\right)^{3/2}} \left(\frac{T^{1/4}}{2 \sqrt{\pi }}-\frac{5 T^{7/4}}{96 \sqrt{\pi }}+\frac{385 T^{13/4}}{9216 \sqrt{\pi }}-\frac{85085 T^{19/4}}{1327104 \sqrt{\pi }}+\frac{37182145 T^{25/4}}{254803968 \sqrt{\pi }}\right)$$
is of the form $1+O(T^{15/2})$ since the next term in the expansion is
$$
-\frac{5391411025 \ T^{31/4}}{12230590464 \sqrt{\pi }}\sim -0.248702\ T^{31/4}
$$
while the first term is $\frac{T^{1/4}}{2 \sqrt{\pi }}$. The parameter $T$ in this approximation is real and positive and one lets $T\to 0+$. For each $\alpha>0$ we have a natural subgroup $G_\alpha$ of the multiplicative group of non-zero functions defined by the condition
\begin{equation}\label{subgroup}
G_\alpha=\{h\mid h(T)=1+O(T^\alpha)\ \ \text{for}\ T\to 0+\}.
\end{equation}
With this notation we can rewrite the above equivalence of the functions $\ai(\frac 1T)$ and $h_5(T)$ as
$$
\ai(\frac 1T)/h_5(T)\in G_\alpha \,, \ \alpha =\frac{15}{2}.
$$
\begin{figure}
\begin{center}
\includegraphics[scale=0.7]{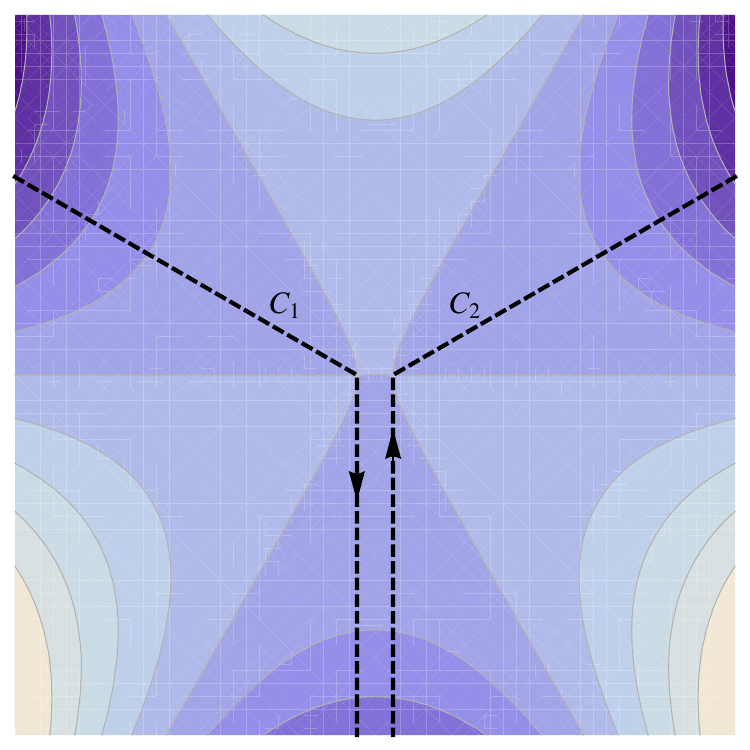}
\caption{For $z$ real negative one deforms the path $C=[-\infty,\infty]$ to the disjoint union of $C_1$ and $C_2$. The levels are those of the real part of the function $\Phi(s,z)$
where  $\Phi(s,z)=i\left(\frac{s^3}{3}+zs\right)$.}\label{fig1}
\end{center}
\end{figure}
It is then natural to ask what kind of algebraic object one obtains if one considers the quotient of a field $K$ of functions by the above equivalence relation (for fixed value of $\alpha$). By construction $G_\alpha$ is a subgroup of the multiplicative group $K^\times$ and thus the quotient $K/G_\alpha$ is a {\em hyperfield}. This implies in particular that having strong expansions for two functions does not uniquely determine a strong asymptotic expansion for their sum. We illustrate this conclusion by considering the expansion of the Airy function on
 the negative real axis. There, the function admits zeros and the expansion is more involved and usually written in the form
\begin{small}
\begin{equation*}
\ai(x)\sim  \frac{1}{2\pi^{3/2}}(-x)^{-1/4} \left(\cos\left(\frac{\pi }{4}+\frac{2 x\sqrt{-x} }{3}\right) \sum_{n\,{\rm even}}  \frac{\Gamma(n+\frac 56)
\Gamma(n+\frac 16)}{n!}(3/4)^n x^{-\frac{3n}{2}}\right. \end{equation*}
\begin{equation}\label{expandairy4}
   \left. -\sin\left(\frac{\pi }{4}+\frac{2 x\sqrt{-x} }{3}\right)\sum_{n\,{\rm odd}}   \frac{\Gamma(n+\frac 56)
\Gamma(n+\frac 16)}{n!}(3/4)^n(-1)^{(n-1)/2} (-x)^{-\frac{3n}{2}}\right).
\end{equation}
\end{small}
In this case we cannot expect that the ratio of the left hand side with a truncation of the right hand side belongs to $G_\alpha$ for some $\alpha>0$ (after changing variables to $x=-\frac 1T$) since the equivalence relation
preserves the zeros except for finitely many (since for $\alpha>0$ and $h\in G_\alpha$ one has $h(T)=0$ for only finitely many $T>0$ in a neighborhood of $T=0$). In fact, what the above asymptotic expansion suggests is that one can decompose the function $\ai(-\frac 1T)$ as a sum of two functions which  are equivalent, in the above strong sense, respectively to (with $x=-\frac 1T$)
$$
\frac{1}{2\pi^{3/2}}(-x)^{-1/4} \cos\left(\frac{\pi }{4}+\frac{2 x\sqrt{-x} }{3}\right) \sum_{n\,{\rm even}}  \frac{\Gamma(n+\frac 56)
\Gamma(n+\frac 16)}{n!}(3/4)^n x^{-\frac{3n}{2}}
$$
and
\begin{small}
$$
-\frac{1}{2\pi^{3/2}}(-x)^{-1/4}
\sin\left(\frac{\pi }{4}+\frac{2 x\sqrt{-x} }{3}\right)\sum_{n\,{\rm odd}}   \frac{\Gamma(n+\frac 56)
\Gamma(n+\frac 16)}{n!}(3/4)^n(-1)^{(n-1)/2} (-x)^{-\frac{3n}{2}}.
$$
\end{small}

To obtain the required decomposition of the function $\ai(x)$ one uses its definition as an oscillatory integral
\eqref{airy1}, \ie as an integral along a path slightly above the real axis, \ie of the form $C=[-\infty+i\epsilon ,\infty+i\epsilon]$ with $\epsilon >0$.
In order to obtain the decomposition  for $x=-\frac 1T$ real and negative, one deforms the path of integration $C$ in the complex domain to the disjoint union of two paths $C_1$ and $C_2$ as shown in Figure \ref{fig1}.
The integrals over the paths $C_j$ give complex conjugate numbers and this splitting as a sum $\int_C=\int_{C_1}+\int_{C_2}$ gives $\ai(x)=2\Re(\int_{C_2})$. In fact, the imaginary part $2\Im(\int_{C_2})$ gives the other Airy function $\bi(x)$ which is known to be a solution of the second order linear differential equation $y''-xy=0$. This function can be defined directly as the following oscillatory integral
\begin{equation}\label{airy2}
   \bi(x)=\frac{1}{\pi}\int_{0}^\infty \left(e^{-\left(\frac{s^3}{3}-xs\right)}+\sin\left(\frac{s^3}{3}+xs\right)\right)ds
\end{equation}
described by the two pieces of a path $C'_2$ going through the lower half of the imaginary axis and the right half of the real axis. $\bi(x)$ is characterized, among the solutions of $y''-xy=0$, by
$$
\bi(0)=\frac{1}{3^{1/6} \Gamma(\frac{2}{3})}\, , \ \
\bi'(0)=\frac{3^{1/6}}{\Gamma(\frac{1}{3})}\, .
$$
To obtain the required decomposition of $\ai(x)$ one uses the stationary phase method to evaluate $\int_{C_2}$ where $C_2$ goes through the critical point $\sqrt{-x}$ with an angle of $\pi/4$ with respect to the real axis, so that it follows the line of steepest descent. This shows that the argument of the complex number $\int_{C_2}=\frac 12(\ai(x)+i\bi(x))$ is close to $\alpha(x)=\frac{\pi }{4}+\frac{2 x\sqrt{-x} }{3}$ and thus one introduces the rotation matrix
$$
R(x)=\left[
\begin{array}{cc}
 \cos\,\alpha(x) & \sin\,\alpha(x) \\
 -\sin\,\alpha(x) & \cos\,\alpha(x)
\end{array}
\right]\,, \ \ \alpha(x)=\frac{\pi }{4}+\frac{2 x\sqrt{-x} }{3}
$$
which one applies to the column vector $\xi(x)$ with entries $(\ai(x),\bi(x))$. By using the inverse rotation matrix it follows that
$$
\ai(x)=\cos\,\alpha(x)\left(R(x)\xi(x)\right)_1-\sin\,\alpha(x)\left(R(x)\xi(x)\right)_2
=\ai_0(x)+\ai_1(x)
$$
It is exactly the decomposition of $\ai(x)$ as a sum of two terms $\ai_j(x)$ which gives the precise meaning to the asymptotic expansion. Indeed, the stationary phase method shows that
\begin{equation}\label{stat1}
2e^{-i\alpha(x)}\int_{C_2}\sim \frac{(-x)^{-1/4}}{2\pi^{3/2}}
 \sum_{n=0}^\infty   \frac{\Gamma(n+\frac 56)
\Gamma(n+\frac 16)}{n!}(3/4)^n x^{-\frac{3n}{2}}.
\end{equation}
 Thus since $e^{-i\alpha(x)}(\ai(x)+i\bi(x))=2e^{-i\alpha(x)}\int_{C_2}$, this shows that   one has
\begin{equation}\label{ai1}
\ai_0(x)\sim \frac{(-x)^{-1/4}}{2\pi^{3/2}} \cos\,\alpha(x)\sum_{n\,{\rm even}}  \frac{\Gamma(n+\frac 56)
\Gamma(n+\frac 16)}{n!}(3/4)^n x^{-\frac{3n}{2}}.
\end{equation}
It follows that for any $m>0$ the ratio of the left hand side with the right hand side truncated at $n\leq m$ is of the form $1+O(x^{-\frac{3m}{2}})$ for $x<0$, $x\to -\infty$ as above.
Similarly one shows that \begin{small}
\begin{equation}\label{ai2}
\ai_1(x)\sim -\frac{(-x)^{-1/4}}{2\pi^{3/2}}
\sin\,\alpha(x)\sum_{n\,{\rm odd}}   \frac{\Gamma(n+\frac 56)
\Gamma(n+\frac 16)}{n!}(3/4)^n(-1)^{(n-1)/2} (-x)^{-\frac{3n}{2}}
\end{equation}
\end{small}
in the above strong sense. Notice that the two equivalences \eqref{ai1} and \eqref{ai2} are stronger than the original one \eqref{expandairy4} for $\ai(x)$. In particular they determine {\em exactly} the positions of the zeros of $\ai_j(x)$ for $x<<0$ as the $x_n=-\frac{1}{4} 3^{2/3} (-\pi +4 n \pi )^{2/3}$for $\ai_0$ and $y_n=-\frac{1}{4} 3^{2/3} (\pi +4 n \pi )^{2/3}$for $\ai_1$. On the other hand the zeros of the Airy function are not given by an elementary formula. Moreover even the overall sizes of the two terms $\ai_j(x)$ are not the same since while $\ai_0(x)$ is of the order of $(-x)^{-1/4}$ the function $\ai_1(x)$ is of the order of $(-x)^{-7/4}$.

\begin{figure}
\begin{center}
\includegraphics[scale=0.8]{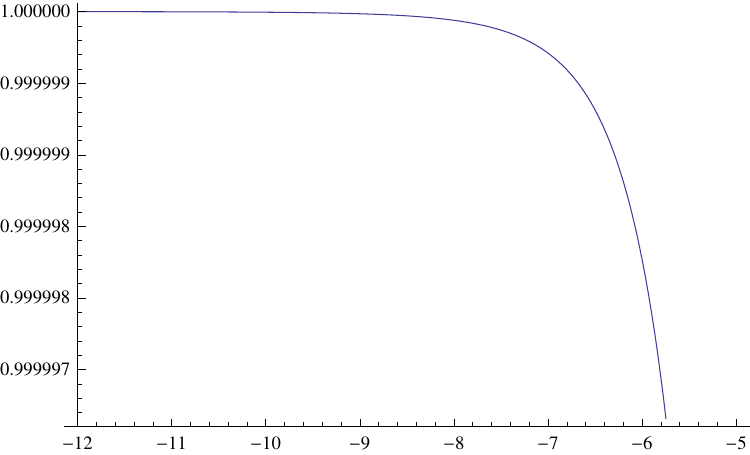}
\caption{The ratio of $\ai_1(x)$ with its approximation using the first $4$ terms of the asymptotic series}
\end{center}
\end{figure}

\subsection{Strong asymptotic expansion of $\int_{C_2}$}

We now work out the details of the stationary phase method, first for the asymptotic expansion of $\int_{C_2}=\frac 12(\ai(x)+i\bi(x))$ (when $x<0$). We  perform a change of variables in
\begin{equation}\label{airy1rep}
   \ai(x)=\frac{1}{2\pi}\int_{-\infty}^\infty e^{i\left(\frac{s^3}{3}+xs\right)}ds
\end{equation}
and let $x=-u^{\frac 23}$ with $u>0$ and $s=u^{\frac 13}t$. We then get
\begin{equation}\label{airy1rep1}
   \ai(x)=\frac{u^{\frac 13}}{2\pi}\int_{-\infty}^\infty e^{iu\left(\frac{t^3}{3}-t\right)}dt
\end{equation}
and we are looking for the expansion when $u\to +\infty$. After the above change of variables the two critical points correspond now to $t=\pm 1$. For the path $C_2$ we take the path, in the complex domain, through the critical point $t=1$ and such that the real part of $\frac{t^3}{3}-t$ remains constant (equal to $-\frac 23$) along the path. In this way the variation of the phase will only come from the term $dt$.
With $t=\xi +i\eta$ one has
\begin{equation*}
    \Re\left(\frac{t^3}{3}-t\right)=\frac{\xi^3}{3}-\xi\eta^2-\xi
\end{equation*}
and we take for $C_2$ the branch of the curve
\begin{equation*}
   \frac{\xi^3}{3}-\xi\eta^2-\xi+\frac 23=0
\end{equation*}
which is given by the formula, valid for $\xi>0$,
\begin{equation*}
    \eta=\frac{(-1+\xi ) \sqrt{2+\xi }}{\sqrt{3} \sqrt{\xi }}.
\end{equation*}

\begin{figure}
\begin{center}
\includegraphics[scale=0.5]{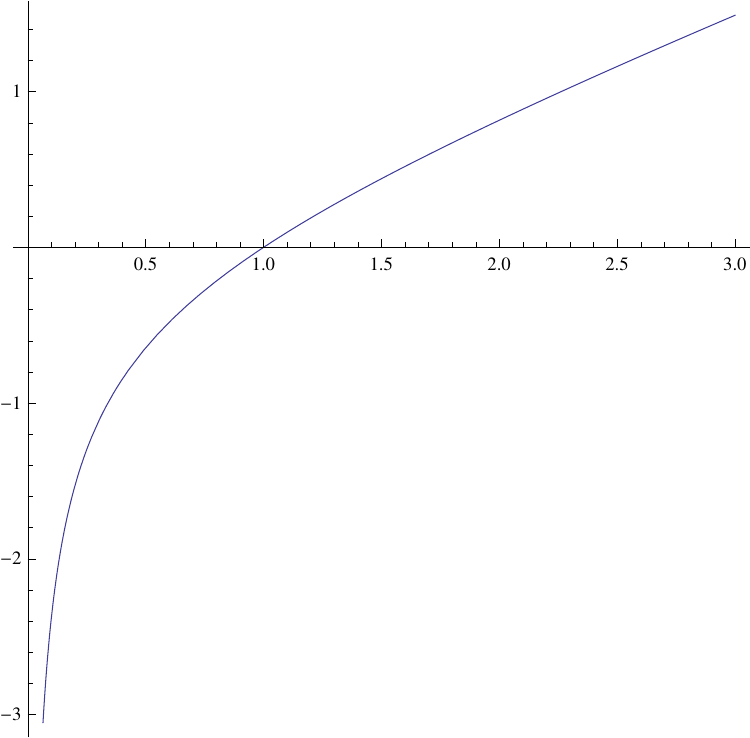}
\caption{The path $C_2$}
\end{center}
\end{figure}
Along the path $C_2$ one has the equality
\begin{equation}\label{constphase}
 i\left(\frac{t^3}{3}-t\right)=-\frac 23 i-\frac 12 w^2, \ \   w(\xi)=\frac{2 (-1+\xi ) (2+\xi )^{1/4} (1+2 \xi )}{3^{5/4} \xi ^{3/4}}
\end{equation}
and $w(\xi)$ varies from $-\infty$ (for $\xi=0$) to $+\infty$ for $\xi\to \infty$. One needs to compute $dt$ and one has
\begin{equation}\label{detadxi}
    d\eta/d\xi=\frac{1+\xi +\xi ^2}{\sqrt{3} \xi ^{3/2} \sqrt{2+\xi }}
\end{equation}
and
\begin{equation}\label{dwdxi}
    dw/d\xi=\frac{1+3 \xi ^2+2 \xi ^3}{3^{1/4} \xi ^{7/4} (2+\xi )^{3/4}}.
\end{equation}
For $\xi\to 0$ one has
\begin{equation}\label{equiv}
w\sim -(\frac 23)^{\frac 54}\xi^{-\frac 34}, \ \
d\eta/d\xi\sim 6^{-\frac 12}\xi^{-\frac 32}, \  \
dw/d\xi\sim 3^{-\frac 14} 2^{-\frac 34}\xi^{-\frac 74}
\end{equation}
so that
\begin{equation}\label{equiv0}
    d\xi/dw\sim c_1 |w|^{-\frac 74}, \ \ d\eta/dw\sim c_2 |w|^{\frac 14}
\end{equation}
For $\xi \to +\infty$ one has
\begin{equation}\label{equivinfty}
w\sim 4\times 3^{-\frac 54}\xi^{\frac 32}, \ \
d\eta/d\xi\sim 3^{-\frac 12}, \  \
dw/d\xi\sim 2\times 3^{-\frac 14} \xi^{\frac 12}
\end{equation}
so that
\begin{equation}\label{equiv2}
    d\xi/dw\sim c_3 |w|^{-\frac 13}, \ \ d\eta/dw\sim c_4 |w|^{-\frac 13}.
\end{equation}
We can now justify the asymptotic expansion of $\int_{C_2}$.
\begin{lem}\label{asexpc2} When $u\to +\infty$ one has
\begin{equation}\label{asexp}
 e^{-i( \frac \pi 4 -\frac 23 u)}   \int_{C_2}e^{iu\left(\frac{t^3}{3}-t\right)}dt\sim
  \frac{u^{-1/2}}{2\sqrt{\pi}}
 \sum_{n=0}^\infty  \frac{\Gamma(n+\frac 56)
\Gamma(n+\frac 16)}{n!}(\frac{3\, i}{4})^n u^{-n}
\end{equation}
\end{lem}
\proof Using \eqref{constphase} we have
\begin{equation}\label{gfunct0}
\int_{C_2}e^{iu\left(\frac{t^3}{3}-t\right)}dt= e^{-i\frac 23 u}   \int_{-\infty}^\infty e^{-u w^2/2}g(w)dw
\end{equation}
where the function $g(w)$ is given by
\begin{equation}\label{gfunct}
    g(w)=g_1(w)+ig_2(w)=d\xi/dw+i d\eta/dw
\end{equation}
where one expresses $\xi$ as a function $\xi(w)$ of $w\in \R$. The two functions $g_j(w)$ are smooth and $O(|w|^\ell)$ when $|w|\to \infty$ by \eqref{equiv0} and \eqref{equiv2}.
The Taylor expansion of $g(w)$ at $w=0$ is of the form
\begin{equation*}
   g(w)= \left(\frac{1}{2}+\frac{i}{2}\right)-\frac{i w}{6}-\left(\frac{5}{96}-\frac{5 i}{96}\right) w^2+\frac{w^3}{27}-\left(\frac{385}{27648}+\frac{385 i}{27648}\right) w^4+\frac{7 i w^5}{648}+\ldots
\end{equation*}
and the even part $\frac 12(g(w)+g(-w))$ takes the simpler form
\begin{equation*}
 \frac 12(g(w)+g(-w))=\frac{1+i}{2}\left( 1+\frac{5 i w^2}{48}-\frac{385 w^4}{13824}-\frac{17017 i w^6}{1990656}+\frac{1062347 w^8}{382205952}+\ldots\right)
\end{equation*}
In fact one has an equality of the form
\begin{equation}\label{gevendec}
    \frac 12(g(w)+g(-w))=e^{i\frac \pi 4}(h_0(w)+i h_1(w))
\end{equation}
where $h_0(w)$ is the real part of $e^{-i\frac \pi 4}\frac 12(g(w)+g(-w))$. By construction both $h_j$ are smooth even functions and $O(|w|^\ell)$ when $|w|\to \infty$. Moreover $h_j(w)=k_j(w^2)$ where again both $k_j$ are smooth, $k_0$ is even and $k_1$ is odd.

 Let
\begin{equation}\label{fjfunct}
    f_j(u):=\int_{-\infty}^\infty e^{-u w^2/2}h_j(w)dw
    =\int_0^\infty e^{-u v/2}k_j(v)\frac{dv}{\sqrt v}.
\end{equation}
The asymptotic expansion of $f_j(u)$ for $u\to \infty$ follows directly from the Taylor expansion of  $h_j(w)$ at $w=0$ (or of $k_j(v)$ at $v=0$, using \eg Watson's Lemma). It is given by the well known explicit formulas
\begin{equation}\label{f0}
    f_0(u)\sim
    \frac{u^{-1/2}}{2\sqrt{\pi}}
 \sum_{n\, \rm even} (-1)^{\frac n2}  \frac{\Gamma(n+\frac 56)
\Gamma(n+\frac 16)}{n!}(\frac{3\, }{4})^n u^{-n}
\end{equation}
and
\begin{equation}\label{f1}
    f_1(u)\sim
    \frac{u^{-1/2}}{2\sqrt{\pi}}
 \sum_{n\, \rm odd} (-1)^{\frac{n-1}{2}}  \frac{\Gamma(n+\frac 56)
\Gamma(n+\frac 16)}{n!}(\frac{3\, }{4})^n u^{-n}.
\end{equation}
Thus since by \eqref{gfunct0},
\begin{equation}\label{decompo}
  \int_{C_2}e^{iu\left(\frac{t^3}{3}-t\right)}dt= e^{i(\frac \pi 4-\frac 23 u)}(f_0(u)+if_1(u))
\end{equation}
one obtains \eqref{asexp}.\endproof

\begin{figure}
\begin{center}
\includegraphics[scale=0.7]{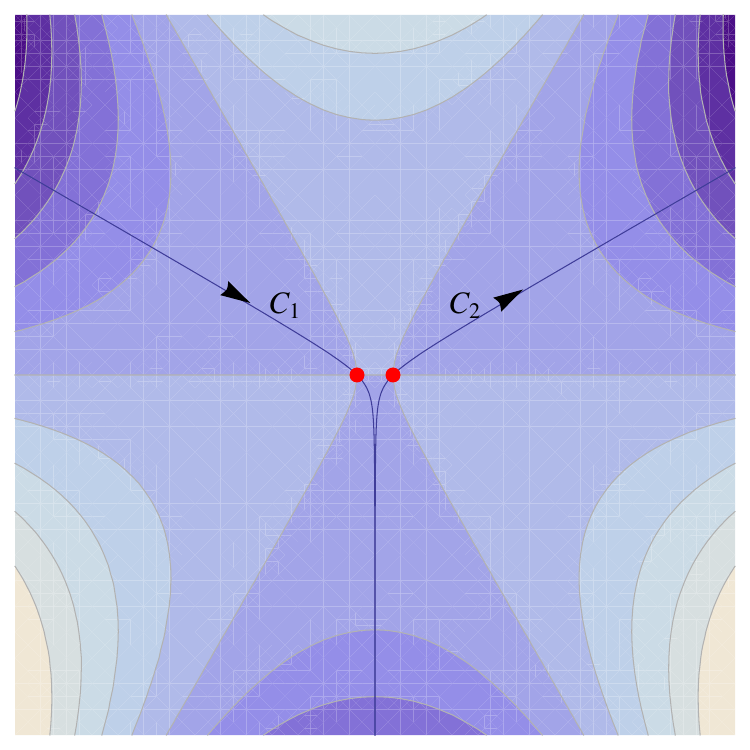}
\caption{The paths $C_j$}\label{fig4}
\end{center}
\end{figure}

The above asymptotic expansions hold in the classical sense as defined by Poincar\'e. We shall now see that when one passes to the real part an interesting phenomenon occurs. We first need to define more precisely the notion of {\em strong asymptotic expansion}.

Let us consider, for $\alpha>0$ the following  multiplicative subset of   functions of the variable $u$.
\begin{equation}\label{subgroupz}
G_\alpha=\{h\mid h(u)=1+O(u^{-\alpha})\ \ \text{for}\ u\to +\infty\}.
\end{equation}

\begin{defn}\label{defnstrong} Let $f(u)$, $t_n(u)$ be functions of the positive real variable $u$. The expansion $f(u)\sim \sum_1^\infty t_n(u)$ is called a {\em strong asymptotic expansion} when for any $\alpha>0$ there exists $n_\alpha$ such that
\begin{equation}\label{galp}
    f\in \left(\sum_1^n t_k\right)G_\alpha\qqq n\geq n_\alpha.
\end{equation}
\end{defn}

\begin{prop} \label{prophyperdec} For $x\in \R$, $x<0$, there exists a decomposition
\begin{equation}\label{hyperdec1}
    \ai(x)=\ai_0(x)+\ai_1(x)
\end{equation}
as a sum of two real analytic functions of $x$ with strong asymptotic expansions
\begin{equation}\label{ai1u}
   \ai_0(-u^{\frac 23})\sim  \frac{u^{-1/6}}{2\pi^{3/2}} \cos(\frac{\pi }{4}-\frac 23 u )\sum_{n\,{\rm even}}(-1)^{n/2}  \frac{\Gamma(n+\frac 56)
\Gamma(n+\frac 16)}{n!}(3/4)^n u^{-n}
\end{equation}
\begin{equation}\label{ai2u}
\ai_1(-u^{\frac 23})\sim -\frac{u^{-1/6}}{2\pi^{3/2}}
\sin(\frac{\pi }{4}-\frac 23 u )\sum_{n\,{\rm odd}}(-1)^{\frac{(n-1)}{2}}   \frac{\Gamma(n+\frac 56)
\Gamma(n+\frac 16)}{n!}(3/4)^n u^{-n}.
\end{equation}
\end{prop}
\proof One has, using \eqref{airy1rep1}, and deforming the path $(-\infty, \infty)+i\epsilon$ into the union of two paths $C_j$ as in Figure \ref{fig4}
\begin{equation}\label{dec1}
    \ai(-u^{\frac 23})=\frac{u^{\frac 13}}{2\pi}\sum_j\int_{C_j} e^{iu\left(\frac{t^3}{3}-t\right)}dt.
\end{equation}
The symmetry $s(t)=-\bar t$ transforms $C_2$ into $C_1$ but reverses the natural orientation. One has
\begin{equation*}
    iu\left(\frac{s(t)^3}{3}-s(t)\right)=\overline{\left(iu\left(\frac{t^3}{3}-t\right)\right)}
\end{equation*}
and thus the terms $\int_{C_j}$ in \eqref{dec1} are complex conjugate. With the notations of \eqref{fjfunct} we define
\begin{equation}\label{ai0def}
    \ai_0(x):=\frac{u^{\frac 13}}{\pi} \cos(\frac{\pi }{4}-\frac 23 u )f_0(u)\qqq x=-u^{\frac 23}
\end{equation}
and
\begin{equation}\label{ai1def}
    \ai_1(x):=\frac{u^{\frac 13}}{\pi} \sin(\frac{\pi }{4}-\frac 23 u )f_1(u)\qqq x=-u^{\frac 23}.
\end{equation}
By \eqref{dec1} one has
\begin{equation}\label{decomp1}
   \ai(-u^{\frac 23})=\frac{u^{\frac 13}}{\pi}\Re\left(\int_{C_2} e^{iu\left(\frac{t^3}{3}-t\right)}dt\right)
\end{equation}
and taking the real part of both sides of \eqref{decompo} one gets the decomposition \eqref{hyperdec1}. Using \eqref{f0} and \eqref{f1} one obtains the strong asymptotic expansions \eqref{ai1u} and \eqref{ai2u}. \endproof

\begin{prop} \label{propcomplexbb}
There exists an element $h$ of the algebra $B^+_{\C,0}=B^+_{\infty,0}\otimes_\R \C$ such that, for any $u>0$ one has
\begin{equation}\label{hyperdec2}
    \ai(-u^{\frac 23})=  u^{\frac 13}\Re\left(e^{i(\frac{\pi }{4}-\frac 23 u )}h(u)\right).
\end{equation}
\end{prop}
\proof Let $h(u)=\frac 1 \pi(f_0(u)+if_1(u))$, then by \eqref{decomp1} and \eqref{decompo} one has \eqref{hyperdec2}. It remains to show that each $f_j$ belongs to $B^+_{\infty,0}$.
By \eqref{fjfunct} one has
\begin{equation}\label{fjfunctbis}
    f_j(u)=\int_0^\infty e^{-u v/2}k_j(v)\frac{dv}{\sqrt v}
\end{equation}
where the function $k_j(v)$ is smooth and of polynomial growth at $\infty$. It follows that the measure $d\mu_j=k_j(v)\frac{dv}{\sqrt v}$ is a Radon measure such that
$\int_0^\infty e^{-\alpha v}|d\mu_j|<\infty$ for any $\alpha>0$ and one obtains
the conclusion using \eqref{fjfunctbis} and Definition \ref{Binf+}. \endproof

 \subsection{Source term and perturbative treatment of the Airy integral}

In this section we investigate what happens if we treat the Airy integral by introducing a source term and by performing perturbation theory around a Gaussian: this is a familiar method in the theory of Feynman integrals.
We consider the Airy integral in the form
$$
F(u)=u^{-\frac 13}\ai(-u^{\frac 23})=\frac{1}{2\pi}
\int_{-\infty}^\infty e^{iu\left(\frac{t^3}{3}-t\right)}dt
$$
and we introduce a source term
\begin{equation}\label{source0}
F(u,j)=\frac{1}{2\pi}
\int_{-\infty}^\infty e^{iu\left(\frac{t^3}{3}-t+ jt\right)}dt
\end{equation}
in  order to understand the relative roles of the variables $u$ and $j$. One has, for $t=
(1-j)^{1/2}s$,
$$
iu\left(\frac{t^3}{3}-t+ jt\right)=iu(1-j)^{3/2}\left(\frac{s^3}{3}-s\right)
$$
so that for $j<1$ one gets
\begin{equation}\label{source1}
F(u,j)=(1-j)^{1/2}F(u(1-j)^{3/2}).
\end{equation}
The formula \eqref{source1} gives us, for $j<1$, fixed the control of the behavior of the integral \eqref{source0} when $u\to \infty$. 

Next, we compare this with  the perturbative method around a critical point. We choose a critical point for the action without the source, we take $t=1$ and write $t=1+\phi$. We are then dealing with the exponent $iu\left(-\frac{2}{3}+j+j \phi +\phi ^2+\frac{\phi ^3}{3}\right)$, and thus with the integral
\begin{equation}\label{pert1}
F(u,j)=e^{iu (-\frac{2}{3}+j)}\frac{1}{2\pi}
\int_{-\infty}^\infty e^{iu\left(\phi ^2+\frac{\phi ^3}{3}+ j\phi\right)}d\phi.
\end{equation}
We now see how this integral is treated in the perturbative manner. One first introduces a coupling constant $g$ in front of the interaction term. The reason for doing that is to be able to proceed by integrating against a Gaussian. When $g=0$ the integral is Gaussian  and one then
expands around $g=0$ to obtain the result in general. Thus one deals with
 \begin{equation}\label{pert2}
F(u,j,g)=e^{iu (-\frac{2}{3}+j)}\frac{1}{2\pi}
\int_{-\infty}^\infty e^{iu\left(\phi ^2+g\frac{\phi ^3}{3}+ j\phi\right)}d\phi
\end{equation}
 so that $F(u,j)=F(u,j,1)$. One treats $g$ as small and one looks for an asymptotic expansion in powers of $g$. Since the interaction term is of higher order,  we get the equation $W_0={\rm Legendre}(S)$. The change of variables is that $u=\frac 1\hbar$, and the action $S(\phi)$ is given by
 \begin{equation}\label{sphi}
    S(\phi)=\phi ^2+g\frac{\phi ^3}{3}.
 \end{equation}
 One computes the Legendre transform of $S$ perturbatively. One has two solutions of the equation $\delta S/\delta \psi=-j$ which are given by
 \begin{equation*}
    \psi=\frac{-1\pm\sqrt{1-g j}}{g}
 \end{equation*}
and the solution which is selected by the perturbative expansion is
\begin{equation}\label{pertsol}
    \psi_+=\frac{-1+\sqrt{1-g j}}{g}=-\frac{j}{2}-\frac{g j^2}{8}-\frac{g^2 j^3}{16}-\frac{5 g^3 j^4}{128}+O(j)^5.
\end{equation}
One thus gets at the perturbative level
\begin{equation*}
    W_0(j)=S(\psi_+)+j\psi_+
\end{equation*}
and taking into account the term $e^{iu (-\frac{2}{3}+j)}$ one gets the following evaluation for the exponent
$$
iu\left( -\frac{2}{3}+j+S(\psi_+)+j\psi_+\right)=iu\left(-\frac{2}{3}+j-\frac{j^2}{4}-\frac{j^3}{24}+\frac{1}{64} \left(-2 g+g^2\right) j^4+O(j)^5\right)
$$
where in closed form one has
$$
S(\psi_+)+j\psi_+=\frac{\left(-1+\sqrt{1-g j}\right) \left(3 g^2 j+3 g \left(-1+\sqrt{1-g j}\right)+\left(-1+\sqrt{1-g j}\right)^2\right)}{3 g^3}.
$$
Taking $g=1$, the above expression simplifies and  one obtains
$$
-\frac{2}{3}+j+S(\psi_+)+j\psi_+=-\frac{2}{3} (1-j)^{3/2}
$$
which gives the exponent
$$
iu\left( -\frac{2}{3} (1-j)^{3/2}\right).
$$
This shows that the perturbative expansion corresponds to taking the integral over the path $C_2$ in the expression of the Airy function and gives a strong asymptotic expansion of this term but of course 
it completely ignores the contribution of $C_1$ which is nevertheless essential.

Let us now look at the non-perturbative behavior of the functional integral as a function of the source $j$. One uses the usual normalization which amounts to divide by the value at $j=0$ and thus we consider
\begin{equation}\label{truefunct}
    \left(\frac{1}{2\pi}
\int_{-\infty}^\infty e^{iu\left(\frac{t^3}{3}-t+ jt\right)}dt\right)/
 \left(\frac{1}{2\pi}
\int_{-\infty}^\infty e^{iu\left(\frac{t^3}{3}-t\right)}dt\right)=F(u,j)/F(u).
\end{equation}
One has, using  \eqref{source1}, and for $j<1$
\begin{equation}\label{truefunct1}
F(u,j)/F(u)=(1-j)^{1/2}F(u(1-j)^{3/2})/F(u)=\frac{\ai(-(1-j)u^{\frac 23})}{\ai(-u^{\frac 23})}.
\end{equation}
Since the denominator has many zeros which do not correspond to zeros of the numerator one obtains a function which oscillates wildly between the poles
as shown in Figure \ref{airysin}.

\begin{figure}
\begin{center}
\includegraphics[scale=0.8]{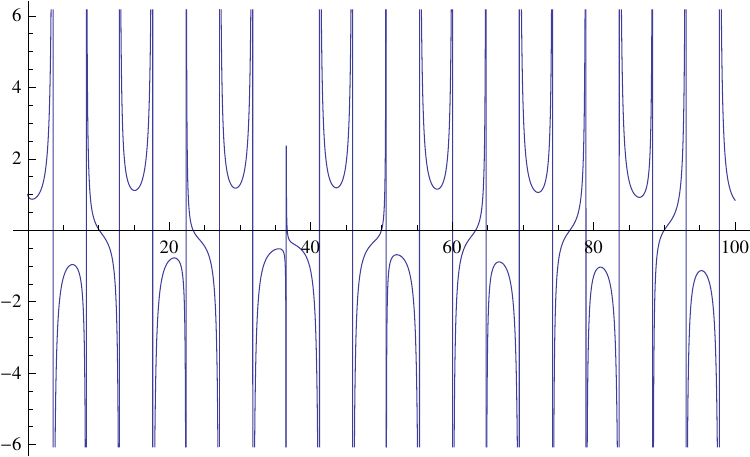}
\caption{The graph of $\frac{\ai(-(1-j)u^{\frac 23})}{\ai(-u^{\frac 23})}$
for $j=\frac 12$.}\label{airysin}
\end{center}
\end{figure}
We can thus summarize the treatment of the integral with a source \eqref{source0} as follows 
\begin{equation}\label{source0bis}
F(u,j)=\frac{1}{2\pi}
\int_{-\infty}^\infty e^{iu\left(\frac{t^3}{3}-t+ jt\right)}dt
\end{equation}
\begin{enumerate}
  \item The quotient $F(u,j)/F(u,0)$ has infinitely many poles in $u$ when $j\neq 0$ and its logarithm does not make sense.
  \item This behavior is due to the presence of two critical points which each contribute by a wave without any coherence between the two waves.
  \item The perturbative treatment chooses one of the critical points and makes an expansion of the contribution of this critical point, thus giving this term up to strong equivalence.
\item The presence of several critical points forces one to add the contributions of each critical point and replaces the exact knowledge of each of these terms up to strong equivalence by a hypersum in the  quotient hyperfield. 
\end{enumerate}
 Thus one can conclude that the way the perturbative method disregards the problem of the hypersum relies in the fact that it does not compute the full integral but only a portion of it corresponding to the choice of a critical point. Clearly, a complete understanding of the process requires  to consider the full integral by add up the various contributions, therefore  the appearance of the hypersum cannot be avoided. In the context of gauge theories in physics, the presence of several critical points is unavoidable and for this reason we expect that the formalism deployed by the theory of hyper-structures (hyperrings and hyperfields) might shed some light on the evaluation of Feyman integrals in that context.

\subsection{A toy model: the hyperfield $\cfl$}

The choice of a critical point in the asymptotic expansions which one performs in quantum physics to interpret the result in a classical manner is guided by the Wick rotation whose effect is to replace an integral of imaginary exponentials by an integral of real exponentials. This process is justified in quantum field theory where one then rotates back by analytic continuation from the Euclidean formulation to the Minkowski space physical description. This suggests to select the critical points using the following total ordering on $\C$.
Let $\C_+\subset \C$ be defined by
\begin{equation}\label{cplus}
    \C_+=\{z\in \C\mid \Re(z)\geq 0\; {\rm and}\, \Im(z)\geq 0\; {\rm if} \; \Re(z)=0\}
\end{equation}
We write $z\leq z'$ for $z'-z\in \C_+$. This defines a total order relation compatible with addition. 

\subsubsection{The hyperfield $\cfl$}

As a set  $\cfl=(\C\coprod\C)\cup\{0\}$ is the union of two copies of $\C$ and $\{0\}$. We write its non-zero elements as $\epsilon\, e(z)$ where $\epsilon=\pm 1$ and $z\in \C$. The multiplicative structure  is defined by 
$
\epsilon e(z)\cdot \epsilon' e(z'):=
\epsilon \epsilon' e(z+z')
$ and the additive structure is given by
\begin{equation}\label{hyperhp}
    \epsilon\, e(z)+ \epsilon'\, e(z')=
    \left\{
                 \begin{array}{ll}
                  \epsilon\, e(z), & \hbox{if $z'<z$;} \\
                   \epsilon'\, e(z'), & \hbox{if $z<z'$;} \\
                  \epsilon\, e(z)& \hbox{if $z=z'$, $\epsilon=\epsilon'$;}\\
  \{0\}\cup\{\epsilon''e(z'')\mid z''\leq z\}, & \hbox{if $z=z'$, $\epsilon=-\epsilon'$.}
                 \end{array}
               \right.
\end{equation}
One checks that these laws define a hyperfield structure on $\cfl$. Moreover as in the real case one obtains

\begin{prop}\label{perfprop}
The hyperfield $\cfl$ is perfect of characteristic one.
\end{prop}

\subsubsection{Description of $\cfl$ as perfection of $\C$}

 We shall use an analogous formula as in the real case where we take $\kappa=\frac 13$ in Theorem \ref{limitcase}. The only difference is that in the complex case we require that the sequence $(x^{(j)})_{j\in \Z}$ of complex numbers $x^{(j)}\in \C$, is doubly infinite and  {\em convergent} when $j\to -\infty$. This nuance  makes no difference in the real case since any sequence $(x^{(j)})_{j\geq 0}$ of real numbers such that $x^{(j+1)} = (x^{(j)})^3\qqq j$, uniquely extends to 
a doubly infinite sequence $(x^{(j)})_{j\in \Z}$ fulfilling $x^{(j+1)} = (x^{(j)})^3\qqq j$, and the obtained sequence is automatically convergent when $j\to -\infty$ (and its limit belongs to $\{-1,0,1\}$). 

\begin{thm}\label{nicecflat}
One has a canonical isomorphism of $\cfl$ with doubly infinite sequences $x^{(j)}\in \C$ as follows
\begin{equation}\label{projfon1}
\cfl\stackrel{\sqrt{}}{\to}\{ (x^{(j)})_{j\in \Z},\,{\rm convergent\  for }\ j\to -\infty \mid x^{(j+1)} = (x^{(j)})^3\qqq j\in\Z\}
\end{equation}
The map $\sqrt{}$ associates to $0$ the sequence $x^{(j)}=0$ and to $x=\epsilon\, e(z)\in \cfl$ the sequence
$
  x^{(j)}=\epsilon\; e^{3^j z}
$.
\end{thm}
\proof The map $\sqrt{}$ which associates to $x=\epsilon\, e(z)\in \cfl$ the sequence
$
  x^{(j)}=\epsilon\; e^{3^j\,  z}
$
is well defined since $(x^{(j)})^3 = x^{(j+1)}$ as $3$ is odd, while
$x^{(j)}\to \epsilon$ as $j\to -\infty$. We show that the map $\sqrt{}$ is bijective. It is injective because the sequence $
  x^{(j)}=\epsilon\; e^{3^j\,z}
$ determines both $\epsilon$ and $z$ by the equalities
\begin{equation*}
    \epsilon=\lim_{j\to -\infty} x^{(j)}\,, \ \ z=\lim_{j\to -\infty}3^{-j}\left(\epsilon \,
    x^{(j)}-1\right).
\end{equation*}
Let now $x^{(j)}$ be a doubly infinite sequence of complex numbers as in \eqref{projfon1}. If $x^{(0)}=0$ then all $x^{(j)}$ are $0$. Assume that $x^{(0)}\neq 0$ and let $\mu=|x^{(0)}|$. Then one has $|x^{(-j)}|=\mu^{1/3^j}\to 1$ when $j\to \infty$. The  limit $\epsilon=\lim_{j\to -\infty} x^{(j)}$ fulfills $\epsilon\neq 0$ and
\begin{equation*}
    \epsilon^3=\lim (x^{(j-1)})^3=\lim x^{(j)}=\epsilon
\end{equation*}
so that $\epsilon=\pm 1$. Replacing $x^{(j)}$ by $-x^{(j)}$ we can assume that $\epsilon=1$ \ie that $x^{(j)}\to 1$ when $j\to- \infty$. Since $x^{(j)}\to 1$ one has $|x^{(j)}-1|<1$ for $j\leq j_0$ and thus
\begin{equation}\label{jplusone}
    x^{(j)}=e^{z_0/3^{j_0-j}}\qqq j\leq j_0, \ \ z_0=\log(x^{(j_0)})
\end{equation}
where $\log(x^{(j_0)})$ is defined by the convergent series
\begin{equation*}
    \log(x^{(j_0)})=-\sum_1^\infty \frac{(1-x^{(j_0)})^n}{n}.
\end{equation*}
Thus one gets, with $z=3^{-j_0} z_0$ the equality
$
    x^{(j)}=e^{3^{j}\,z}\qqq j,
$
and hence the surjectivity of $\sqrt{}$.\endproof

It is important to have an explicit formula for the natural extension of the sequence $x^{(j)}$ to a continuous, one parameter family $x(t)\in \C$, $t\in \R$.
\begin{cor}\label{extend}
Let $x^{(j)}$ be a non-zero  sequence of complex numbers such that $ (x^{(j)})^3 = x^{(j+1)}$ for all $j\in \Z$
and which is convergent for $j\to -\infty$.
There exists a unique continuous one parameter family $x(t)\in \C$, $t\in \R$ such that
\begin{itemize}
  \item $x(3^{j})=x^{(j)}\qqq j\in \Z$.
  \item $x(kt)=x(t)^k$ for all odd $k\in \Z$.
\end{itemize}
Let $\epsilon=\lim_{j\to -\infty} x^{(j)}$. Then one has for any $t>0$
\begin{equation}\label{xt}
   x(t)=\epsilon\,\prod (\epsilon \,x^{(j)})^{a_j}\qqq a_j\in \{0,1,2\}, \ \sum a_j\,3^{j}=t.
\end{equation}
\end{cor}
\proof The existence of the $x(t)$ follows from Theorem \ref{nicecflat}. Its uniqueness follows from the density in $\R$ of the $a\, 3^{-k}$ where $a\in \Z$ is odd. To prove \eqref{xt} one can assume that
$\epsilon=\lim x^{(j)}$ is $1$. One then has $x(t)=e^{zt}$ for some $z\in \C$ and \eqref{xt} follows. Note that the infinite product is absolutely convergent since
$
    \sum_{j\leq 0} |\epsilon\,x^{(j)}-1|<\infty.
$
\endproof
The product of two convergent sequences is convergent and thus
it is immediate to get the product of two elements of $\cfl$ from their representation as doubly infinite sequences
\begin{equation}\label{producth}
   (x^{(j)})_{j\in \Z}. (y^{(j)})_{j\in \Z}=(x^{(j)}y^{(j)})_{j\in \Z}.
\end{equation}
For the addition, the natural formula to try  is (as in the $p$-adic and real cases)
\begin{equation}\label{padicbis}
(x+y)^{(i)}= \lim_{j \ra \infty} (x^{(i+j)}+y^{(i+j)})^{3^{-j}}
\end{equation}
however, one needs to handle here the ambiguity in the extraction of roots of order a power of $3$.
When $|x^{(0)}|>|y^{(0)}|$ this is easily done since, for $j\geq 0$
\begin{equation*}
    x^{(i+j)}+y^{(i+j)}=x^{(i+j)}\left(1+(y^{(i)}/x^{(i)})^{3^{j}}\right)
\end{equation*}
while for $j\to \infty$ one has, since  $|y^{(i)}/x^{(i)}|<1$
\begin{equation}\label{naturalguess}
  \left(1+(y^{(i)}/x^{(i)})^{3^{j}}\right)^{3^{-j}}\to 1
\end{equation}
using the unique extraction of roots in a neighborhood of $1$. Thus this gives
\begin{equation}\label{padicter}
(x+y)^{(i)}= x^{(i)} \ \ \text{if} \ |x^{(0)}|>|y^{(0)}|.
\end{equation}
What is new in the complex case  is that in the case when $|x^{(0)}|=|y^{(0)}|$ (\ie when the two sequences have the same modulus) it is the behavior of $x(t) + y(t)$  on the {\em imaginary} axis (\ie for $it\to +\infty$) which allows one to get the hypersum. The required analytic continuation in the parameter $t$ connects with the Wick rotation of quantum physics. Note that there is nevertheless also a direct manner  to decide, assuming $y\neq -x$, which between $x$ and $y$ is the hypersum $x+y$, this is achieved by considering the behavior of the sequences $(x^{(j)})_{j\geq 0}$ and $ (y^{(j)})_{j\geq 0}$  for  $j\to -\infty$. When $|x|\neq |y|$ it is the sequence of largest modulus. When $|x|= |y|$ one considers the sequence
\begin{equation*}
    u(j)=x^{(j)}/y^{(j)}, \ \ \ j\to -\infty,
\end{equation*}
of complex numbers of modulus $1$ and the hypersum is $x$ (resp. $y$) when the sequence rotates in a clockwise (resp. anticlockwise) manner for $j\to -\infty$.

  \subsubsection{$\C$ as the quotient of $\cfl$ by the Euler relation $e^{i\pi}=-1$.}

  We show that $\cfl$ appears naturally as the perfection of the  hyperfield $\troc$ of Viro tropical complex numbers (\cf \cite{V}). The multiplicative structure of $\troc$ is the same as for ordinary complex numbers and we recall the definition of the hypersum $a\smile b$ in the case of Viro tropical complex numbers. One sets
 \begin{enumerate}
   \item If $|a|<|b|$:~~$a\smile b=b$;~ if $|a|>|b|$:~~ $a\smile b=a$.
   \item If $|a|=|b|$ and $a\neq -b$, with $a = re^{i\alpha}, b= re^{i\beta}$ and $|\alpha-\beta|<\pi$
       \begin{equation*}
     a\smile b=  \{re^{i\varphi} |  \, |\alpha-\varphi| + |\varphi -\beta| = |\alpha-\beta|\}
       \end{equation*}
   \item If $a+b=0$:~~  $a\smile b$ is the closed disk $\{c\in \C\mid |c|\leq |a|\}$.
 \end{enumerate}
The hyperfield $\troc$ is not perfect. This conclusion is  obvious since the map $x\mapsto x^n$ is not bijective say for $n=3$.

 \begin{prop}\label{mapexp}
 $(i)$~The following map ${\rm ev}: \cfl\to \troc $ is a hyperfield homomorphism
 \begin{equation}\label{exp}
    {\rm ev}(\epsilon e(z)):=\epsilon e^z\qqq  \epsilon \in \{\pm 1\}, \ z \in \C,\ \  {\rm ev}(0)=0.
\end{equation}
Moreover (with the notations of Theorem \ref{nicecflat}) one has ${\rm ev}(x)=x^{(0)}$ for all $x\in \cfl$.

$(ii)$~The hyperfield homomorphism ${\rm ev}$ is surjective and at the level of the multiplicative groups one has the exact sequence
\begin{equation}\label{exsequ}
  1\to  (-e(i\pi))^\Z\to \cfl^\times\stackrel{{\rm ev}}{\to} \C^\times \to 1.
\end{equation}
 \end{prop}
 \proof The map ${\rm ev}: \cfl\to \troc $ is multiplicative, we need to check that it is compatible with the hyperaddition. 
 Let $x=\epsilon e(z)$, $x'=\epsilon e(z')$, we show that
 \begin{equation}\label{addev}
    {\rm ev}(x+x')\subset {\rm ev}(x)\smile {\rm ev}(x').
 \end{equation}
 Assume first that $\Re(z)<\Re(z')$. Then one has $z'-z\in \C_+$ and thus
  $x+x'=x'$ by \eqref{hyperhp}. One has $|{\rm ev}(x)|=e^{\Re(z)}$ and thus
$|{\rm ev}(x)|<|{\rm ev}(x')|$ so that $ {\rm ev}(x)\smile {\rm ev}(x')={\rm ev}(x')$.
This shows that \eqref{addev} holds when $\Re(z)\neq\Re(z')$. Assume now that
$\Re(z)=\Re(z')$. Then $|{\rm ev}(x)|=|{\rm ev}(x')|$ and the definition of the hypersum
$\smile$ shows that in this case $ {\rm ev}(x)\smile {\rm ev}(x')\supset\{{\rm ev}(x),{\rm ev}(x')\}$. This shows, using \eqref{hyperhp}, that \eqref{addev} holds when $x'\neq - x$. Assume now that $x'=-x$. Then ${\rm ev}(x)=- {\rm ev}(x')$ and
${\rm ev}(x)\smile {\rm ev}(x')$ is the closed disk $\{c\in \C\mid |c|\leq |{\rm ev}(x)|\}$.
With $x=\epsilon \,e(z)$ one has
\begin{equation*}
x+x'=\{0\}\cup\{\epsilon''e(z'')\mid z''\leq_P z\}.
\end{equation*}
But $z''\leq_P z$ implies $\Re(z'')\leq \Re(z)$ and hence
$|{\rm ev}(\epsilon''e(z''))|=e^{\Re(z'')}\leq e^{\Re(z)}=|{\rm ev}(x)|$  so that
 ${\rm ev}(\epsilon''e(z''))$ belongs to the closed disk $\{c\in \C\mid |c|\leq |{\rm ev}(x)|\}$. We thus get \eqref{addev} in this case also and this shows that the map ${\rm ev}$ is a hyperfield homomorphism. \newline
For the second statement note that the map ${\rm ev}$ is a group homomorphism $\cfl^\times\stackrel{{\rm ev}}{\to} \C^\times$ and its kernel is the cyclic group generated by the element $-e(i\pi)\in \cfl$.\endproof

\subsubsection{Universal $W$-model of $\cfl$}

 The  hyperfield $\cfl$ admits a universal $W$-model. Given  finitely many elements $z_j\in \C$ we denote by $\vee z_j$ the unique largest element for the  total order associated to $ \C_+$.
 The following formula defines a homomorphism $\rho$ from the group ring $R=\Q[\C]$ to $\cfl$
\begin{equation}\label{toctr}
    \rho(\sum_1^n a_j\epsilon_j u(z_j)):=\epsilon_k \, e(z_k)\, , \ z_k=\vee z_j
\end{equation}
which extends to a homomorphism of hyperfields from the field $K={\rm Frac}(R)$ to $\cfl$.

\begin{thm}\label{thmunic} The triple $(W=K,\rho, \tau_W)$ is the universal $W$-model for $H=\cfl$. The homomorphism $\rho$ induces an isomorphism of hyperfields  $W/G\stackrel{\sim}{\to} \cfl$, where $G = \ker(\rho: W^\times \to \cfl^\times)$.
\end{thm}
\proof The proof is similar to the proof of Theorem \ref{thmuni}  and is left to the reader.\endproof

\subsubsection{The map $\theta_\C$ and the ring $\C_\infty$}
We proceed as in the real case and construct the universal formal pro-infinitesimal thickening of the field $\C$.
Theorem~\ref{thmunic} gives not only the field $W(\cfl)$ but also the subalgebra
$W_\Q(\cfl)$  generated by the \te lifts $[x]$ for $x\in \cfl$. 
\begin{prop}\label{propthetac} There exists a unique ring homomorphism $\theta_\C: W_\Q(\cfl)\to \C$ such that ($[~]=\tau$)
\begin{equation}\label{22c}
\theta_\C([x]) =\theta_\C(\tau(x))= {\rm ev} (x),\quad\forall x\in \cfl.
\end{equation}
\end{prop}
\begin{proof} By construction $W_\Q(\cfl)$ is the subalgebra (over $\Q$) generated by the \te lifts $[x]$ for $x\in \cfl$. With $x=\epsilon e(z)$ one has $[x]=\epsilon \,u(z)$ and thus one gets that $W_\Q(\cfl)= \Q[\C]$. Thus the natural map $u(z)\mapsto e^z$ extends by linearity and uniquely to a ring homomorphism
\begin{equation}\label{23c}
\theta_\C(\sum_ia_i[x_i]) = \sum_ia_i{\rm ev}(x_i)\in\C.
\end{equation}
\end{proof}

As in the real case, this suggests to consider
the homomorphism  $\theta_\C: W_\Q(\cfl)\to \C$ of Proposition~\ref{propthetac} and introduce the following
  \begin{defn} The universal formal pro-infinitesimal thickening $\C_\infty$ of $\C$ is the $\Ker(\theta_\C)$-adic completion of $W_\Q(\cfl)$, \ie \[\C_\infty=\varprojlim_n W_\Q(\cfl)/\Ker(\theta_\C)^n.\]
  \end{defn}

  We shall now proceed as  in the real case to show that $\Ker(\theta_\C)/\Ker(\theta_\C)^2$ is an infinite dimensional complex vector space. Our main goal will be that to construct explicitly
a two dimensional complex space of linear forms on $\Ker(\theta_\C)/\Ker(\theta_\C)^2$.
We introduce the following vector spaces over $\C$
\begin{defn}\label{homcc} We let $\Hom_\Z(\C,\C)$ be the complex vector space of all {\em additive} maps $L:\C\to \C$, and $\Hom_\R(\C,\C)\subset\Hom_\Z(\C,\C)$ the two dimensional subspace of $\R$-linear maps.
\end{defn}
One has by definition
\begin{equation*}
    (a\phi+b\psi)(z):=a\phi(z)+b\psi(z)\in \C\qqq a,b,z\in \C\,, \ \phi,\psi\in \Hom_\Z(\C,\C).
\end{equation*}
Note that $\Hom_\R(\C,\C)\subset\Hom_\Z(\C,\C)$ is also the subspace of additive maps which are measurable and that it is only by the virtue of the axiom of choice that $\Hom_\Z(\C,\C)$ is infinite dimensional, while only the elements of the subspace $\Hom_\R(\C,\C)$ can be concretely exhibited. We write the elements of $\Hom_\R(\C,\C)$ in the form
\begin{equation}\label{hom}
    L(z)=az+b\bar z\qqq z\in \C
\end{equation}
so that $(a,b)$ are the natural coordinates in this complex vector space.
\begin{lem}\label{tayl}
$(i)$~Let $\ell\in \Hom_\Z(\C,\C)$,  then the  map
 \begin{equation}\label{taulc}
    \cT_\ell(X)(u):=\sum_ia_i e^{z_i+ u\ell(z_i)}\qqq X=\sum_ia_iu(z_i)\in W_\Q(\cfl)
\end{equation}
defines a ring homomorphism $\cT_\ell:W_\Q(\cfl)\to \cE$ to the ring of entire  functions
of the variable $u\in \C$ and
\begin{equation}\label{valatval}
   \theta_\C(X)=\cT_\ell(X)(0),~\forall X\in W_\Q(\cfl).
\end{equation}
$(ii)$~Let $\ell\in \Hom_\Z(\C,\C)$. Then  the following defines a linear form on $\Ker(\theta_\C)/\Ker(\theta_\C)^2$
\begin{equation}\label{ll}
\Ker(\theta_\C) \ni X\mapsto \delta_\ell(X)= \left(\frac{d}{du} \cT_\ell(X)(u)\right)_{u=0}
\end{equation}
\end{lem}
\proof $(i)$~For each $u\in \C$ the map $z\mapsto e^{z+ u\ell(z)}$ is a group homomorphism from the additive group $\C$ to the multiplicative group $\C^\times$. Thus this map extends to a ring homomorphism from the group ring $W_\Q(\cfl)= \Q[\C]$ to $\C$. This shows that $\cT_\ell$ is a homomorphism to the algebra of functions with pointwise operations. Since $\cT_\ell(X)$ is a finite linear combination of exponential functions of $u$ it is an entire function. One checks \eqref{valatval} using the definition \eqref{22c} of $\theta_\C$.\newline
$(ii)$~First the right hand side of \eqref{ll} vanishes when $X\in \Ker(\theta_\C)^2$ since the entire function $\cT_\ell(X)(u)$ admits a zero of order at least two at $u=0$ as can be seen  using \eqref{valatval}. This shows that $\delta_\ell$ is well defined. It is clearly additive. Let us show that it is $\C$-linear. For the structure of complex vector space on $W_\Q(\cfl)/\Ker(\theta_\C)=\C$,  the multiplication by a complex number $y\in \C$ is provided by the multiplication by any $s\in  W_\Q(\cfl)$ such that $\theta_\C(s)=y$.  We then have, with $X\in \Ker(\theta_\C)$, the expansion at $u=0$
\begin{equation*}
  \cT_\ell(sX)(u)=\cT_\ell(s)(u)\cT_\ell(X)(u)=\theta_\C(s)\delta_\ell(X)u+O(u^2).
\end{equation*}
This shows that $\delta_\ell$ is $\C$-linear.
\endproof

\subsubsection{The periods $\epsilon$ and $\pi_p$}

As in Theorem \ref{thmkerth} one can use  Lemma \ref{tayl} to show that the ``periods" of the form $\pi_p=[e(\log p)]-p$ are linearly independent elements of $\Ker(\theta_\C)/\Ker(\theta_\C)^2$. Next, we construct another ``period" which is purely complex. We start with the analogue of the element $\epsilon\in F(\C_p)$ of the $p$-adic Hodge theory. We define in our case
\begin{equation}\label{defeps}
    \varepsilon:= e(2i\pi)\in \cfl.
\end{equation}
The natural square root $\varepsilon^{(2)}$ of $\varepsilon$ is $\varepsilon^{(2)}= e(i\pi)\in \cfl$ and
one has
\begin{equation}\label{epseps}
    \theta_\C([\varepsilon])=1, \ \ \omega\in \Ker(\theta_\C), \ \omega=([\eps]-1)/([\varepsilon^{(2)}]-1).
\end{equation}
The last part follows from $\omega=1+[\varepsilon^{(2)}]$ and the fact that $e^{i\pi}=-1$. Now that we have these various ``periods" we can evaluate on them the natural linear forms $\delta_\ell$ on $\Ker(\theta_\C)/\Ker(\theta_\C)^2$ given by elements of $\Hom_\R(\C,\C)$. 
\begin{lem}\label{tayl2}
For $L\in \Hom_\R(\C,\C)$ given by \eqref{hom} one has
\begin{equation}\label{ll1}
    \delta_L(\pi_p)=(a+b)p\log(p),\ \  \delta_L(\omega)=i\pi(b-a).
\end{equation}
\end{lem}
\proof 
Let  $\ell=L$ with $L$ given by \eqref{hom}. One has
\begin{equation*}
   \delta_L(\pi_p)=\left(\frac{d}{du}\right)_{u=0} \cT_\ell([e(\log p)]-p)(u)
   =\left(\frac{d}{du}\right)_{u=0}e^{\log p+ u L(\log p)}=pL(\log p)
\end{equation*}
which gives $(a+b)p\log(p)$ by \eqref{hom}. Similarly
\begin{equation*}
   \delta_L(\omega)=\left(\frac{d}{du}\right)_{u=0}  \cT_\ell(1+[\varepsilon^{(2)}])(u)
   =\left(\frac{d}{du}\right)_{u=0} e^{i\pi+ u L(i\pi)}=-L(i\pi)
\end{equation*}
which gives $i\pi(b-a)$ by \eqref{hom}.\endproof

  \newpage

\appendix

\section{Table of correspondences with $p$-adic Hodge theory}\label{tabrec}

The following table reports the archimedean structures that we have defined and discussed in this paper and their $p$-adic counterparts (\cf \cite{FF1})
\vspace{.5in}

\begin{center}
  \begin{tabular}{|c|c|}
  \hline &\\
    $p$-{\bf adic case} & {\bf Archimedean case} \\
    &\\
     \hline
       &\\
    $\F_p$ & $\sign=\{-1,0,1\}$\quad hyperfield of signs \\
    &\\
    \hline
       &\\
    $F=F(\C_p)$ & $F(\R)=\trop \subset \cfl =F(\C)$ \\
    &\\
    \hline &\\
$\epsilon\in F(\C_p)$ & $ \varepsilon:= e(2i\pi)\in \cfl =F(\C)$ \\
    &\\
    \hline &\\
        $\etplus=\displaystyle{\varprojlim_{x\mapsto x^p}}\ \OO_{\C_p}$ & $\cO=\displaystyle{\varprojlim_{x\mapsto x^\kappa}} \ [-1,1]\subset \displaystyle{\varprojlim_{z\mapsto z^\kappa}} \ \{z, \vert z\vert \leq 1\}$ \\&\\
        \hline     &\\
    $ \wittOovpi=\wittO[1/\pi]
  $ & $\btplusarc=\{f(z)=\int_0^\infty e^{-\xi z} d\mu(\xi)\mid \mu$ finite real measure$\}$  \\
    &\\
    \hline
     &\\
       $x=\displaystyle{\sum_{n\gg-\infty}}[x_n]\pi^n\in\wittOovpi$& $f=\int_{s_0}^\infty [f_s]e^{-s}ds\in\btplusarc$, $f_s\smile f_t=f_s$ for $s\leq t$ \\&\\
    \hline
    &\\
    $\wittO =\{x\in\wittOovpi|~x=\displaystyle{\sum_{n\geq 0}}[x_n]\pi^n\}$ &   $\{f\in \btplusarc\mid \|f\|_0\leq 1\}=\{f\in \btplusarc\mid f=\int_{0}^\infty [f_s]e^{-s}ds\}$\\
    &\\
    \hline
     &\\
    $\theta: \wittOovpi \ra \Cp$ & $\theta: \btplusarc \ra \R$, $\   \   \theta_\C: W_\Q(\cfl)\to \C$ \\
    &\\
    \hline
    &\\
    $ \varphi(\displaystyle{\sum_{n\gg-\infty}}[x_n]\pi^n)=\displaystyle{\sum_{n\gg-\infty}}[x_n^q]\pi^n$&
    $\arf_\lambda(\int_{s_0}^\infty [f_s]e^{-s}ds)=\int_{s_0}^\infty [f_s^\lambda]e^{-s}ds$\\
    &\\
        \hline
         &\\
        $|x|_\rho^\alpha=\displaystyle{\max_\Z} |x_n|^\alpha q^{-n}$&  $\|f\|_\rho=\int_{s_0}^\infty |f_s|^\alpha e^{-s}ds$            \\
         &\\
        \hline
  \end{tabular}
  \end{center}\bigskip

\newpage

\section{Universal perfection}\label{uniperf}

In this appendix we give a short overview of the well-known construction of universal perfection in number theory: we refer to \cite{F1}, Chapter V \S 1.4; \cite{W}, \cite{F2} \S 2.1;  \cite{FF1}, \S 2.4 for more details.

The universal perfection is a procedure which associates, in a canonical way, a perfect field $F(L)$ of characteristic $p$ to a $p$-perfect field $L$. This construction is particularly relevant when $\text{char}(L)=0$, since it determines the first step toward the definition of a universal, Galois equivariant cover of $L$ (\cf Appendix~\ref{algebras}).

We recall that a field $L$ is said to be $p$-perfect if it is complete with respect to a non archimedean absolute value $|~|_L$, $L$ has a residue field of characteristic $p$ and the endomorphism of $\cO_L/p\cO_L$, $x\to x^p$ is surjective. Furthermore, the field $L$ is said to be strictly $p$-perfect if  $\cO_L$ is not a discrete valuation ring.

Starting with a $p$-perfect field $L$, one introduces the  set
\begin{equation}\label{F(L)}
F(L) = \{x = (x^{(n)})_{n\in\N}|~x^{(n)}\in L;~(x^{(n+1)})^p = x^{(n)}\}.
\end{equation}
If $x,y\in F(L)$, one sets
\begin{equation}\label{laws}
(x+y)^{(n)} = \lim_{m\to\infty}(x^{(n+m)}+y^{(n+m)})^{p^m};\qquad (xy)^{(n)} = x^{(n)}y^{(n)}.
\end{equation}
We recall from \cite{FF1} (\cf \S2.4) the following result
\begin{prop} Let $L$ be a $p$-perfect field. Then $F(L)$ with the above two operations is a perfect field of characteristic $p$, complete with respect to the absolute value defined by $|x|=|x^{(0)}|_L$. Moreover
if $\mathfrak a\subset\mathfrak m_L$ is a finite type (\ie principal) ideal of $\cO_L$ containing $p\cO_L$, then the map reduction mod.~$\mathfrak a$ induces an isomorphism of topological rings
\begin{equation}\label{bij}
\cO_{F(L)} \stackrel{\sim}{\longrightarrow} \varprojlim_{n\in\N}\cO_L/\mathfrak a,\qquad x=(x^{(n)})_{n\in\N}\mapsto \bar x=(x^{(n)}\text{mod.}~\mathfrak a)_{n\in\N}
\end{equation}
where the transition maps in the projective limit are given by the ring homomorphism $\bar x\to \bar x^p$.
\end{prop}
In other words, the bijection \eqref{bij} allows one to transfer (uniquely) the natural (perfect) algebra structure on $\displaystyle{\varprojlim_{v\to v^p}} \cO_L/\mathfrak a$ over the inverse limit set $\cO_{F(L)}=\displaystyle{\varprojlim_{x\mapsto x^p}}\cO_{L}$ of $p$-power compatible sequences $x = (x^{(n)})_{n\ge 0}$, $x^{(n)}\in \cO_{L}$.  Indeed, one shows that for any $v = (v_n)\in \displaystyle{\varprojlim_{v\to v^p}} \cO_L/\mathfrak a$ and arbitrary lifts $x_n\in\cO_L$ of $v_n\in \cO_L/\mathfrak a$ $\forall n\ge 0$, the limit $x^{(n)}=\lim_{m\to\infty}x_{n+m}^{p^m}$ exists in $\cO_L$ $\forall n\ge 0$ and is independent of the choice of the lifts $x_n$. This lifting process is naturally multiplicative, whereas the additive structure on $\displaystyle{\varprojlim_{v\to v^p}} \cO_L/\mathfrak a$ lifts on $\cO_{F(L)}$ as  \eqref{laws}.

 \section{Witt construction and pro-infinitesimal thickening}\label{FW}

In this appendix we provide, for completeness, a proof of Proposition~\ref{propwittiso}. We recall that a pro-infinitesimal thickening of a ring $R$ (\cf \cite{F3}, \S 1.1.1 with $\Lambda=\Z$) is a surjective ring homomorphism $\theta: A\to R$, such that the ring $A$ is Hausdorff and complete for the $\Ker(\theta)$-adic topology \ie
\begin{equation}\label{projcond1}
    A=\varprojlim_n A/\Ker(\theta)^n.
\end{equation}
As a minor variant, we consider triples
 $(A,\theta,\tau)$, where $\theta:A\to R$ is a ring homomorphism with multiplicative section
$\tau:R\to A$ and condition \eqref{projcond1} holds.

A morphism from the triple $(A_1,\theta_1,\tau_1)$ to the triple $(A_2,\theta_2,\tau_2)$ is given by a ring homomorphism $\alpha: A_1\to A_2$ such that
\begin{equation}\label{morpwitt}
\tau_2=\alpha\circ\tau_1,\quad \theta_1=\theta_2\circ\alpha.
\end{equation}
Let $R$ be a perfect ring of characteristic $p$ and
let $W(R)$ be the $p$-isotypical Witt ring of $R$. Let  $\rho_R: W(R)\to R$ be the canonical homomorphism and $\tau_R:R\to W(R)$ the multiplicative section given by the \te lift.

By construction one has $\Ker(\rho_R)=pW(R)$ and condition \eqref{projcond1} holds.

We  show that for any triple
$(A,\rho,\tau)$ fulfilling \eqref{projcond1}, there exists a unique ring homomorphism from $(W(R),\rho_R,\tau_R)$ to $(A,\rho,\tau)$ (Compare with Theorem 4.2 of \cite{G} and Theorem 1.2.1 of \cite{F3}).
The ring $A$ with the sequence of ideals $\ffa_n=\Ker(\rho)^n$ fulfills the hypothesis of  \cite{S} (II, \S 4, Proposition 8). Thus it follows from \cite{S}  (II, \S 5, Proposition 10) that there exists a (unique) ring homomorphism
$\alpha:W(R)\to A$ such that $\rho\circ\alpha=\rho_R$. Moreover the uniqueness of the multiplicative section shown in \cite{S}  (II, \S 4, Proposition 8) proves that one  has
$\tau=\alpha\circ\tau_R$. This completes the proof of Proposition \ref{propwittiso}.

 Next we show  that the notion of thickening involving a multiplicative section $\tau$ is in general different from the classical notion.

  Consider $R=\Z$. Then,  for any surjective ring homomorphism $\theta: A\to R$, the map $\Z\ni n\mapsto n 1_A$ is the unique homomorphism from the pair $(\Z,id)$ to the pair $(A,\theta)$. It follows that the pair $(\Z,id)$ is the {\em universal pro-infinitesimal thickening of $\Z$}. This no longer holds when one involves the multiplicative section $\tau$.

Given a ring $R$, we consider $R$-triples $(A,\rho,\tau)$ where $\rho:A\to R$ is a ring homomorphism, $\tau:R\to A$ is a multiplicative section (\ie a morphism of monoids such that $\tau(0)=0$ and $\tau(1)=1$) and one also assumes \eqref{projcond1}.
A morphism between two triples is a ring homomorphism $\alpha: A_1\to A_2$ such that $\rho_1=\rho_2\circ \alpha$ and $\tau_2=\alpha\circ \tau_1$.

Proposition \ref{propwittiso}
shows that when $R$ is a perfect ring of characteristic $p$ there exists an initial object in the category of $R$-triples. For $R=\Z$ the triple $(\Z,id,id)$ is a $\Z$-triple but it is not the universal one. The latter is in fact obtained using the ring $\Z[[\{\delta_p\}]]\otimes(\Z\oplus \Z_2 e)$ of formal series with independent generators $\delta_p=[p]-p$, for each prime $p$ and an additional generator $e=[-1]+1$ such that $e^2=2e$. The augmentation defines a surjection $\epsilon:A\to \Z$, $\rho(e)=0$, and there exists a unique multiplicative  section $\tau, \,\tau(1)=1$, such that
\begin{equation}\label{taumap}
    \tau(p)=p+\delta_p\qqq p \ \text{prime}, \ \tau(-1)=-1+e.
\end{equation}
\begin{prop}\label{sumofuni}
The triple $(\Z[[\{\delta_p\}]]\otimes(\Z\oplus \Z_2 e),\epsilon,\tau)$ is the universal $\Z$-triple.
The map
\begin{equation}\label{mapD}
    D:\Z\to \Ker(\epsilon)/\Ker(\epsilon)^2, \ \ D(n):=\tau(n)-n
\end{equation}
fulfills the Leibnitz rule and its component on $\delta_p$ coincides with the map $\frac{\partial}{\partial p}:\Z\to \Z$ defined in \cite{KOW}.
\end{prop}
\proof By construction one has $\delta_p\in \Ker(\epsilon)$  and thus $\epsilon\circ \tau=id$. Consider first the subring $\Z[\{\delta_p\}][e]$ freely generated by the $\delta_p=[p]-p$ for each prime $p$ and an additional generator $e=[-1]+1$ such that $e^2=2e$. Given a $\Z$-triple $(A,\rho,\tau)$, there exists a unique ring homomorphism
\begin{equation}\label{alpha}
   \alpha: \Z[\{\delta_p\}][e]\to A\,, \ \ \alpha(\delta_p)=\tau(p)-p\,, \ \alpha(e)=\tau(-1)+1.
\end{equation}
This ring homomorphism extends uniquely, by continuity, to a homomorphism
\begin{equation*}
    \alpha:\varprojlim_n \Z[\{\delta_p\}][e]/\Ker(\epsilon)^n\to A=\varprojlim_n A/\Ker(\rho)^n
\end{equation*}
 and this shows that $(\Z[[\{\delta_p\}]]\otimes(\Z\oplus \Z_2 e),\epsilon,\tau)$  is the universal $\Z$-triple.

 The second assertion follows from \cite{KOW} (\cf Theorem 1) and the identity
 \begin{equation*}
  [nm]-nm=([n]-n)m+n([m]-m)+([n]-n)([m]-m)\qqq n, m\in \Z.
 \end{equation*}

\section{Relevant structures in $p$-adic Hodge theory}\label{algebras}

In this appendix we shortly review  some relevant constructions in $p$-adic Hodge theory which lead to the definition of the rings of $p$-adic periods. The main references are \cite{FF} and \cite{FF1}.

We fix a non-archimedean locally compact field $K$ of characteristic zero with a {\em finite} residue field $k$ of characteristic $p$: $q = |k|$. Let $v_K$ be the (discrete) valuation of $K$ normalized by $v_K(K^*)=\Z$.

Let $F$ be  any {\em perfect} field containing $k$. We assume that $F$ is complete for a given (non-trivial) absolute value $|~|$. By $W(F)$ and $W(k)$ we denote the rings of isotypical Witt-vectors.

There exists a {\em unique} (up-to a unique isomorphism) field extension $\mathfrak E_{F,K}$ of $K$, complete with respect to a {\em discrete} valuation $v$ extending $v_K$ such that:\vspace{.05in}

$\bullet$~$v(\mathfrak E_{F,K}^*)=v_K(K^*)=\Z$\vspace{.05in}

$\bullet$~$F$ is the residue field of $\mathfrak E_{F,K}$.\vspace{.05in}

One sees that $\mathfrak E_{F,K}$ can be identified with $K\otimes_{W(k)}W(F)$. Thus, if $\pi$ is a {\em chosen} uniformizing parameter of $K$, then an element of $\mathfrak E_{F,K}$ can be written {\em uniquely} as $\mathfrak e=\sum_{n\gg-\infty}[a_n]\pi^n$, $a_n\in F$. In particular $\mathfrak e\in K$ if and only if $a_n\in k$ $\forall n$.

Let $\cO_{\mathfrak E_{F,K}}$ be the (discrete) valuation ring of $\mathfrak E_{F,K}$. Each element of the ring $\cO_{\mathfrak E_{F,K}}$ can be written {\em uniquely} as $\sum_{n\ge 0}[a_n]\pi^n$, $a_n\in F$. The projection $\cO_{\mathfrak E_{F,K}}\twoheadrightarrow F$ has a unique multiplicative section \ie the \te map $a\mapsto [a]=1\otimes (a,0,0,\ldots,0,\ldots)$.

 There is a {\em universal} (local) subring $W_{\cO_K}(\cO_F)\subset\cO_{\mathfrak E_{F,K}}$ which describes the {\em unique} $\pi$-adic  torsion-free lifting of the perfect $\cO_K$-algebra $\cO_F$.
 If $K_0$ denotes the maximal unramified extension of $\Q_p$ inside $K$,  there is a {\em canonical} isomorphism: $\cO_K \otimes_{\cO_{K_0}}W(\cO_F)\stackrel{\sim}{\longrightarrow}W_{\cO_K}(\cO_F)$,  $1\otimes[a]_F\mapsto [a]$.

 If $A$ is any separated and complete $\pi$-adic $\cO_K$-algebra with field of fractions $L$ and $F=F(L)$ (\cf Appendix~\ref{FW} for notation), there is a ring homomorphism
\begin{equation}\label{thetahom}
\theta: W_{\cO_K}(\cO_{F(L)})\longrightarrow A,\qquad \sum_{n\ge 0}[x_n]\pi^n \mapsto \sum_{n\ge 0}x_n^{(0)}\pi^n.
\end{equation}
In the particular case of the algebra $A=\cO_F=\cO_{F(\mathbf C_K)}$ ($\mathbf C_K=$ completion of a fixed algebraic closure of $K$), the surjective ring homomorphism $\theta_0: \cO_F \twoheadrightarrow \cO_{\mathbf C_K}/(p)$, $\theta_0((x^{(n)})_{n\ge 0})=x^{(0)}$ lifts to a surjective ring homomorphism of $\cO_K$-algebras
\begin{equation}\label{the}
\theta: W_{\cO_K}(\cO_{F(\mathbf C_K)}) \twoheadrightarrow \cO_{\mathbf C_K}, \qquad \sum_{n\ge 0}[x_n]\pi^n \mapsto \sum_{n\ge 0}x_n^{(0)}\pi^n
\end{equation}
which is {\em independent} of the choice of the uniformizer $\pi$.\vspace{.05in}

The valued field $(\mathfrak E_{F,K}, |\cdot|)$ ($|\cdot|$ non discrete) contains two further sub-$\mathcal O_K$-algebras which are also independent of the choice of a uniformizer $\pi\in\cO_K$. They are
\begin{equation}\label{Bb}
\wittOovpi:=W_{\cO_K}(\cO_F)[\frac{1}{\pi}]=\{x=\sum_{n\gg-\infty}[x_n]\pi^n\in \mathfrak E_{F,K}|x_n\in\cO_F,\forall n\}
\end{equation}
and if $a\in\mathfrak m_F\setminus\{0\}\subset\cO_F$, the ring
$\inverta:=\wittOovpi[\frac{1}{[a]}]$
which can be equivalently described as
\begin{equation}\label{B+}
\inverta =B^b_{F,K}=\{f=\sum_{n\gg -\infty}[x_n]\pi^n\in \mathfrak E_{F,K}|\exists C>0, |x_n|\le C,\forall n\}.
\end{equation}
If a $p$-perfect field $L$ contains $K$ as a closed subfield, the ring homomorphism \eqref{thetahom} extends to a surjective homomorphism of $K$-algebras
\begin{equation}\label{teta}
\theta: B^b_{F(L),K} \to L,\qquad \theta(\sum_{n\gg-\infty}[x_n]\pi^n) = \sum_{n\gg-\infty}x_n^{(0)}\pi^n
\end{equation}
which is independent of the choice of $\pi$. If moreover $L$ is a strictly $p$-perfect field, then $|F(L)| = |L|$ and the kernel of the map $\theta$ in \eqref{teta} is a prime ideal of $B^b_{F(L),K}$ of degree one. One has
\[
\theta(B^{b,+}_{F(L),K})=L\quad\text{and}\quad \theta(W_{\cO_K}(\cO_{F(L)}))=\cO_L.
\]

Let $\mathfrak E_0=\mathfrak E_{k_F,K}$. Then the projection $\cO_F\to k_F$, $x\to \bar x$ induces an augmentation map
\begin{equation}\label{aug}
\varepsilon: B^{b,+} \to \mathfrak E_0,\qquad \varepsilon(\sum_{n\gg-\infty}[x_n]\pi^n)=\sum_{n\gg-\infty}[\bar x_n]\pi^n
\end{equation}
with $\varepsilon(W_{\cO_K}(\cO_F)) = \cO_{\mathfrak E_0}$. $E=\{\sum_{n\gg-\infty}[x_n]\pi^n|~x_n\in k_F,\forall n\}$ is a local sub-field of $\wittOovpi$.\vspace{.1in}

One introduces for $r\in\R_{\ge 0}$ the family of valuations on $\wittOovpi$:
 \[
 x=\sum_{n\gg-\infty}[x_n]\pi^n,\qquad v_r(x) = \text{inf}_{n\in\Z}\{v(x_n)+nr\}\in\R\cup\{+\infty\}
 \]
 and defines $\rigid$ as the completion of $\wittOovpi$ for the family of norms
$(q^{-v_r})_{r>0}$ ($q = |k|$), $r\in v(F)$.

An equivalent definition of these multiplicative norms  is given as follows: for $\R\ni\rho\in[0,1]$ one defines
\begin{align}\label{norms}
|x|_\rho &= \displaystyle{\max_{n\in\Z}}|x_n|\rho^n\\
|x|_0 &= q^{-r},~\text{$r$ smallest integer, $x_r\neq 0$};\qquad |x|_1 =\displaystyle{\sup_{n\in\Z}}|x_n|.\notag
\end{align}
In view of the description of $\inverta$ given in \eqref{B+}, the norms \eqref{norms} are well-defined on the larger ring $\inverta$. The completion $B$ of $\inverta$ for these norms, as $\rho\in(0,1)$, contains $\rigid[\frac{1}{[a]}]$ for any chosen $a\in\mathfrak m_F\setminus\{0\}$.\newline $B$ is the analogue in mixed characteristics of the ring of rigid analytic functions on the punctured unit disk in equal characteristics.

The subalgebra $\rigid\subset B$ is characterized by the condition
\begin{equation}\label{bsub}
    \rigid=\{ b\in B\mid |b|_1\leq 1\}.
\end{equation}
The  extension of  rings $\rigid\subset \rigid[\frac{1}{[a]}]$ gives a perfect control of the divisibility as explained in \cite{FF} Theorem 6.55.

A remarkable property of this construction is that the Frobenius endomorphism on $\wittOovpi$
\[
\varphi: \wittOovpi \to \wittOovpi,\quad \varphi(\sum_{n\gg -\infty}[x_n]\pi^n) = \sum_{n\gg -\infty}[x_n^q]\pi^n
\]
extends to a Frobenius {\em automorphism} on $\inverta$ and on $\rigid$ thus to a {\em continuous} Frobenius {\em automorphism}  $\varphi: B\stackrel{\sim}{\longrightarrow} B$ (\ie the unique $K$-automorphism which induces $x\mapsto x^q$ on $F$) which satisfies $|\varphi(f)|_{\rho^q} = (|f|_\rho)^q$, $\forall \rho\in(0,1)$.

The homomorphism \eqref{teta} (in particular for $L = \mathbf C_K$) extends to a {\em canonical continuous} universal (Galois equivariant) cover of $L$
\[
\theta: B \twoheadrightarrow L.
\]

\section{Archimedean Witt construction for semi-rings}\label{tropical}

In this appendix we explain the connection of the above construction with the archimedean analogue of the Witt construction in the framework of perfect semi-rings of characteristic one (\cite{CC1} and \cite{C}).
Given a multiplicative cancellative perfect semi-ring $R$ of characteristic $1$, one keeps the same multiplication but deforms the addition into the following operation \begin{equation}\label{pertop}
    x+_wy=\sum_{\alpha\in I}w(\alpha)x^\alpha y^{1-\alpha}\,,  \ I=(0,1)\cap \Q,
    \end{equation}
    which is commutative provided
$
w(1-\alpha)=w(\alpha),\ \forall\alpha\in I,
$
 and associative provided that the following equation holds
\begin{equation}\label{funequa}
    w(\alpha)w(\beta)^\alpha=w(\alpha \beta)w(\gamma)^{(1-\alpha \beta)}\,, \ \ \gamma=\frac{\alpha(1-\beta)}{1-\alpha \beta}\qquad\forall\alpha,\beta\in I.
\end{equation}
By applying Theorem~5.4 in \cite{C}, one sees that the positive symmetric solutions to  \eqref{funequa} are parameterized by $\rho\in R$, $\rho>1$, and they are given by the following formula involving the entropy $S$,
\begin{equation}\label{23n}
w(\alpha)=\rho^{ S(\alpha)},\qquad S(\alpha) = -\alpha\log(\alpha)-(1-\alpha)\log(1-\alpha)\qqq\alpha\in I.
\end{equation}
We apply this result to the semi-field $\rmax$ of tropical geometry.
We write the elements $\rho\in \rmax$, $\rho> 1$, in the convenient form $\rho=e^\varh$ for some $\varh>0$. In this way, we can view $w(\alpha)$ as the function of $\varh$ given by
\begin{equation}\label{walphat}
  w_\varh(\alpha)= w(\alpha,\varh)=e^{\varh S(\alpha)}\qquad\forall\alpha\in I\,.
\end{equation}
By performing a direct computation one obtains for $x,y>0$ that the perturbed sum $x+_{w_\varh}y$ is given by
\begin{equation}\label{entrop2}
  x+_{w_\varh}y=\left(x^{1/\varh}+y^{1/\varh}\right)^\varh\,.
\end{equation}
The formula \eqref{entrop2} shows that the sum of two elements of $\rmax$, computed by using $w_\varh$, is a function which depends explicitly on the variable $\varh$.
The functions $[x](\varh)=x$ (for $x\in \rmax$) which are constant in $\varh$ describe the \te lifts. The sum of  such functions is no longer constant in $\varh$.
In particular one can compute the sum of $n$ constant functions all equal to $1$:
\begin{equation}\label{1ntimes}
    1+_{w_\varh} 1+_{w_\varh}\cdots +_{w_\varh}1=n^\varh
\end{equation}
which shows that the sum of $n$ terms equal to the unit of the structure is given by the function of the variable $\varh$: $\varh\mapsto n^\varh$.
\begin{prop} \label{chirepprop}
The following map $\chi$ is a homomorphism from the semi-ring $R$ generated by the functions $\varh\mapsto \alpha^\varh$, $\alpha\in \Q_+$, and the \te lifts to the algebra of real-valued functions from $(0,\infty) $ to $\R_+$ with pointwise sum and product
\begin{equation}\label{chimap}
    \chi(f)(\varh)=f(\varh)^{1/\varh}\qqq \varh>0.
\end{equation}
The range of the map $\chi$ is the semi-ring of  finite linear combinations, with positive rational coefficients, of \te lifts of elements $x\in \rmax$ given in the $\chi$-representation by
\begin{equation}\label{teichrep}
    \chi([x])(\varh)=x^{1/\varh}\qqq \varh>0\qqq x\in \rmax
\end{equation}
The following defines a homomorphism from $R$ to $W_\Q(\trop)$,
\begin{equation}\label{mapback}
    f\mapsto \beta(f), \ \ \beta(f)(z)=\chi(f)(\frac 1z)\,.
\end{equation}
\end{prop}
\proof
The proof is straightforward using \cite{C}.\endproof

Thus \eqref{mapback} gives the translation from the framework of \cite{CC1} and \cite{C} to the framework of the present paper.


\begin{thebibliography}{99} 
\bibitem{B} L.~Berger {\em An introduction to the theory of $p$-adic representations}, 
Geometric aspects of Dwork theory. Vol. I, II, 255--292, Walter de Gruyter GmbH  Co. KG, Berlin, 2004.

\bibitem{Berry} M.~Berry, {\em Stokes' phenomenon, smoothing a Victorian discontinuity}. IHES publications No 68 (1989) 211--221.

\bibitem{BC} O.~Brinon, B.~Conrad, {\em CMI Summer School Notes on $p$-adic Hodge Theory}, preliminary version (2005).
\bibitem{CC1}  A.~Connes, C.~Consani,  {\em Characteristic one, entropy and the absolute point}, `` Noncommutative Geometry, Arithmetic, and Related Topics'', the Twenty-First Meeting of the Japan-U.S. Mathematics Institute, Baltimore 2009, JHUP (2012), pp 75--139.
\bibitem{CC2} A.~Connes, C.~Consani {\em The hyperring of ad\`ele classes},  Journal of Number Theory 131 (2011) 159--194.
\bibitem{C} A.~Connes, {\em The Witt construction in characteristic one  and Quantization}. Noncommutative geometry and global analysis, 83--113, Contemp. Math., 546, Amer. Math. Soc., Providence, RI, 2011.

\bibitem{CCas} A.~Connes, C.~Consani, {\em The Arithmetic Site}, to appear on Comptes Rendus Mathematique; arXiv:1405.4527 [math.NT] (2014).

 \bibitem{Dingle}   R. B. Dingle, {\em Asymptotic Expansions : their Derivation and Interpretation}, New York and London, Academic Press, 1973.  

\bibitem{FF} L.~Fargues, J-M.~Fontaine {\em Courbes et fibr\'es vectoriels en th\'eorie de Hodge $p$-adique}, preprint (2011).
\bibitem{FF1} L.~Fargues, J-M.~Fontaine {\em Vector bundles and $p$-adic Galois representations}, AMS/IP Studies in Advanced Mathematics (51), 2011.
\bibitem{F1} J-M.~Fontaine, {\em Groupes $p$-divisibles sur les corps locaux}, Asterisque 47--48, Soci\'ete Mathematique de France, Paris, 1977.
\bibitem{F2} J-M.~Fontaine, {\em Sur certains types de repr\'esentations p-adiques du groupe de
Galois d'un corps local; construction d'un anneau de Barsotti-Tate} Ann. of Math.
115 (1982), 529–-577.
\bibitem{F3} J-M.~Fontaine, {\em Le corps des p\'eriodes $p$-adiques} Asterisque, (223), 59--111, 1994. With an appendix by Pierre Colmez, P\'eriodes $p$-adiques (Bures-sur-Yvette, 1988).
\bibitem{F4} J-M.~Fontaine, {\em Arithm\'etique des repr\'esentations galoisiennes $p$-adiques} Asterisque, (295), 1--115, 2004.
\bibitem{GPS} H.~Glaeske, A.~Prudnikov, K.~Skornik, {\em Operational Calculus and related topics}. Analytical Methods and Special Functions, 10. Chapman et Hall/CRC, Boca Raton, FL, 2006.

    \bibitem{G} A.~Grothendieck, {\em Groupes de Barsotti-Tate et cristaux de Dieudonn\'e}, S\'eminaire de Math\'ematiques Sup\'erieures, No. 45 (\'Et\'e, 1970). Les Presses de l'Universit\'e de Montr\'eal, Montreal, Que., 1974. 155 pp.




 \bibitem{Ko} M.~Kontsevich {\em The $1\frac 12$-logarithm} Friedrich Hirzebruchs Emeritierung, Bonn, November 1995.

     \bibitem{Kr}  M.~Krasner, {\em Approximation des corps valu\'es complets de caract\'eristique
$p\not=0$ par ceux de caract\'eristique $0$},  (French) 1957 Colloque
d'alg\`ebre sup\'erieure, tenu \`a Bruxelles du 19 au 22 d\'ecembre 1956 pp.
129--206 Centre Belge de Recherches Math\'ematiques \'Etablissements
Ceuterick, Louvain; Librairie Gauthier-Villars, Paris.

\bibitem{KOW} N.~Kurokawa, H.~Ochiai, M.~Wakayama, {\em Absolute derivations and zeta functions} in Kazuya Kato's fiftieth birthday. Doc. Math. 2003, Extra Vol., 565--584.

\bibitem{L} J. L.~Lions, {\em Supports de produits de composition}. I. (French) C. R. Acad. Sci. Paris 232, (1951). 1530--1532.

\bibitem{Lit} G.~Litvinov {\em Tropical Mathematics, Idempotent Analysis, Classical Mechanics and Geometry}. Spectral theory and geometric analysis, 159--186, Contemp. Math., 535, Amer. Math. Soc., Providence, RI, 2011.


\bibitem{R} W.~Rudin, {\em Real and Complex Analysis}, McGraw-Hill, NewYork 1987.

\bibitem{S}J. P.~Serre {\em Corps Locaux}, (French) Deuxi\`eme \'edition. Publications de l'Universit\'e de Nancago, No. VIII. Hermann, Paris, 1968. 245 pp.

 \bibitem{Stokes}   G. G. Stokes, {\em  On the critical values of the sums of periodic series}, Trans. Camb. Phil. Soc., 8 (1847), 533-610, reprinted in Mathematical and Physical Papers... (ref. [4]), vol. I, p. 236-313.

 \bibitem{T} E. C.~Titchmarsh, {\em
The zeros of certain integral functions}. Proceedings of the London Mathematical Society 25 (1926) 283–-302.
\bibitem{V} O.~Viro, {\em Hyperfields for tropical geometry I, hyperfields and dequantization}
arXiv:1006.3034v2.



\bibitem{W} J. P. Wintenberger, {\em Le corps des normes de certaines extensions infinies de corps locaux; applications}.
Ann. Sci. École Norm. Sup. (4) 16 (1983), no. 1, 59–89.

\end{thebibliography}
\end{document}